%% file: kzpn.tex
\documentclass[10pt,letterpaper,twoside]{article}

\title{On the $\K$-theory of $\bZ/p^n$}
\author{Benjamin Antieau, Achim Krause, Thomas Nikolaus}

\usepackage[letterpaper,left=1in,right=1in,top=1in,bottom=1in]{geometry}

\usepackage[pdfstartview=FitH,
            pdfauthor={},
            pdftitle={2021 notes},
            colorlinks,
            linkcolor=reference,
            citecolor=citation,
            urlcolor=e-mail,
            backref]{hyperref}

\usepackage[utf8]{inputenc}
\usepackage{multirow}
\usepackage{amsmath}
\usepackage{amscd}
\usepackage{amsbsy}
\usepackage{amssymb}
\usepackage{verbatim}
\usepackage{eufrak}
\usepackage{eucal}
\usepackage{microtype}
\usepackage{hyperref}
\usepackage{mathrsfs}
\usepackage{amsthm}
\usepackage{stmaryrd}
\usepackage[all,cmtip]{xy}
\input xy
\xyoption{all}
\usepackage{tikz}
\usetikzlibrary{matrix,arrows}
\usepackage{tikz-cd}
\usepackage{dsfont}
\usepackage[toc]{appendix}
\usepackage{slashed}
\usepackage[algoruled,vlined,english,linesnumbered]{algorithm2e}

\usepackage{enumitem}
\setenumerate{label=(\arabic*)}

\usepackage{fancyhdr}
\newcommand{\stackspace}{2.5}
\newcommand{\stack}[2][1cm]{\;\tikz[baseline, yshift=.65ex]%
    {\foreach \k [evaluate=\k as \r using (.5*#2+.5-\k)*\stackspace] in {1,...,#2}{%
    \ifodd\k{\draw[->](0,\r pt)--(#1,\r pt);}%
    \else{\draw[<-](0,\r pt)--(#1,\r pt);}\fi
    }}\;}

\usepackage[bbgreekl]{mathbbol}
\usepackage{amsfonts}
\DeclareSymbolFontAlphabet{\mathbb}{AMSb} 
\DeclareSymbolFontAlphabet{\mathbbl}{bbold}
\newcommand{\Prism}{{ \mathbbl{\Delta}}}
\newcommand{\prism}{{ \mathbbl{\Delta}}}

\newcommand{\Prismbar}{\overline{\mathbbl{\Delta}}}
\newcommand{\Prismpackage}{\underline{\,\Prism\,}}

\usepackage{yhmath}

\newcommand{\DHod}{\slashed{D}}

\usepackage{color}
\definecolor{todo}{rgb}{1,0,0}
\definecolor{conditional}{rgb}{0,1,0}
\definecolor{e-mail}{rgb}{0,.40,.80}
\definecolor{reference}{rgb}{.20,.60,.22}
\definecolor{mrnumber}{rgb}{.80,.40,0}
\definecolor{citation}{rgb}{0,.40,.80}

\let\oldmarginpar\marginpar
\renewcommand\marginpar[1]{\-\oldmarginpar[\raggedleft\footnotesize #1]%
{\raggedright\footnotesize #1}}

\newcommand{\Oscr}{\mathcal{O}}

\renewcommand{\d}{\mathrm{d}}
\newcommand{\D}{\mathrm{D}}
\newcommand{\E}{\mathrm{E}}
\newcommand{\F}{\mathrm{F}}
\newcommand{\G}{\mathrm{G}}
\renewcommand{\H}{\mathrm{H}}

\newcommand{\K}{\mathrm{K}}
\renewcommand{\L}{\mathrm{L}}

\newcommand{\N}{\mathrm{N}}

\newcommand{\bC}{\mathbf{C}}

\newcommand{\bE}{\mathbf{E}}
\newcommand{\bF}{\mathbf{F}}

\newcommand{\bQ}{\mathbf{Q}}

\newcommand{\bS}{\mathbf{S}}

\newcommand{\bZ}{\mathbf{Z}}

\newcommand{\fib}{\mathrm{fib}}

\newcommand{\gr}{\mathrm{gr}}

\newcommand{\deltatilde}{\widetilde{\delta}}
\newcommand{\varphitilde}{\widetilde{\varphi}}
\newcommand{\cB}{\mathrm{coBar}}

\newcommand{\conj}{\mathrm{conj}}
\newcommand{\syn}{\mathrm{syn}}
\newcommand{\dz}{{\mathrm{d}z}}
\newcommand{\can}{\mathrm{can}}

\newcommand{\id}{\mathrm{id}}

\newcommand{\im}{\mathrm{im}}
\renewcommand{\geq}{\geqslant}
\renewcommand{\leq}{\leqslant}

\newcommand{\iso}{\cong}
\newcommand{\we}{\simeq}

\newcommand{\THH}{\mathrm{THH}}

\newcommand{\TP}{\mathrm{TP}}
\newcommand{\TC}{\mathrm{TC}}

\newcommand{\mot}{\mathrm{mot}}

\newcommand{\dR}{\mathrm{dR}}

\newcommand{\Prismhat}{\widehat{\Prism}}

\newcommand{\ku}{\mathrm{ku}}
\newcommand{\KU}{\mathrm{KU}}

\newcommand{\GL}{\mathbf{GL}}

\DeclareMathOperator*{\Tot}{Tot}

\DeclareMathOperator{\Spec}{Spec}

\DeclareMathOperator{\red}{red}

\newcommand{\cO}{\Oscr}

\newcommand{\grey}[1]{{\color{gray}#1}}

\DeclareMathOperator{\Or}{Or}

\newcommand{\xto}[1]{\xrightarrow{#1}}

\theoremstyle{plain}
\newtheorem{theorem}{Theorem}[section]
\newtheorem*{theorem*}{Theorem}
\newtheorem{lemma}[theorem]{Lemma}

\newtheorem{proposition}[theorem]{Proposition}

\newtheorem{corollary}[theorem]{Corollary}
\newtheorem*{corollary*}{Corollary}

\theoremstyle{plain}

\theoremstyle{definition}

\newtheoremstyle{named}{}{}{\itshape}{}{\bfseries}{.}{.5em}{#1 \thmnote{#3}}
\theoremstyle{named}

\theoremstyle{definition}
\newtheorem{definition}[theorem]{Definition}
\newtheorem{warning}[theorem]{Warning}

\newtheorem{notation}[theorem]{Notation}

\newtheorem{example}[theorem]{Example}
\newtheorem*{example*}{Example}

\newtheorem*{question*}{Question}
\newtheorem{construction}[theorem]{Construction}

\newtheorem{remark}[theorem]{Remark}

\begin{document}
\maketitle

\begin{abstract}
    \noindent
    We give an explicit algebraic description, based on prismatic cohomology, of the algebraic $\K$-groups of rings
    of the form $\Oscr_K/\mathfrak{I}$ where $K$ is a
    $p$-adic field and $\mathfrak{I}$ is a non-trivial ideal in the ring of integers
    $\Oscr_K$; this class includes the rings $\bZ/p^n$ where $p$ is a prime.

    The algebraic description allows us to describe a practical algorithm
    to compute individual $\K$-groups as well as to obtain several theoretical results: the vanishing
    of the even $\K$-groups in high degrees, the determination of the orders of the odd
    $\K$-groups in high degrees, and the degree of nilpotence of $v_1$ acting on the mod
    $p$ syntomic cohomology of $\bZ/p^n$.
\end{abstract}

\tableofcontents

\section{Introduction}

We fix the following notation for the entire paper.
Let $\Oscr_K$ be a complete discrete valuation ring with finite residue field
$k=\bF_q$ and quotient field $K$ of characteristic zero, which is necessarily
finite over $\bQ_p$ of some degree $d$. Let $q=p^f$, where $f$ is the
residual degree. One has $d=fe$ for some positive integer $e$ called the
ramification index. Fix a choice of a uniformizer $\varpi\in\Oscr_K$; $e$ is
the unique integer such that the ideals $(p)$ and $(\varpi^e)$ are equal in $\Oscr_K$.

The fields $K$ that arise this way are the $p$-adic number fields;
rings of integers such as $\Oscr_K$ are $p$-adic number rings. The quotients
$\Oscr_K/\varpi^n$ are called finite chain rings~\cite{clark-liang}.

\subsection{Results}

The problem of computing the $\K$-groups of the quotient rings
$\Oscr_K/\varpi^n$ was raised by Swan in 1972;
see~\cite[Prob.~20]{gersten-problems}.
When $n=1$, we have $\Oscr_K/\varpi=\bF_q$;
the $\K$-groups of finite fields were computed by Quillen
in~\cite{Qui} (see also~\cite{jardine-quillen-revisited}):
\begin{equation}\label{eq:quillen}
    \K_r(\bF_q)\iso\begin{cases}
        \bZ&\text{if $r=0$,}\\
        \bZ/(q^i-1)&\text{if $r=2i-1$, and}\\
        0&\text{otherwise.}
    \end{cases}
\end{equation}
Since $p$ does not divide $q^i-1$,
letting $\K(-;\bZ[\tfrac{1}{p}])$ denote the
prime-to-$p$ localization of algebraic $\K$-theory, we have
$\K_r(\bF_q)\iso\K_r(\bF_q;\bZ[\tfrac{1}{p}])$ for $r\geq 1$,
while $\K(\bF_q;\bZ_p)\we\bZ_p$, where $\K(-;\bZ_p)$ denotes the $p$-adic
completion.
When $n\leq e$, $\Oscr_K/\varpi^n\cong\bF_q[z]/z^n$.
The $\K$-groups of truncated polynomial rings over finite fields were completely
described by Hesselholt and Madsen in~\cite{hesselholt-madsen-truncated} (see
also~\cite{speirs,sulyma}).
These are the only cases where a complete computation is known.

The $\K$-groups $\K_r(\Oscr_K/\varpi^n)$ of $\Oscr_K/\varpi^n$ are finitely generated abelian
groups, torsion for $r\geq 1$, and their
prime-to-$p$ information is determined by Quillen's computation of
the $\K$-groups of finite fields:
$\K(\Oscr_K/\varpi^n;\bZ[\tfrac{1}{p}])\simeq\K(\bF_q;\bZ[\tfrac{1}{p}])$, as can be
proved by using group homology techniques.
It remains to compute the $p$-complete $\K$-groups
$\K_r(\Oscr_K/\varpi^n;\bZ_p)$. For $r\geq 0$,
$\K_r(\Oscr_K/\varpi^n;\bZ_p)\iso\TC_r(\Oscr_K/\varpi^n;\bZ_p)$ by the
Dundas--Goodwillie--McCarthy theorem (see
Corollary~\ref{cor:connective_analysis}).

Bhatt, Morrow, and Scholze introduced in~\cite{bms2} $p$-adic syntomic
complexes $\bZ_p(i)(R)$ for quasisyntomic rings $R$; these are objects in the
$p$-complete derived category $\D(\bZ_p)_p^\wedge$. They gave a complete decreasing filtration
$\F^{\geq\star}_\mot\TC(R;\bZ_p)$ with graded pieces $\gr^i_\mot\TC(R;\bZ_p)\we\bZ_p(i)(R)[2i]$. The associated spectral
sequence then converges from the cohomology of these
complexes to the homotopy groups of $\TC(R;\bZ_p)$. We will show that in the case of $\Oscr_K/\varpi^n$, the spectral sequence collapses for degree
reasons and that there are
no possible extensions (see Corollary~\ref{cor:k_groups}). As
$\H^0(\bZ_p(i)(\Oscr_K/\varpi^n))=0$ for $i\geq 1$ (Corollary~\ref{cor:zigzag}), we conclude that
$$\K_{2i-1}(\Oscr_K/\varpi^n;\bZ_p)\cong\H^1(\bZ_p(i)(\Oscr_K/\varpi^n))$$ for
$i\geq 1$ and
$$\K_{2i-2}(\Oscr_K/\varpi^n;\bZ_p)\cong\H^2(\bZ_p(i)(\Oscr_K/\varpi^n))$$ for
$i\geq 2$, $\K_0(\Oscr_K/\varpi^n;\bZ_p)\iso\H^0(\bZ_p(0)(\Oscr_K/\varpi^n))\iso\bZ_p$.

This paper has three main results. The first gives a practical algorithm for computing the
$p$-adic syntomic complexes
$\bZ_p(i)(\Oscr_K/\varpi^n)$.

\begin{theorem}
\label{thm:intro-cochaincomplex}
    Given $K$, $n$, $i\geq 1$, and a fixed uniformizer $\varpi\in\Oscr_K$,
    there is an explicit, algorithmically computable three-term cochain complex
    \begin{equation}\label{eq:cochain}
        \cdots \to 0\to \bZ_p^{f(in-1)} \xto{\syn^0} \bZ_p^{2f(in-1)}
        \xto{\syn^1} \bZ_p^{f(in-1)} \to 0 \to  \cdots
    \end{equation}
    concentrated in cohomological degrees $0,1,2$ of finite rank free $\bZ_p$-modules which is
    is quasi-isomorphic to $\bZ_p(i)(\Oscr_K/\varpi^n)$.
\end{theorem}

Our approach relies on a detailed examination of prismatic envelopes and
is heavily influenced by the work of Liu--Wang~\cite{liu-wang} who introduced a topological
version of the local-to-global descent arguments given below to recover the calculation of $\K(\Oscr_K;\bF_p)$
of Hesselholt--Madsen~\cite{hesselholt-madsen}.

\begin{remark}
    The differentials $\syn^0$ and $\syn^1$ are $p$-adic matrices;
    the algorithm is to construct these matrices up to sufficient $p$-adic precision.
    As mentioned above, the map $\syn^0$ is injective by Corollary~\ref{cor:zigzag} so that this cochain complex has cohomology only in
    degrees $1$ and $2$.
\end{remark}

\begin{remark}
    The computability of these groups is not new and follows also from a combination of homological
    stability (see~\cite{suslin-stability,vanderkallen-stability}) and the fact that homotopy groups
    of finite CW complexes with finite homotopy groups are computable, for example via minimal Kan
    complexes or by Serre's fibration method. See~\cite{brown-computability} for the classical approach
    and~\cite{rubio-sergeraert-constructive} for a more modern survey which also presents {\ttfamily
    Kenzo}, a
    computer package for making such computations. In~\cite{kenzo-sage}, a {\ttfamily SAGE} interface
    for {\ttfamily Kenzo} is described. In practice, these approaches are feasible only in very low degrees. The novelty in our approach is the
    reduction to linear algebra using prismatic cohomology which also allows us to compute high-degree
    $\K$-groups independently of low-degree groups.
\end{remark}

In particular, writing $\bZ_p(i)(\Oscr_K/\varpi^n)^\bullet$ for our explicit
cochain complex model for $\bZ_p(i)(\Oscr_K/\varpi^n)$, we have for $i\geq 1$
that
\begin{align*}
    \H^0(\bZ_p(i)(\Oscr_K/\varpi^n)^\bullet)&=0,\\
    \H^1(\bZ_p(i)(\Oscr_K/\varpi^n)^\bullet)&\iso\K_{2i-1}(\Oscr_K/\varpi^n;\bZ_p),\\
    \H^2(\bZ_p(i)(\Oscr_K/\varpi^n)^\bullet)&\iso\K_{2i-2}(\Oscr_K/\varpi^n;\bZ_p)
    \quad\text{(if additionally $i\geq 2$)}.
\end{align*}
These groups can thus be computed using elementary divisors from the
matrices $\syn^0$ and $\syn^1$.

We have implemented in~\cite{akn_syn_coh} the computation of the matrices $\syn^0$ and $\syn^1$ in
{\ttfamily SAGE 10.2}~\cite{sagemath} in the special case where $f=1$, so that the
residue field is $\bF_p$.
For example, Figure~\ref{fig:introtable} displays a table of $\K$-groups for
some small rings we obtained through computer algebra computations.
Similar computations can be carried out for many more examples and degrees and
we present extended calculations in Appendix~\ref{appendix:tables}.

\begin{figure*}[h!]
    \centering
    \[
    \begin{array}{c|c|c|c|c}
    & \bF_2 & \bZ/4 & \bZ/8 & \bF_2[z]/z^3\\
    \hline
    \K_1 & \grey{0} & \grey{\bZ/2} & \grey{\bZ/4} & \grey{\bZ/4}\\
    \K_2 & \grey{0} & \grey{\bZ/2} & \grey{\bZ/2} & \grey{0}\\
        \K_3 & \grey{\bZ/3} & \grey{\bZ/3} \oplus \bZ/8 & \grey{\bZ/3}\oplus \bZ/4\oplus\bZ/8 & \grey{\bZ/3\oplus \bZ/2\oplus\bZ/8}\\
    \K_4 & \grey{0} & 0 & \bZ/2 & \grey{0}\\
        \K_5 & \grey{\bZ/7} & \grey{\bZ/7} \oplus \bZ/8 & \grey{\bZ/7}\oplus \bZ/2\oplus\bZ/2^6 & \grey{\bZ/7\oplus (\bZ/2)^2\oplus\bZ/16}\\
    \K_6 & \grey{0} & 0 & 0 & \grey{0}\\
        \K_7 & \grey{\bZ/15} & \grey{\bZ/15} \oplus \bZ/2\oplus \bZ/8 & \grey{\bZ/15}\oplus \bZ/16\oplus \bZ/16 &\grey{\bZ/15 \oplus (\bZ/2)^2 \oplus \bZ/4 \oplus \bZ/16}\\
    \K_8 & \grey{0} & 0 & 0 & \grey{0}\\
        \K_9 & \grey{\bZ/31} & \grey{\bZ/31} \oplus (\bZ/2)^2\oplus\bZ/8 & \grey{\bZ/31} \oplus \bZ/2\oplus \bZ/4\oplus \bZ/2^7 & \grey{\bZ/31 \oplus (\bZ/2)^2 \oplus (\bZ/4)^2 \oplus\bZ/16}\\
        \K_{10} & \grey{0} & 0 & 0 & \grey{0}\\
    \end{array}
    \]
    \caption{Low-degree $\K$-groups of some small rings. The gray terms
    were known prior to this paper.}
    \label{fig:introtable}
\end{figure*}

In addition to computational results in small degrees, our approach allows
us to deduce some abstract statements about $\K_*(\cO_K/\varpi^n;\bZ_p)$. A striking
pattern in Figure~\ref{fig:introtable} is that the $\K$-groups seem to vanish in even
degrees from some point.

Our second theorem is that this is always true, which we prove by showing that the
differential $\syn^1$ from~\eqref{eq:cochain} is surjective for $i$ sufficiently large.

\begin{theorem}[The even vanishing theorem]\label{thm:even_intro}
    If 
    \[
      i-1\geq
    \frac{p}{p-1}\left(\frac{p}{p-1}(p^{\lceil\tfrac{n}{e}\rceil}-1)-p^{\lceil\tfrac{n}{e}\rceil}(\lceil\tfrac{n}{e}\rceil-\tfrac{n}{e})\right)
  \]
  (and $i\geq 2$),
     then
    $\K_{2i-2}(\Oscr_K/\varpi^n;\bZ_p)=0$ for
    all $p$-adic fields $K$ of ramification index $e$.\footnote{Note that
    $\lceil\tfrac{n}{e}\rceil$ is the exponent of the abelian group $\Oscr_K/\varpi^n$.}
\end{theorem}

Another result we extract from our methods is the following proposition, in which we use $|A|$ to
denote the order of a finite abelian group $A$. It was
stated and proved by Angeltveit~\cite{angeltveit} in the unramified case, i.e., for
rings of the form $W(\bF_q)/p^n$.

\begin{proposition}[Angeltveit's quotient]\label{prop:intro-angeltveit-quotient}
    For any $p$-adic number field $K$ and any $n\geq 1$ and $i\geq 2$,
    $$\frac{\left|\K_{2i-1}(\Oscr_K/\varpi^n;\bZ_p)\right|}{\left|\K_{2i-2}(\Oscr_K/\varpi^n;\bZ_p)\right|}=q^{i(n-1)}.$$
\end{proposition}

Together, the even vanishing theorem and Angeltveit's quotient imply the
following corollary.

\begin{corollary}\label{cor:ordersintro}
    For $i\gg 0$, $\K_{2i-1}(\cO_K/\varpi^n)$ has order $(q^i - 1)q^{i(n-1)}$.
\end{corollary}

A nice way to summarize these results is that, in large degrees, the
\emph{orders} of $\K_*(\cO_K/\varpi^n)$ and $\K_*(\bF_q[z]/z^n)$ agree, even
though the groups might differ.

For a fixed ring $\Oscr_K/\varpi^n$,
by applying Theorem~\ref{thm:intro-cochaincomplex} in low degrees and Corollary~\ref{cor:ordersintro} in high degrees,
one reduces the determination of the orders of all $\K$-groups of $\Oscr_K/\varpi^n$ to a finite
number of applications of the algorithm.

\begin{example}
  For $\bZ/4$, Theorem~\ref{thm:even_intro} and the values from the table in Appendix A.1 show that the only positive degree non-zero even $\K$-group is $\K_2(\bZ/4)\iso\bZ/2$.
    It follows that the orders of the $\K$-groups
    of $\bZ/4$ are
    $$\left|\K_r(\bZ/4)\right|\iso\begin{cases}
        2&\text{if $r=1,2$,}\\
        24&\text{if $r=3$,}\\
        (2^i-1)2^{i}&\text{if $r=2i-1$ for $i\geq 3$,}\\
        0&\text{otherwise.} 
    \end{cases}$$
    for $r\geq 1$.
\end{example}

Another application of our results is to reprove the main step in the proof of the theorem of
Bhatt--Clausen--Mathew~\cite{bhatt-clausen-mathew} (reproved and generalized by 
Land--Mathew--Meier--Tamme~\cite{land-mathew-meier-tamme}), which says that
$\L_{T(1)}\K(R)\we\L_{T(1)}\K(R[\tfrac{1}{p}])$, where $\L_{T(1)}$ denotes the height $1$ telescopic
localization at the prime $p$. The main step is to prove that $$\L_{T(1)}\K(\bZ/p^n)\we 0$$ for all
$n\geq 1$. For $n=1$, this follows from Quillen's computation~\eqref{eq:quillen}.
For $n\geq 2$, other arguments are needed. In~\cite{bhatt-clausen-mathew}, the argument goes by reduction to the case
of $\Oscr_{\bC_p}/p^n$ and an argument using prismatic cohomology and prismatic envelopes;
in~\cite{land-mathew-meier-tamme}, the argument uses excision results.

Our third theorem is a finer vanishing result,
giving a bound on the degree of $v_1$-nilpotence of $\K(\bZ/p^n;\bF_p)$. See
Section~\ref{sec:v1nilpotence}.

\begin{theorem}\label{thm:v1nilpotenceintro}
    The class $v_1\in\bF_p(p-1)(\bZ/p^n)$ has nilpotence degree exactly $[n]_p=\tfrac{p^n-1}{p-1}$
    in the motivic associated graded $$\gr^\star_\mot\TC(\bZ/p^n;\bF_p)\we\bigoplus_{i\geq
    0}\bF_p(i)(\bZ/p^n).$$
\end{theorem}

\begin{remark}
    It follows that for primes $p\geq 5$, the self map $v_1$ of $\K(\bZ/p^n;\bF_p)$ has nilpotence degree
    between $[n]_p$ and $2[n]_p$ using that in these cases $\K(\bZ/p^n;\bF_p)$ is a homotopy
    associative ring spectrum; see Corollary~\ref{cor:nil_on_k}.
    However, the work of Hahn--Levy--Senger (see Remark~\ref{rem:hls}) can be used to show that it has nilpotency
    degree either $[n]_p$ or $[n]_p+1$.
\end{remark}

\begin{remark}
    There is a version of Theorem~\ref{thm:v1nilpotenceintro} for ramified extensions as
    well. In that case, one has that $v_1$ has nilpotence degree at most
    $$\frac{p^{\lceil\tfrac{n}{e}\rceil}-1}{p-1}$$ 
    when acting on the motivic associated graded.
    This bound is sharp when $e$ divides $n$. See Theorem~\ref{thm:nilpotent}.
\end{remark}

\subsection{History}

Kuku proved that if $R$ is a finite associative ring, then $\K_r(R)$ is a
finite abelian group for $r>0$; see~\cite[Prop.~IV.1.16]{weibel-kbook}.

Work of Dennis and Stein~\cite{dennis-stein} treats $\K_2(\Oscr/\varpi^n)$. This is the only previous paper we
know which considers the higher algebraic $\K$-theory of general rings of this form.

As mentioned above, the $\K$-groups of the characteristic $p$ quotients
$\bF_q[z]/z^n$ (so, when $n\leq e$) were described by
Hesselholt--Madsen~\cite{hesselholt-madsen-truncated}. Their calculations have been revisited by
Speirs~\cite{speirs}, Mathew~\cite{mathew-survey}, and Sulyma~\cite{sulyma} using modern
advances in cyclotomic spectra and prismatic cohomology. (They have also been generalized to
truncated polynomial rings over perfectoid rings by Riggenbach~\cite{riggenbach-truncated}.)

The $\K$-groups in the unramified case have been considered by many researchers.
Evens and
Friedlander~\cite{evens-friedlander} computed $\K_3(\bZ/p^2)$ and
$\K_4(\bZ/p^2)$ for $p\geq 5$ and Aisbett, Lluis-Puebla, and
Snaith~\cite{aisbett-lluis-puebla-snaith} worked on
$\K_3(\bZ/p^n)$ and $\K_3(\bF_q[z]/z^2)$. Geisser computed 
$\K_3(W(\bF_q)/p^2)$ for $p\geq 3$, correcting a mistake in Aisbett et al. at
$p=3$.

Brun~\cite{brun} computed $\K_r(\bZ/p^n)$ for $r\leq p-3$ and
Angeltveit~\cite{angeltveit} computed $\K_r(W(\bF_q)/p^n)$ for $r\leq 2p-2$, but his results in
conjunction with Proposition~\ref{prop:intro-angeltveit-quotient} also give the order of $\K_{2p-1}(W(\bF_q)/p^n)$.
One consequence of Angeltveit's calculation is the discovery of an error in the
calculation of $\K_3(\bZ/4)$ in Aisbett et al.; Angeltveit's calculation gives
the correct order and the correct group structure is given by our calculations
in Figure~\ref{fig:introtable}. We thank Markus Szymik for bringing this 
to our attention.

In the limit, the $p$-adic $\K$-theory of $\Oscr_K$ itself was an important
test case for the Quillen--Lichtenbaum conjecture and was computed by
Hesselholt and Madsen~\cite{hesselholt-madsen}. This calculation was revisited
by Liu and Wang~\cite{liu-wang} with $\bF_p$-coefficients.

Our approach is very similar to that of Liu and Wang~\cite{liu-wang} and also to the 
method used in~\cite{krause-nikolaus} for computing $\THH$ of $\Oscr_K$ and $\Oscr_K/\varpi^n$;
indeed, what we do can be viewed as a purely prismatic analogue of the topological story
in~\cite{krause-nikolaus,liu-wang}.

\begin{remark}\label{rem:hls}
    Hahn, Levy, and Senger have recently announced a complete calculation of
    $\K(\bZ/p^n)/(p,v_1)$ for odd primes; the resulting groups are
    independent of $n$, for $n\geq 2$.
\end{remark}

\paragraph{Organization.} Section~\ref{sec:overview} gives an overview of our
approach to the calculation of $\K_*(\Oscr_K/\varpi^n;\bZ_p)$. Section~\ref{sec:envelopes} contains a detailed analysis
of the algebra of prismatic envelopes and Section~\ref{sec:bkdescent} gives our
approach to explicit relative-to-absolute descent. Section~\ref{sec:evenvanishing} gives the proof of the
even vanishing theorem and Section~\ref{sec:v1nilpotence} gives the proof of the $v_1$-nilpotence
theorem. Finally, Section~\ref{sec:algorithm} details the
algorithm for computing the $\K$-groups.

\paragraph{Acknowledgments.}
We are very grateful to Bhargav Bhatt, Dustin Clausen, Thomas Geisser, Jeremy Hahn, Lars
Hesselholt, Ishan Levy, Akhil Mathew, Noah
Riggenbach, Peter Scholze, Andrew Senger, Markus Szymik, and
Chuck Weibel for comments on~\cite{akn-announcement} and the results in this paper.
Additionally, Yuanning Zhang and Jeremy Hahn each pointed out a typo in one of the tables
in~\cite{akn-announcement}.

BA was supported by NSF grants DMS-2120005,
DMS-2102010, and DMS-2152235, by Simons Fellowship 666565, and by the Simons
Collaboration on Perfection; he would like to thank
Universit\"at M\"unster
for its hospitality during a visit in 2020. 
AK and TN were funded by the Deutsche Forschungsgemeinschaft
(DFG, German Research Foundation) – Project-ID 427320536 – SFB 1442, as well as
under Germany’s Excellence Strategy EXC 2044 390685587, Mathematics Münster:
Dynamics–Geometry–Structure. They would also like to thank the Mittag--Leffler
Institute for its hospitality while working on this project.

This paper also benefited from the opportunity for BA and AK to give a masterclass on the
topic in Copenhagen in early 2023. We thank the organizers (Shachar Carmeli, Lars Hesselholt, Ryomei Iwasa, and
Mikala Jansen) and the
Copenhagen Centre for Geometry and Topology for the opportunity.

Finally, BA and AK thank the Institute for Advanced Study for its hospitality in early
2024 when we finished this paper. BA was supported there by NSF grant DMS-1926686. AK was supported by a grant from the Institute of Advanced Study School of Mathematics.

The tables in Appendix~\ref{appendix:tables} were created under allocation {\ttfamily p31679} using {\ttfamily QUEST}, a
high-performance cluster at Northwestern University.

\section{Strategy}\label{sec:overview}

In this section, we give the theoretical results which reduce the computation of the $\K$-groups of
$\Oscr_K/\varpi^n$ to a finite computation involving prismatic cohomology. We will make heavy use of
prismatic cohomology relative to $\delta$-rings and its filtered variant, for which
we refer to our companion paper~\cite{akn-delta}.
In general, let $\Prismpackage_{R/A}$ denote the prismatic package of a $\delta$-pair,
possibly filtered, from which we can functorially extract relative syntomic complexes
$\bZ_p(i)(R/A)$, which are equipped with natural filtrations for filtered $\delta$-pairs.
We are most interested in the case when $A=\bZ_p$ or $W(k)$ with the trivial filtration, in which case we will
write $\Prismpackage_R$ and $\bZ_p(i)(R)$ for the absolute theories. The prismatic cohomology
package contains the data of prismatic cohomology, the Hodge--Tate tower, Breuil--Kisin twists, and
Nygaard-filtered Nygaard-completed Frobenius-twisted prismatic cohomology, together with canonical
and Frobenius maps.

\subsection{General reductions}\label{sec:reductions}

To begin with, the groups we are interested in are finitely generated abelian groups, making
an algorithmic approach feasible. Weibel~\cite{weibel-kbook} attributes the group homology proof of
this fact to Aderemi Kuku.

\begin{proposition}[Kuku]
    If $A$ is a finite associative ring, then
    for each $r\geq 1$ the $\K$-group $\K_r(A)$ is finitely
    generated and torsion.
\end{proposition}

\begin{corollary}
    For each $r\geq 1$, the $\K$-group $\K_r(\Oscr_K/\varpi^n)$ is finitely
    generated and torsion.
\end{corollary}

As above, let $\bZ[\tfrac{1}{p}]$ denote the prime-to-$p$ localization of $\bZ$.
We let $\K(-;\bZ[\tfrac{1}{p}])$ be the prime-to-$p$ localization of $\K(-)$, and
similarly for other functors.

\begin{corollary}\label{cor:prime_to_p}
    The natural map
    $\K(\Oscr_K/\varpi^n;\bZ[\tfrac{1}{p}])\rightarrow\K(\bF_q;\bZ[\tfrac{1}{p}])$ is
    an equivalence.
\end{corollary}

\begin{proof}
    This follows from the fact that the kernel of the surjective homomorphism
    $\GL(\Oscr_K/\varpi^n)\rightarrow\GL(\bF_q)$ is a $p$-group.
\end{proof}

The result we need for $p$-adic $\K$-theory is the Dundas--Goodwillie--McCarthy
theorem; see also the rigidity result of Clausen--Mathew--Morrow~\cite{cmm}.

\begin{theorem}[Dundas--Goodwillie--McCarthy~\cite{dundas-goodwillie-mccarthy}]\label{thm:dgm}
    If $A$ is an associative ring and $I$ is a $2$-sided nilpotent ideal in $A$, then
    the natural commutative square
    $$\xymatrix{
        \K(A)\ar[r]\ar[d]&\K(A/I)\ar[d]\\
        \TC(A)\ar[r]&\TC(A/I)
    }$$
    is a pullback square.
\end{theorem}

In the theorem, note that if $A$ is $p$-nilpotent, then $\TC(A)$ and $\TC(A/I)$ are
$p$-complete.

\begin{corollary}\label{cor:connective_analysis}
    The natural map
    $\K_r(\Oscr_K/\varpi^n;\bZ_p)\rightarrow\TC_r(\Oscr_K/\varpi^n)$
    is an isomorphism for $r\geq 0$.
\end{corollary}

\begin{proof}
    This follows from Theorem~\ref{thm:dgm} and the fact that $\K(\bF_q;\bZ_p)\rightarrow\TC(\bF_q)$ is
    $(-2)$-truncated by Quillen's computation of the $\K$-theory of finite fields and the
    calculation of Hesselholt--Madsen~\cite[Thm.~B]{hesselholt-madsen-finite} of $\TC$ of
    perfect fields of characteristic $p$; see also~\cite[Cor.~IV.4.10]{nikolaus-scholze}.
\end{proof}

In order to compute the $\TC$-groups of $\Oscr_K/\varpi^n$, and hence the higher
$p$-adic $\K$-groups, we use the theory of syntomic cohomology first defined
in~\cite{bms2} (see also~\cite{fontaine-messing,kato-vanishing} and~\cite{ammn} for a comparison). This gives, for a quasisyntomic ring $R$ a sequence of
$p$-complete complexes $\bZ_p(i)(R)$ in the $p$-complete derived category
$\D(\bZ_p)_p^\wedge$ and a complete multiplicative decreasing
filtration $\F^{\geq\star}_\mot\TC(R;\bZ_p)$ with
graded pieces $$\gr^i_\mot\TC(R;\bZ_p)\we\bZ_p(i)(R)[2i].$$
The associated spectral sequence takes the form
$$\E_2^{i,j}=\H^{i-j}(\bZ_p(-j)(R))\Rightarrow\TC_{-i-j}(R;\bZ_p)$$
with differentials $d_r$ of bidegree $(r,1-r)$.
The goal in this paper is to compute the cohomology of the complexes $\bZ_p(i)(\Oscr_K/\varpi^n)$.
It turns out that in this case the motivic filtration on $\TC(R;\bZ_p)$ is
simple enough that there is no additional work needed to go from the cohomology
of the $\bZ_p(i)(\Oscr_K/\varpi^n)$ to the $p$-adic $\K$-groups: the associated
spectral sequence degenerates and there are no non-trivial extensions for degree
reasons. See~Corollary~\ref{cor:k_groups}.

\subsection{Prismatic and syntomic cohomology}\label{sec:prismatic}

By the definition in~\cite{bms2}, for $i\in\bZ$, the $i$th syntomic complex of a quasisyntomic
ring $R$ is $$\bZ_p(i)(R)=\fib\left(\N^{\geq
i}\Prismhat_R\{i\}\xrightarrow{\can-\varphi}\Prismhat_R\{i\}\right)$$
in $\D(\bZ_p)_p^\wedge$,
where $\Prismhat_R\{i\}$ is the $i$th Breuil--Kisin twisted Nygaard-complete absolute
prismatic cohomology of $R$ and $\N^{\geq i}\Prismhat_R\{i\}$ is the $i$th stage
of the Nygaard filtration.

Our goal in this section is to bound the complexity of $\bZ_p(i)(\Oscr_K/\varpi^n)$, which we do by
comparison with the characteristic $p$ case via a philosophy we will refer to as crystalline
degeneration. Indeed, the rings $\Oscr_K/\varpi^n$ are filtered by $\varpi$-adic powers. The
associated graded ring is $\bF_q[z]/z^n$, which we view as being $z$-adically filtered. These
filtrations induce
secondary filtrations on all of the invariants in the picture: prismatic
cohomology, $\THH$, $\TC$, etc. We will in all cases write
$\F^{\geq\star}$ for this secondary filtration.
On associated graded pieces,
$$\gr^\star_\F\bZ_p(i)(\Oscr_K/\varpi^n)\we\gr^\star_\F\bZ_p(i)(\bF_q[z]/z^n)$$
by the following proposition,
and similarly for the other theories (by~\cite[Cor.~10.32]{akn-delta}).
    
\begin{proposition}\label{prop:degeneration}
    Let $\F^{\geq\star} R$ and $\F^{\geq\star}S$ be filtered quasisyntomic
    rings (in the sense that the associated $p$-completed Rees algebras are
    quasisyntomic over the $p$-completed Rees algebra of the trivial filtration on
    $\bZ_p$). If the associated graded rings $\gr^\star_\F R$ and
    $\gr^\star_\F S$ are equivalent, then there is an induced equivalence
    $$\gr^\star_\F\Prismpackage_R\we\gr^\star_\F\Prismpackage_S$$ and hence there are induced
    equivalences
    $$\gr^\star_\F\bZ_p(i)(R)\we\gr^\star_\F\bZ_p(i)(S)$$
    for each $i$.
\end{proposition}

\begin{proof}
    See~\cite[Prop.~10.44, Cor.~10.45]{akn-delta}.
\end{proof}

\begin{remark}
    A similar use of filtrations is the basis of the work of
    Brun~\cite{brun} and Angeltveit~\cite{angeltveit}.
\end{remark}

Mathew~\cite{mathew-survey} (for $n=2$ and $p$ odd) and Sulyma~\cite{sulyma}
(in general) have computed
the syntomic cohomology of truncated polynomial rings $R=k[z]/z^n$ when $k$ is a
perfect field of characteristic $p$. In this case, prismatic cohomology agrees with crystalline cohomology. Since $R$ admits a flat lift to $W(k)$
together with a lift of Frobenius, one can compute $\Prism_R$ as a complex
\begin{equation}\label{eq:dp_connection}
    \Prism_R^\bullet\colon\left(D_{(z^n)}W(k)[z]\xrightarrow{\d}D_{(z^n)}W(k)[z]\dz\right)_p^\wedge,
\end{equation}
the $p$-completion of the divided power de Rham complex. The Nygaard filtration
$\N^{\geq\star}\Prism_R$ is given by the Day convolution of the $p$-adic
filtration on $\bZ_p$ and the pd-Hodge filtration on $\Prism_R^\bullet$.
(For characteristic $p$ rings such as $R$, the Breuil--Kisin twists are trivializable.)

\begin{lemma}
    These presentations also compute the graded absolute prismatic cohomology of
    $R=k[z]/z^n$ together with its Nygaard filtration; see
    also~\cite[Ex.~8.11]{mathew-survey}.
\end{lemma}

\begin{proof}
    The lemma can be proved along the lines of~\cite[Prop.~8.7]{bms2} by keeping
    track of the grading everywhere. 
\end{proof}

\begin{lemma}\label{lem:truncated}
    Let $R=k[z]/z^n$ where $k$ is a perfect field of characteristic $p$.
    \begin{enumerate}
        \item[{\em (a)}] For $j>0$, $\gr^j_\F\Prismhat_R$ is concentrated in cohomological
            degree $1$, while $\gr^0_\F\Prismhat_R\we\Prismhat_k\we W(k)$.
        \item[{\em (b)}]  For $j>0$, $\gr^j_\F\N^{\geq i}\Prismhat_R$ is
            concentrated in cohomological degree $1$, while $\gr^0_\F\N^{\geq
            i}\Prismhat_R\we (p^i)\subseteq W(k)$ for $i\geq 0$.
        \item[{\em (c)}] For $i\geq 1$, $\gr^j_\F\bZ_p(i)(R)$ is in $\D(\bZ_p)_{[-2,-1]}$ and
            vanishes for $j\geq in$ and for $j=0$.
    \end{enumerate}
\end{lemma}

\begin{proof}
    See the computation in~\cite[Sec.~3]{sulyma}, noting that in this case the Nygaard-completion
    and $\F$-adic completion of $\Prism_R$ agree.
\end{proof}

\begin{remark}
    The $\F$-adic filtration on $\Prismhat_R$ need not be complete for general
    filtered commutative rings $\F^\star R$, even if the filtration on
    $R$ is complete. Nevertheless, Proposition~\ref{prop:completeness} implies that in good cases
    the induced filtration on syntomic cohomology is complete.
\end{remark}

\begin{proposition}[Completeness]\label{prop:completeness}
    Suppose that $\F^{\geq\star}R$ is a complete filtered ring.
    If $\F^{\geq\star}R$ is the filtered quotient of the completion of a finitely generated filtered
    polynomial ring by a filtered regular sequence,
    then the induced $\F$-adic filtration on $\bZ_p(i)(R)$ is complete.
\end{proposition}

\begin{proof}
    It suffices to check the claim modulo $p$. However, modulo $p$, we can
    compute $\bZ_p(i)(R)/p=\bF_p(i)(R)$ as
    $$\bF_p(i)(R)\we\fib\left(\frac{\N^{\geq i}\Prismhat_R\{i\}/p}{\N^{\geq
    i+M}\Prismhat_R\{i\}/p}\xrightarrow{\can-\varphi}\frac{\Prismhat_R\{i\}/p}{\N^{\geq
    i+M}\Prismhat_R\{i\}/p}\right)$$ for $M$ sufficiently large by~\cite[Lem.~7.22]{bms2} or~\cite[Cor.~5.31]{ammn} or~\cite[Prop.~8.6]{akn-delta}.
    The terms in the fiber sequence can be expressed via
    finitely many associated graded pieces of the Nygaard filtration on
    absolute prismatic cohomology (modulo $p$), which in turn can be expressed in terms of
    finitely many filtered pieces of the conjugate filtration
    $\F_{\leq\star}^\conj\Prismbar_R\{i\}$ on absolute Hodge--Tate
    cohomology $\Prismbar_R\{i\}$. Finally, each filtered piece
    $\F_{\leq j}^\conj\Prismbar_R\{i\}$ can be expressed in terms of
    the finitely many filtered pieces of the conjugate filtration on the diffracted Hodge complex $\Omega_R^{\DHod}$ of
    Bhatt--Lurie~\cite{bhatt-lurie-apc}. But,
    $\gr_j^\conj\Omega_R^\DHod\we\widehat{\L\Omega}^j_{R/\bZ_p}[-j]$, the
    $(-j)$-fold shift of the $p$-complete
    derived absolute $j$-forms of $R$. It suffices now to use the fact that $\L\Omega^j_{R/\bZ_p}$
    is complete for all $j$ by Lemma~\ref{lem:formsarecomplete}.
\end{proof}

\begin{lemma}\label{lem:formsarecomplete}
    Suppose that $\F^{\geq\star}R$ is a complete filtered ring.
    If $\F^{\geq\star}R$ is the filtered quotient of the completion of a finitely generated filtered
    polynomial ring by a filtered regular sequence, then $\L\Omega^j_{R/\bZ_p}$ is complete for
    each $j\geq 0$.
\end{lemma}

\begin{proof}
    By hypothesis, $\L_{R/\bZ_p}$ can be written as the cofiber of maps
    $\F^{\geq\star} P_1\rightarrow\F^{\geq\star} P_0$ where $\F^{\geq \star} P_1$ and
    $\F^{\geq\star}
    P_0$ are filtered projective $\F^{\geq\star} R$-modules on finitely many
    generators. But, in this case, $\L\Omega^j_{R/\bZ_p}$ has a finite increasing
    filtration with graded pieces
    $\L\Lambda^aP_0\otimes\L\Lambda^b(P_1[1])\we\L\Lambda^aP_0\otimes\L\Gamma^b(P_1)[b]$ where
    $a+b=j$. Each of these is a (shift of a) filtered projective $\F^{\geq\star} R$ module on
    finitely many generators. It follows that $\L\Omega^j_{R/\bZ_p}$ is
    complete since $\F^{\geq \star} R$ is.
\end{proof}

\begin{corollary}\label{cor:associatedgraded}
    If $R=\Oscr_K/\varpi^n$ is given the $\varpi$-adic filtration, then
    \begin{enumerate}
        \item[{\em (i)}] the induced filtration $\F^{\geq\star}\bZ_p(i)(R)$ on $\bZ_p(i)(R)$ is complete for all $i$,
        \item[{\em (ii)}] $\gr^j_\F\bZ_p(i)(R)$ is in $\D(\bZ_p)_{[-2,-1]}$ for $i\geq 1$ and
            vanishes for $j\geq in$, and
        \item[{\em (iii)}] $\gr^0_\F\bZ_p(i)(R)\we 0$ for $i\geq 1$.
    \end{enumerate}
\end{corollary}

\begin{proof}
    This follows from Proposition~\ref{prop:completeness}, Proposition~\ref{prop:degeneration}, and Lemma~\ref{lem:truncated}.
\end{proof}

The main result of this section is the following, which says that the weight $i$ syntomic cohomology of
$\Oscr_K/\varpi^n$ is concentrated in bounded $\F$-weights. Later, this will permit us to work
with the absolute prismatic cohomology of $R$ and its Nygaard filtration in bounded $\F$-weights as well.
For a filtered spectrum $\F^{\geq\star} M$ and integers $a\leq b$, we write $\F^{[a,b]}$ for the cofiber of
$\F^{\geq b+1}M\rightarrow\F^{\geq a}M$.

\begin{corollary}\label{cor:zigzag}
    For $i\geq 1$,
    there is a zig-zag of equivalences
    $$\bZ_p(i)(\Oscr_K/\varpi^n)\rightarrow\F^{[0,in-1]}\bZ_p(i)(\Oscr_K/\varpi^n)\leftarrow\F^{[1,in-1]}\bZ_p(i)(\Oscr_K/\varpi^n).$$
    Moreover, $\bZ_p(i)(\Oscr_K/\varpi^n)\in\D(\bZ_p)_{[-2,-1]}$.
\end{corollary}

\begin{proof}
    This follows from Corollary~\ref{cor:associatedgraded}.
\end{proof}

\begin{remark}
  \label{rem:generalfiltrationbounds}
  One can extend the filtration bounds from Corollary \ref{cor:zigzag} to a fairly
    general situation. Let $A$ be a filtered $\delta$-ring $A$ and $R$ a filtered
    $A$-algebra, and assume that the filtration of $R$ is concentrated in $[0,c]$ and
    that $L_{R/A}$ as an $R$-module has ``filtered amplitude in $[a,b]$'', i.e. can be
    built from copies of $R$ with filtration shifted by values in $[a,b]$. Then $\L\Omega^j_{R/A}$ has filtered amplitude in $[ja,jb]$, and an argument using the Nygaard and conjugate filtrations shows that the filtration on $\bZ_p(i)(R/A)$ is concentrated in $[0,(i-1)b+c]$, i.e. 
  \[
    \bZ_p(i)(R/A) \simeq \F^{[0,(i-1)b+c]} \bZ_p(i)(R/A).
  \]
\end{remark}

\begin{corollary}\label{cor:k_groups}
    For $i\geq 1$, $\K_{2i-1}(\Oscr_K/\varpi^n;\bZ_p)\iso\H^1(\bZ_p(i)(\Oscr_K/\varpi^n))$ and, for
    $i\geq 2$, $\K_{2i-2}(\Oscr_K/\varpi^n;\bZ_p)\iso\H^2(\bZ_p(i)(\Oscr_K/\varpi^n))$.
\end{corollary}

\begin{proof}
    Indeed the motivic spectral sequence computing $\TC$ and hence $p$-adic $\K$-theory degenerates
    for degree reasons and there are no possible extensions by Corollary~\ref{cor:zigzag}.
\end{proof}

\begin{remark}\label{rem:in_the_limit}
    The analogue of Corollary~\ref{cor:k_groups} holds in the limit for $\Oscr_K$ as well.
    This can be proved from Corollary~\ref{cor:k_groups} using $p$-adic continuity for syntomic cohomology
    (see~\cite[Cor.~7.4.11]{bhatt-lurie-apc}) and $\TC(-;\bZ_p)$ (see~\cite[Thm.~5.20]{cmm}), or it
    can be proved directly by using crystalline degeneration again to reduce to
    characteristic $p$.
\end{remark}

Already, it is possible to extract qualitative information about the $\K$-groups using crystalline
degeneration for syntomic cohomology and Sulyma's calculation~\cite{sulyma}. The following result
was proved in the unramified case in~\cite{angeltveit}.

\begin{proposition}[Angeltveit quotient]\label{prop:angeltveit}
    For $i\geq 1$ and any finite chain ring $R=\Oscr_K/\varpi^n$ with residue field $\bF_q$,
    $$\frac{\left|\H^1(\bZ_p(i)(R))\right|}{\left|\H^2(\bZ_p(i)(R))\right|}=q^{i(n-1)}.$$
\end{proposition}

\begin{proof}
    By Corollary~\ref{cor:zigzag}, $$\bZ_p(i)(R)\we\fib\left(\F^{[1,in-1]}\N^{\geq
    i}\Prismhat_R\{i\}\xrightarrow{\can-\varphi}\F^{[1,in-1]}\Prismhat_R\{i\}\right).$$
    On associated graded pieces, $$\bigoplus_{j=1}^{in-1}\gr^j_\F\N^{\geq
    i}\Prismhat_R\{i\}\we\bigoplus_{j=1}^{in-1}\gr^j_\F\N^{\geq
    i}\Prismhat_{\bF_q[z]/z^n}$$
    (where the Breuil--Kisin twist is trivialized on the right as we are in the
    crystalline setting there)
    and similarly for $\F^{[1,in-1]}\Prism_R\{i\}$.

    By~\cite{sulyma}, $\gr^j_\F\Prismhat_{\bF_q[z]/z^n}\iso W(\bF_q)/\{j,n\}[-1]$
    for $j\geq 1$,
    where $$\{j,n\}=\begin{cases}
        n&\text{if $n\mid j$,}\\
        j&\text{otherwise.}
    \end{cases}$$
    On the other hand, $\gr^j_\F\N^{\geq i}\Prismhat_{\bF_q[z]/z^n}\iso
    W(\bF_q)/p^{\epsilon(i,j)}\{j,n\}[-1]$ for $j\geq 1$,
    where $$\epsilon(i,j)=\begin{cases}
        1&\text{if $\lfloor\tfrac{j}{n}\rfloor<\lceil\tfrac{j}{n}\rceil\leq i$}\\
        0&\text{otherwise.}
    \end{cases}$$
    It follows that the quotient appearing in the statement of the proposition
    has order $$\prod_{j=1}^{in-1}q^{\epsilon(i,j)}.$$
    Note that $\epsilon(i,j)=1$ if and only if $j$ is not divisible by $n$. In
    the range $1,\ldots,in-1$ there are $i-1$ numbers $j$ divisible by $n$. So, the
    product has order $q^{in-1-(i-1)}=q^{i(n-1)}$, as desired.
\end{proof}

\begin{remark}
    Proposition~\ref{prop:angeltveit} can be proved directly using the machinery of
    prismatic envelopes; see Lemma~\ref{lem:isogeny}.
\end{remark}

Finally, we observe that while $\Prism_{R/A}^{(1)}$ is generally not complete for
either the $\F$-filtration or the Nygaard filtration, the two completions are closely related.

\begin{theorem}\label{thm:n_v_f}
  Let $R/A$ be a filtered prismatic $\delta$-pair.
  \begin{enumerate}
      \item[{\em (1)}] If $\F^{\geq\star}R$ is complete and $\L_{R/A}$ is a perfect filtered $R$-module (i.e. in the thick subcategory generated by filtered shifts of $R$), then the Nygaard-completion of $\Prism_{R/A}^{(1)}$ is also $\F$-complete.
      \item[{\em (2)}] If $\gr^0 R$ is smooth over $\gr^0 \overline{A}$, then the $\F$-completion of $\Prism_{R/A}^{(1)}$ is also Nygaard-complete.
  \end{enumerate}
\end{theorem}
\begin{proof}
  For the first statement, it suffices to check that each term in the Nygaard associated
    graded is $\F$-complete, since this exhibits the Nygaard completion as limit of
    $\F$-complete objects. As $\gr^i_\N\Prism^{(1)}_{R/A} \simeq \F^{\conj}_{\leq i}
    \Prismbar_{R/A}$, it suffices to check that each of the $\gr_i^{\conj}
    \Prismbar_{R/A}\simeq \widehat{\L\Omega}^i_{R/A}[-i]$ is $\F$-complete. By assumption,
    $\L_{R/A}$ is a perfect filtered
    $R$-module, so the same is true for each of its exterior powers. Since
    $\F^{\geq\star} R$ is complete, the $\widehat{\L\Omega}^i_{R/A}[-i]$ are all complete.

    For the second statement, we need to prove that each term in the $\F$-associated graded is Nygaard-complete.
    To do so, we observe that there is a limit
    $\N^{\geq\star}\Prism^{(1)}_{R/A}\we\lim_j\left(\N^{\geq\star}\Prism_{R/A}^{(1)}\otimes_{I^\star
    A}(I^\star A/I^j A)\right)$ in $(p,I)$-complete filtered complexes: this follows
    from $I$-completeness of $\Prism_{R/A}^{(1)}$ upon forgetting the Nygaard
    filtration and from the fact that the filtration is eventually constant in each
    weight on the Nygaard associated graded pieces.
    It is now enough to show that each
    $\gr^k_\F\left(\N^{\geq\star}\Prism_{R/A}^{(1)}\otimes_{I^\star A}(I^\star A/I^j
    A)\right)_p^\wedge$ is a
    complete filtration and for this we can reduce to the case when $j=1$. 
    Recall that the filtered tensor product $\left(\N^{\geq\star}\Prism^{(1)}_{R/A}\otimes_{I^\star
    A}A/I\right)_p^\wedge$ agrees with the Hodge filtration on $p$-complete derived de Rham cohomology,
    for example by~\cite[Thm.~1.8]{li-liu} or~\cite[Cor.~5.2.8]{bhatt-lurie-apc}.
    We prove the following stronger claim: if $r$ is an upper bound for the rank of
    the projective $\gr^0_\F R$-module
    $\L_{\gr^0_\F R/\gr^0_\F A}$,
    then the induced Nygaard/Hodge filtration on
    $\gr^h_\F(\Prism^{(1)}_{R/A} \otimes_{I^\star A} A/I)_p^\wedge\we\gr^h_\F(\dR_{R/\overline{A}})_p^\wedge$
    is bounded by $h+r$ for each $h\geq 0$.
    Indeed, this statement can be checked after base change along $A\to \gr^0_\F A$, so
    we may assume the $\F$-filtration on $A$ to be trivial.
    We may also pass to the $\F$-associated graded to assume that $R$ is in fact graded instead of just filtered.
    Finally, we may write any graded $A$-algebra with the given smooth degree $0$ part as sifted colimit of smooth graded $A$-algebras, constant on the degree $0$ part.
    So the statement is reduced to the case where $R$ is graded smooth over $A$, and
    $\gr^0_\F R$ has relative dimension bounded by $r$. In that case,
    the $p$-completed derived de Rham complex
    $(\dR_{R/\overline{A}})_p^\wedge$ is in fact Hodge-complete (see for
    example~\cite[Prop.~E.12]{bhatt-lurie-apc}) and
    $\widehat{\L\Omega}^h_{R/\overline{A}}$ has $\F$-filtration bounded below by $h-r$ since
    $\widehat{\L\Omega}^j_{\gr^0_\F R/\overline{A}}$ vanishes for $j>r$.
    Thus, $\gr^h_\F (\Prism^{(1)}_{R/A}\otimes_{I^\star A} A/I)_p^\wedge$ has Nygaard/Hodge filtration bounded above by $h+r$, as claimed.
\end{proof}

\subsection{Relative-to-absolute descent}\label{sub:liuwang}

In order to compute the syntomic complexes, we will compute part of the
prismatic cohomology of $\Oscr_K/\varpi^n$. Specifically, by
Corollary~\ref{cor:zigzag}, it is enough to compute
$$\F^{[1,in-1]}\N^{\geq
i}\Prismhat_{\Oscr_K/\varpi^n}\{i\}\quad\text{and}\quad\F^{[1,in-1]}\Prismhat_{\Oscr_K/\varpi^n}\{i\}$$
and the canonical and Frobenius maps between these complexes.
Unfortunately, these absolute prismatic cohomological objects are not abelian groups but objects of $\D(\bZ_p)$ concentrated in cohomological
degrees $0$ and $1$. However, by working relative to the Breuil--Kisin prism $W(k)\llbracket
z_0\rrbracket$, where $z_0$ maps to $\varpi$, the analogous relative terms become discrete, i.e.,
they are
abelian groups and admit algebraic descriptions in the form of prismatic
envelopes. The absolute prismatic cohomology can then be understood via descent along
$W(k)\rightarrow W(k)\llbracket z_0\rrbracket$ using that
$\Prismpackage_{R/\bZ_p}\we\Prismpackage_{R/W(k)}$ by~\cite[Thm.~1.2]{akn-delta}. This is a complementary approach to the
quasisyntomic descent introduced in~\cite{bms2} and was introduced, for $\THH$ and its
variants, by Liu and Wang in~\cite{liu-wang}.\footnote{In~\cite{krause-nikolaus}, two of us used a
closely related fiber sequence
$\THH(R;\bZ_p)\rightarrow\THH(R/\bS[z];\bZ_p)\rightarrow\THH(R/\bS[z];\bZ_p)[2]$ to
give a new computation of the homotopy groups of $\THH(R;\bZ_p)$ for $R$ of the form $\Oscr_K$ or
$\Oscr_K/\varpi^n$. This is a fiber sequence of spectra and the first map is naturally a map of
cyclotomic spectra, but the cofiber term has a complicated cyclotomic structure. It has never been
clear how to
use this to give a computation, even of $\TP(R;\bZ_p)$, the issue being that the structure of the $S^1$-Tate
construction of the cofiber remains mysterious. The descent approach below can be viewed either as a
way around this or as a way to access this mystery term.}

In more detail,
we consider the cosimplicial $\delta$-ring $A^\bullet$ given by taking the
$p$-complete
Amitsur complex of $W(k)\rightarrow W(k)\langle z_0\rangle$
and completing each term with respect to the augmentation ideal given by
mapping to the constant cosimplicial diagram $\bZ_p$ induced by sending $z_0$
to zero. In other words, $A^\bullet$ is
\[
    W(k)\llbracket z_0\rrbracket  \stack{3} W(k)\llbracket z_0,z_1\rrbracket  \stack{5} \ldots.
\]
Moreover, $A^\bullet$ is a cosimplicial filtered $\delta$-ring in the sense
of~\cite[Sec.~10]{akn-delta}, where each
$A^s$ is given the filtration where $z_0,\ldots,z_s$ have weight $1$ and where
$\delta(z_0)=\cdots=\delta(z_s)=0$.
There is a filtered map from $A^\bullet$ to $\Oscr_K/\varpi^n$, given by sending
each of $z_0,\ldots,z_s$ to $\varpi$, which makes $(A^\bullet,\Oscr_K/\varpi^n)$ into a cosimplicial filtered
$\delta$-pair.

\begin{construction}\label{const:relative_syntomic}
    In~\cite{akn-delta}, we introduce relative syntomic complexes
    $$\bZ_p(i)(R/A)\we\fib\left(\N^{\geq
    i}\Prismhat^{(1)}_{R/A}\{i\}\xrightarrow{\can-\varphi}\Prismhat^{(1)}_{R/A}\{i\}\right)$$
    for any $i\in\bZ$,
    which can be defined from the prismatic package $\Prismpackage_{R/A}$ of any
    $\delta$-pair $(A,R)$. In the case of a filtered $\delta$-pair, there is an induced
    natural filtration $\F^\star\bZ_p(i)(R/A)$ on the relative syntomic complex.
\end{construction}

The following theorem follows from our work in~\cite{akn-delta}.

\begin{theorem}[Relative-to-absolute descent]\label{thm:bkdescent}
    Let $R=\Oscr_K/\varpi^n$. The following maps are equivalences:
    \begin{enumerate}
        \item[{\em (a)}]
            $\F^{\geq\star}\Prism_{R/W(k)}\{i\}\rightarrow\Tot\F^{\geq\star}\Prism_{R/A^\bullet}\{i\}$
            for all $i\in\bZ$;
        \item[{\em (b)}]
            $\F^{\geq\star}\N^{\geq
            j}\Prismhat_{R/W(k)}^{(1)}\{i\}\rightarrow\Tot\F^{\geq\star}\N^{\geq
            j}\Prismhat_{R/A^\bullet}^{(1)}\{i\}$ for all $i,j\in\bZ$;
        \item[{\em (c)}]
            $\F^{\geq\star}\bZ_p(i)(R)\rightarrow\Tot\F^{\geq\star}\bZ_p(i)(R/A^\bullet)$ for all
            $i\in\bZ$;
        \item[{\em (d)}]
            $\F^{[a,b]}\bZ_p(i)(R)\rightarrow\Tot\F^{[a,b]}\bZ_p(i)(R/A^\bullet)$ for all
            integers $i\geq 1$ and all intervals $[a,b]\subseteq\bZ$.
    \end{enumerate}
\end{theorem}

\begin{proof}
    Part (d) follows from part (c), which follows from (a) and (b).
    On $\F^0$, parts (a) and (b) follow from~\cite[Thm.~1.1 and Cor.~1.2]{akn-delta}.
    Therefore, to prove (a) and (b) it is enough to check on $\F$-associated graded pieces and even
    on the $\F$-associated graded pieces of the Hodge--Tate and the Nygaard filtrations, where it
    follows from descent for the graded cotangent complex.
\end{proof}

\begin{proposition}[Cosimplicial prismatic envelope]\label{prop:cosimpprism}
    The cosimplicial filtered $\delta$-ring
    $\F^{\geq\star}\Prism_{R/A^\bullet}$ is discrete; the same is true for
    $\F^{\geq\star}\N^{\geq i}\Prismhat_{R/A^\bullet}^{(1)}$ and
    $\F^{\geq\star}\Prismhat_{R/A^\bullet}^{(1)}$.
\end{proposition}

\begin{proof}
    For $\Prism_{R/A^\bullet}$ the levelwise discreteness follows
    from~\cite[Prop.~3.13]{prisms} and for $\F^\star\Prism_{R/A^\bullet}$ it follows
    from~\cite[Prop.~10.41]{akn-delta}. For $\F^{\geq\star}\N^{\geq
    i}\Prismhat_{R/A^\bullet}^{(1)}$, one reduces, using the conjugate filtration, to the fact $\L_{R/A^\bullet}[-1]$ is a levelwise flat filtered $R$-module.
\end{proof}

\begin{proposition}[Cosimplicial freeness]\label{prop:cosimpfree}
    For $[a,b]\subseteq\bZ$,
    both $\F^{[a,b]}\N^{\geq i}\Prismhat_{R/A^\bullet}^{(1)}\{i\}$ and
    $\F^{[a,b]}\Prismhat_{R/A^\bullet}^{(1)}\{i\}$ are cosimplicial $\bZ_p$-modules which are
    levelwise finite and free.
\end{proposition}

\begin{proof}
  This follows from Proposition~\ref{prop:cosimpprism} and Proposition~\ref{prop:f_filtered_statement}.
\end{proof}

\begin{proposition}\label{prop:bounds32}
    Let $R=\Oscr_K/\varpi^n$. Then, $\gr^j_\F\Prismhat_R\{i\}$ and
    $\gr^j_\F\N^{\geq i}\Prismhat_R\{i\}$ are in $\D(\bZ_p)_{[-1,-1]}$ for
    each $j,i\geq 1$.
\end{proposition}

\begin{proof}
    See Proposition~\ref{prop:degeneration} and Lemma~\ref{lem:truncated}.
\end{proof}

\begin{proof}
    By Proposition~\ref{prop:f_filtered_statement}, all terms are finitely generated free
    $W(k)$-modules of the same ranks. The map is an injection as is seen by rewriting $z_0^n$ in
    terms of $f_{00}$. Thus, it is an isogeny.
\end{proof}

\begin{lemma}\label{lem:cohomology-torsion}
    For $R=\Oscr_K$ or $\Oscr_K/\varpi^n$, the limits $\F^{[a,b]}\N^{\geq i}\Prismhat_{R}\{i\}$ and
    $\F^{[a,b]}\Prismhat_{R}\{i\}$ of the descent diagrams
    \[
        \F^{[a,b]} \N^{\geq i}\Prismhat^{(1)}_{R/A^0 }\{i\} \stack{3} \cdots
    \]
    and
    \[
        \F^{[a,b]} \Prismhat^{(1)}_{R/A^0 }\{i\} \stack{3} \cdots
    \]
    have finitely generated
    torsion cohomology groups for each $i\in\bZ$ and each finite interval $[a,b]\subseteq\bZ$ with $a\geq 1$.
\end{lemma}

\begin{proof}
    This follows from a crystalline degeneration argument. For example, the associate graded of the
    $\varpi$-adic filtration on $\Oscr_K$ is $k\llbracket z_0\rrbracket$. Its prismatic cohomology
    agrees with its crystalline cohomology, but, as it admits a lift to $W(k)$, its crystalline
    cohomology is identified with the de Rham cohomology of the lift, which is computed by the
    complex $$W(k)\llbracket z_0\rrbracket\xrightarrow{\d} W(k)\llbracket
    z_0\rrbracket\,\d z_0.$$
    Restricting to $\F$-filtration $[a,b]$ amounts to restricting to the subcomplex
    $$W(k)\cdot\{z_0^a,\ldots,z_0^b\}\xrightarrow{\d}W(k)\cdot\{z_0^{a-1}\d
    z_0,\cdots,z_0^{b-1}\d
    z_0\},$$ which has finitely generated torsion cohomology, since $a\geq 1$. The other cases are left to the reader.
\end{proof}

\subsection{A finite complex computing syntomic cohomology}\label{sub:construction}

In the absolute case, $\F^{[1,in-1]}\N^{\geq i}\Prismhat_R\{i\}$ and $\F^{[1,in-1]}\Prismhat_R\{i\}$
are each in $\D(\bZ_p)_{[-1,-1]}$ for $i\geq 1$ by Propositions~\ref{prop:degeneration} and~\ref{prop:bounds32}, so
the cochain complexes associated to the cosimplicial abelian groups
$$\F^{[1,in-1]}\N^{\geq
i}\Prismhat_{R/A^\bullet}^{(1)}\{i\}\quad\text{and}\quad\F^{[1,in-1]}\Prismhat_{R/A^\bullet}^{(1)}\{i\}$$
are exact in cohomological degrees at least $2$.

It follows that we can compute $\F^{[1,in-1]}\N^{\geq i}\Prismhat_R$ as the fiber
of an injective map \begin{equation}\label{eq:fiber}\F^{[1,in-1]}\N^{\geq
i}\Prismhat_{R/A^0}^{(1)}\{i\}\to\ker\left(\F^{[1,in-1]}\N^{\geq
    i}\Prismhat_{R/A^1}^{(1)}\{i\}\xrightarrow{d^0-d^1+d^2}\F^{[1,in-1]}\N^{\geq
i}\Prismhat_{R/A^2}^{(1)}\{i\}\right)\end{equation}
and similarly for $\F^{[1,in-1]}\Prismhat_R\{i\}$.

Lemma~\ref{lem:cohomology-torsion} implies that $\F^{[1,in-1]}\N^{\geq
i}\Prismhat_R\{i\}$ and $\F^{[1,in-1]}\Prismhat_R\{i\}$ are finitely generated torsion
abelian groups, while Proposition~\ref{prop:cosimpfree} implies
that the terms in~\eqref{eq:fiber} have the same rank over $\bZ_p$.

We summarize the reductions of this section in the following theorem.

\begin{theorem}\label{thm:main}
    For each $i\geq 1$, the complex
    $\bZ_p(i)(R)\we\F^{[1,in-1]}\bZ_p(i)(R)$ is equivalent to the total complex
    (i.e., total fiber)  of the commutative
    square
    $$\xymatrix{
        \F^{[1,in-1]}\N^{\geq i}\Prismhat_{R/A^0}^{(1)}\{i\} \ar[r]\ar[d]_{\can-\varphi}&\ker\left(\F^{[1,in-1]}\N^{\geq
        i}\Prismhat_{R/A^1}^{(1)}\{i\}\xrightarrow{d^0-d^1+d^2}\F^{[1,in-1]}\N^{\geq
        i}\Prismhat_{R/A^2}^{(1)}\{i\}\right)\ar[d]_{\can-\varphi}\\
        \F^{[1,in-1]}\Prismhat_{R/A^0}^{(1)}\{i\}
        \ar[r]&\ker\left(\F^{[1,in-1]}\Prismhat_{R/A^1}^{(1)}\{i\}\xrightarrow{d^0-d^1+d^2}\F^{[1,in-1]}\Prismhat_{R/A^2}^{(1)}\{i\}\right).
    }$$
\end{theorem}

To compute the total complex of Theorem~\ref{thm:main},
we analyze the prismatic cohomology of $R$ over $A^0$ and $A^1$ via the
prismatic envelopes introduced in~\cite{prisms}. This lets us find
bases of all of the terms involved and to compute the maps between them; see
Sections~\ref{sec:envelopes} and~\ref{sec:bko}.
Moreover, we identify in Corollary~\ref{cor:syntomicsquare} the horizontal maps with a kind of `connection' on
$\Prismhat_{R/A^0}^{(1)}\{i\}$, not unlike the description of prismatic cohomology of
$k[z]/z^n$ in equation~\eqref{eq:dp_connection}.

\section{Envelope algebra}\label{sec:envelopes}

In~\cite{prisms}, Bhatt and Scholze give a prismatic envelope description of $\Prism_{R/A}$ when
$(A,d)$ is a bounded, oriented prism and $R=A/(d,r)$, where $r\in A$ is an element which defines a non-zero
divisor in $\overline{A}=A/d$. The description is as a ``prismatic envelope'', which is defined as
the $(p,d)$-completion of a pushout
$$\label{eqn:envelope_pushout}\xymatrix{
    A\{x\}\ar[r]^{x\mapsto da-r}\ar[d]_{x\mapsto 0}&A\{a\}\ar[d]\\
    A\ar[r]&A\{a\}/(da-r)_\delta
}$$
in $\delta$-$A$-algebras, where $A\{a\}$ denotes the free $\delta$-$A$-algebra on a
generator $a$.\footnote{Recall that $A\{a\}$ is isomorphic to the countably-generated polynomial
ring $A[a,\delta(a),\delta^2(a),\ldots])$ and that the $\delta$-ideal $(da-r)_\delta$ is equal to
the ideal $(da-r,\delta(da-r),\delta^2(da-r),\ldots)$.} We write $A\{\tfrac{r}{d}\}$ for the pushout and
$A\{\tfrac{r}{d}\}_{p,d}^\wedge$ for its completion. The result is a flat $\delta$-$A$-algebra and
the construction commutes with base change of maps of bounded, oriented prisms
by~\cite[Prop.~3.13]{prisms}.
Given an object $(B,I)$ in the relative prismatic site $(R/A)_\Prism$, one has that $R$ maps to
$\overline{B}=B/I$ so that the image of $r$ in $B$ must be divisible by $I$. Using that $I=dB$ and
writing $r=da\in I$ for some unique $a$, we see there is a natural induced map
$A\{\tfrac{r}{d}\}_{p,d}^\wedge\rightarrow B$ of prisms over $A$. Bhatt and Scholze show that in
fact $A\{\tfrac{r}{d}\}_{p,d}^\wedge$ represents a final object of $(R/A)_\Prism$ so that
$\Prism_{R/A}\we A\{\tfrac{r}{d}\}_{p,d}^\wedge$.

In this section, we examine the algebra of the construction of $A\{\tfrac{r}{d}\}_{p,d}^\wedge$ and
closely-related variants. In particular, we show that, as a commutative ring, the prismatic envelope
is the $(p,d)$-completed quotient of a countably generated polynomial ring over $A$ by countably
many relations which are $(p,d)$-completely Koszul-regular. Then, we give generators and relations descriptions of
$$\F_{\leq\star}^\conj\Prismbar_{R/A},\quad\gr^\star_\N\Prismhat^{(1)}_{R/A},\quad\text{and}\quad\N^{\geq\star}\Prismhat_{R/A}^{(1)},$$
as increasingly filtered, graded, and decreasingly filtered commutative rings.
This allows us to give additive generators of these objects. 
When $A=W(k)\llbracket z_0,\ldots,z_s\rrbracket$ is viewed as filtered prism with
respect to some Eisenstein polynomial $E(z_j)$ where $z_0,\ldots,z_s$ have weight $1$ and when $R=\Oscr_K$ or
$\Oscr_K/\varpi^n$, then we give additive bases for the filtered pieces
$\F^{[a,b]}\N^{\geq i}\Prismhat_{R/A}^{(1)}$. Eventually, these additive bases are how we will
compute the syntomic cohomology of $\Oscr_K/\varpi^n$.

\begin{notation}
    Throughout this section, we consider filtered rings and modules.
    Here we will often think of a decreasingly filtered module $\F^{\geq \star} M$ as a graded module $\bigoplus_{i\in\bZ}
    \F^{\geq i} M$ together with an action by the graded ring $\bZ[\tau]$, where $\tau$ acts by the
    transition maps $\F^{\geq i} M\to \F^{\geq i-1} M$.
\end{notation}

\subsection{Some formulas about $\delta$}

Recall the following facts about $\delta$-rings $A$. We have
\begin{equation}\label{eq:delta_add}
    \delta(x+y)=\delta(x)+\delta(y)+w_1(x,y)
\end{equation}
for all $x,y\in A$,
where
\begin{equation}\label{eq:delta_w1}
    w_1(x,y)=(x^p+y^p-(x+y)^p)/p=-\left(\sum_{j=1}^{p-1}\binom{p}{j}x^jy^{p-j}\right)/p,
\end{equation}
a polynomial in $x$ and $y$ with integer coefficients. If $x,y\in A$, then
\begin{align}\label{eq:delta_mul}
    \delta(xy)&=\delta(x)y^p+x^p\delta(y)+p\delta(x)\delta(y)\\
        &=\delta(x)y^p+\varphi(x)\delta(y)\\
        &=\delta(x)\varphi(y)+x^p\delta(y).
\end{align}
If $x\in A$, then
\begin{equation}\label{eq:delta_phi}
    \varphi(x)\in(p,x)
\end{equation}
since $\varphi(x)=x^p+p\delta(x)$.

\begin{lemma}\label{lem:delta_mod}
    If $A$ is a $\delta$-ring and $x,y\in A$, then
    \begin{itemize}
        \item $\delta(x+y)=\delta(x)+\delta(y)\pmod x$,
        \item $\delta(x-y)=\delta(x)-\delta(y)\pmod {x-y}$, and
        \item $\delta(xy)=\delta(x)\varphi(y)\pmod x$.
    \end{itemize}
\end{lemma}

\begin{proof}
    These relations follow from the addition and multiplication formulas for $\delta$.
\end{proof}

\begin{lemma}\label{lem:delta_power}
    If $A$ is a $\delta$-ring, then $\delta(x^p)= 0\pmod p$ for all $x\in A$.
\end{lemma}

\begin{proof}
    Use $x^{p^2}+p\delta(x^p)=\varphi(x^p)=\varphi(x)^p=(x^p+p\delta(x))^p$ to prove the lemma in the
    free $\delta$-ring on one generator $x$, which is $p$-torsion free, and deduce it in general.
\end{proof}

\begin{lemma}
    If $A$ is a $\delta$-ring, then $\delta(ux^p)=\delta(u)x^{p^2}\pmod p$ for all $u,x\in A$.
\end{lemma}

\begin{proof}
    Indeed, $\delta(ux^p)=\delta(u)x^{p^2}+\varphi(u)\delta(x^p)=\delta(u)x^{p^2}\pmod p$ by
    Lemma~\ref{lem:delta_power}.
\end{proof}

\subsection{Prismatic envelopes}

\begin{definition}[Koszul-regularity]
    Say that a sequence $(x_0,x_1,\ldots)$ of elements in a commutative ring $R$ is Koszul-regular if the associated
    Koszul complex $K(x_\bullet)$ has homology concentrated in degree $0$, in which case it is a resolution of $R/(x_0,x_1,\ldots)$. Koszul-regularity is implied by
    regularity. If $I\subseteq R$ is a finitely generated ideal, then such a sequence is
    called
    $I$-completely Koszul-regular if the $I$-completed Koszul complex $K(x_\bullet)_I^\wedge$
    has homology concentrated in degree $0$, in which case it is a resolution of the derived $I$-adic completion of
    $R/(x_0,x_1,\ldots)$.
\end{definition}

\begin{lemma}
  \label{lem:relations}
    Let $(A,d)$ be an oriented prism and fix elements $a,r\in A$.
    Let $\lambda_0=-\tfrac{1}{\delta(d)}$ and $R_0=\tfrac{\delta(r)}{\delta(d)}$ and
    inductively define elements
  \[
    \lambda_{u+1} = \frac{\lambda_u^p}{1-\delta(d^{p^{u+1}} \lambda_u)}
  \]
  and
  \[
    R_{u+1} = \frac{1}{1-\delta(d^{p^{u+1}}\lambda_u)}(\delta(R_u) + w_1(d^{p^{u+1}}\lambda_u \delta^{u+1}(a), R_u))
  \]
  for $u\geq 0$. Then, the following statements hold.
    \begin{enumerate}
        \item[{\em (i)}]  The elements $R_u$, for $u\geq 0$, are polynomials in the $\delta$-powers $\delta^j(a)$ for $j\leq u$ (with
            coefficients in $\bZ_p\{d,r\}[\delta(d)^{-1}]^{\wedge}_{p,d}$), of total degree
            $<p^{u+1}$ with respect to the grading where $|\delta^j(a)| = p^j$. In particular,
            $R_u$ involves only powers of $\delta^u(a)$ with exponent $<p$. 
        \item[{\em (ii)}] The elements $\lambda_u$ are units in $A$ for $u\geq 0$.
        \item[{\em (iii)}]
            The $\delta$-ideal $(da-r)_\delta$ agrees with the ideal generated by $(da-r)$ and
            the elements
          \[
            \delta^u(a)^p - (-p + d^{p^{u+1}} \lambda_u) \delta^{u+1}(a) - R_u.
          \]
    \end{enumerate}
\end{lemma}
\begin{proof}
    To prove (iii), observe that
  \[
    \delta^u(a)^p - (-p + d^{p^{u+1}} \lambda_u) \delta^{u+1}(a) - R_u
    = \varphi(\delta^u(a)) - d^{p^{u+1}}\lambda_u\delta^{u+1}(a) - R_u.
  \]
 We write $I_u$ for the ideal
  \[
    (da-r, \delta(da-r),\ldots, \delta^u(da-r)).
  \]
  To prove that $(da-r)_{\delta}=\bigcup I_u$ agrees with the ideal generated by $(da-r)$ and the elements $\varphi(\delta^u(a)) - d^{p^{u+1}}\lambda_u\delta^{u+1}(a) - R_u$, we claim there exist units $\mu_u$ with
  \[
    \delta^{u+1}(da-r) = \mu_u \left(\varphi(\delta^u(a)) -
    d^{p^{u+1}}\lambda_u\delta^{u+1}(a) - R_u\right)\pmod{I_u}.
  \]
  This in particular implies
  \[
    I_{u+1} = I_u + (\delta^{u+1}(da-r)) = I_u + (\varphi(\delta^u(a)) - d^{p^{u+1}}\lambda_u\delta^{u+1}(a) - R_u),
  \]
  which is what we want to prove.
    To prove the claim, we proceed by induction. We will make repeated use of
    Lemma~\ref{lem:delta_mod}
  Observe that
  \[
      \delta(da - r) = \delta(da) - \delta(r) \pmod {da-r},
  \]
  and
  \[
      \delta(da)-\delta(r)=\delta(d)\varphi(a) + d^p \delta(a) - \delta(r) = \delta(d) \cdot (\varphi(a) - d^p \lambda_0\delta(a) - R_0),
  \]
  showing the claim for $u=0$ (with $\mu_0 = \delta(d)$).

  For the inductive step, assume the claim for $u$. Observe that $\delta(I_u)\subseteq I_{u+1}$, so
  \[
    \delta^{u+2}(da-r) = \delta\left(\mu_u (\varphi(\delta^u(a)) -
    d^{p^{u+1}}\lambda_u\delta^{u+1}(a) - R_u)\right)\pmod {I_{u+1}}.
  \]
    Next, using Lemma~\ref{lem:delta_mod}, we find that
  \begin{align*}
    \delta\big(\mu_u (\varphi(\delta^u(a)) - d^{p^{u+1}}\lambda_u\delta^{u+1}(a) - R_u)\big)
    &= \varphi(\mu_u) \delta(\varphi(\delta^u(a)) - d^{p^{u+1}}\lambda_u\delta^{u+1}(a)
      - R_u)\pmod{I_{u+1}}\\
    &= \varphi(\mu_u) \big(\delta(\varphi(\delta^u(a))
      -\delta(d^{p^{u+1}}\lambda_u\delta^{u+1}(a) + R_u)\big)\pmod{I_{u+1}},
  \end{align*}
  using that $\varphi(\delta^u(a)) - d^{p^{u+1}}\lambda_u\delta^{u+1}(a) - R_u\in I_{u+1}$.

  Now $\delta(\varphi(\delta^u(a))) = \varphi(\delta^{u+1}(a))$, and
  \[
    \delta(d^{p^{u+1}}\lambda_u\delta^{u+1}(a) + R_u)
    = \delta(d^{p^{u+1}}\lambda_u) \varphi(\delta^{u+1}(a)) + d^{p^{u+2}}\lambda_u^p \delta^{u+2}(a) + \delta(R_u) + w_1(d^{p^{u+1}}\lambda_u\delta^{u+1}(a), R_u).
  \]
  Their difference can be factored as
  \begin{align*}
      &\delta(\varphi(\delta^u(a))) - \delta(d^{p^{u+1}}\lambda_u \delta^{u+1}(a)+R_u)\\
      &=\varphi(\delta^{u+1}(a))-\left(\delta(d^{p^{u+1}}\lambda_u)
      \varphi(\delta^{u+1}(a)) + d^{p^{u+2}}\lambda_u^p \delta^{u+2}(a) + \delta(R_u) +
      w_1(d^{p^{u+1}}\lambda_u\delta^{u+1}(a), R_u)\right)\\
      &=(1-\delta(d^{p^{u+1}} \lambda_u)) \cdot (\varphi(\delta^{u+1}(a)) - d^{p^{u+2}} \lambda_{u+1} \delta^{u+2}(a) - R_{u+1}),
  \end{align*}
  using the definition of $\lambda_{u+1}$ and $R_{u+1}$,
  which completes the inductive step (with $\mu_{u+1}=\varphi(\mu_u) \cdot (1-\delta(d^{p^{u+1}}\lambda_u))$).

Part (i) follows easily from the provided recursion, working in the graded $\delta$-ring
\[
  \bZ_p\{d,r,a\}[\delta(d)^{-1}]^\wedge_{p,d}
\]
where $|a|=1$ and $|d|=|r|=0$. Part (ii) follows inductively from
Lemma~\ref{lem:delta_mod}.
\end{proof}

\begin{remark}
    If $\delta(r)=0$, observe that the $R_u$ are all inductively zero.
\end{remark}

\begin{proposition}\label{prop:envelope_regular}
    Let $A$ be a bounded prism with orientation $I=(d)$ and let $r_1,\ldots,r_c\in A$ map to a
    $p$-completely
    Koszul-regular sequence in $\overline{A}=A/d$.
    Then, the prismatic envelope
    $A\{\tfrac{r_1}{d},\ldots,\tfrac{r_c}{d}\}_{(p,d)}^\wedge=\left(A\{a_1,\ldots,a_c\}/(da_1-r_1,\ldots,da_c-r_c)_\delta\right)_{(p,d)}^\wedge$ is the quotient of the
    $(p,d)$-complete free $\delta$-ring $A\{a_1,\ldots,a_c\}_{(p,d)}^\wedge$ on generators
    $a_1,\ldots,a_c$ by the $(p,d)$-completely Koszul-regular sequence of relations
    \[
      (da_1-r_1,\ldots, da_c-r_c)
    \]
    and
    \[
      \delta^u(a_v)^p - (-p + d^{p^{u+1}} \lambda_u)\delta^{u+1}(a_v) - R_{uv}
    \]
    for $u\geq 0$ and $1\leq v\leq c$,
    where $R_{uv}$ is the polynomial from Lemma \ref{lem:relations} (for $a=a_v$ and $r=r_v$).
\end{proposition}
\begin{proof}
  By Lemma \ref{lem:relations}, the $\delta$-ideal $(da_1-r_1,\ldots,da_v-r_v)_\delta$ agrees with the ideal generated by $(da_1-r_1,\ldots,da_c-r_c)$ and the
$\delta^u(a_v)^p - (-p + d^{p^{u+1}} \lambda_u)\delta^{u+1}(a_v) - R_{uv}$.

  Thus it suffices to prove that these form a $(p,d)$-completely Koszul-regular sequence in $A[\delta^u(a_v) \mid u\geq 0, 1\leq v\leq c]^\wedge_{p,d}$.
    It suffices to do this after base change to $\overline{A}=A/d$ and hence
    it suffices to prove that the images of the elements
  \[
\delta^u(a_v)^p - (-p + d^{p^{u+1}} \lambda_u)\delta^{u+1}(a_v) - R_{uv}
  = \delta^u(a_v)^p + p \delta^{u+1}(a_v) - R_{uv}
  \]
  form a $p$-completely Koszul-regular sequence in $R[\delta^u(a_v)\mid u\geq 0, 1\leq v\leq
    c]^\wedge_p$. This ring admits an ascending filtration where $\delta^u(a_v)$ is in
    filtration $\leq p^u$. With respect to that filtration, $R_{uv}$ has filtration
    strictly smaller $p^{u+1}$, as discussed in Lemma \ref{lem:relations}. As the
    leading terms with respect to this filtration
  \[
    \delta^u(a_v)^p +p\delta^{u+1}(a_v)
  \]
  of the relations form a regular sequence, the relations do too.
\end{proof}

\begin{proposition}
  \label{prop:prismbar-basis}
  Let $(A,(d))$ be a bounded, oriented prism and let $r_1,\ldots,r_c\in A$ be elements which map to a $p$-completely Koszul-regular sequence in $\overline{A}$.
    Letting $R=A/(d,r_1,\ldots,r_c)$, the filtered $\overline{A}$-algebra
    $\F_{\leq\star}^\conj\Prismbar_{R/A}$ is a $p$-completely Koszul-regular quotient of
    the filtered polynomial ring
    \[
      R[\delta^u(a_v)\mid u\geq 0, 1\leq v\leq c]^\wedge_p
    \]
    by relations
    \[
      \delta^u(a_v)^p + p \delta^{u+1}(a_v) - R_{uv},
    \]
    where $\delta^u(a_v)$ has filtration weight $\geq p^u$. In particular,
    the monomials 
    \[
      \prod_{uv}\delta^u(a_v)^{e_{uv}}
    \]
    for $1\leq v\leq c$, $u\geq 0$, $0\leq e_{uv}<p$, and $\sum_{v,u}
    e_{v,u}p^u\leq i$ form a
    $p$-completely free basis for the $R$-module $\F_{\leq i}^\conj\Prismbar_{R/A}\subseteq
    (A\{\tfrac{r_1}{d},\ldots,\tfrac{r_c}{d}\}_p^\wedge)/d$.
\end{proposition}

\begin{proof}
  We first argue that the monomials 
    \[
      \prod_{uv}\delta^u(a_v)^{e_{uv}}
    \]
    form a basis for $\Prismbar_{R/A}$. Indeed, Proposition \ref{prop:envelope_regular}
    exhibits $\Prismbar_{R/A}$ as a Koszul-regular quotient of
    \[
      R[\delta^u(a_v)\mid u\geq 0, 1\leq v\leq c]^\wedge_p
    \]
    by elements
    \[
      \delta^u(a_v)^p + p \delta^{u+1}(a_v) - R_{uv},
    \]
    where here there is no claim about the conjugate filtration.
    We have an ascending filtration $\G_{\leq \star}$ with $\delta^u(a_v)$ in filtration
    $\leq p^u$. On the associated graded ring, we are dealing with the quotient of
    $R[\delta^u(a_v)\mid u\geq 0, 1\leq v\leq c]^\wedge_p$ by relations $\delta^u(a_v)^p
    + p \delta^{u+1}(a_v)$. This is a free divided power algebra, which has the
    indicated basis. So these element form a basis before passage to the associated
    graded ring as well.

    It remains to prove that $\G_{\leq \star}$ and the conjugate filtration
    $\F^\conj_{\leq\star}$ agree. Both sides satisfy base change, so it suffices to
    analyze the universal case of
    \[
      A = \bZ_p\{d,r_1,\ldots,r_c\}[\delta(d)^{-1}]^\wedge_{p,d}.
    \]
    Furthermore, by monoidality we can reduce to the case of a single relation,
    $A=\bZ_p\{d,r\}$ and $R=\overline{A}/r$. We may endow $A$ with the structure of a
    graded prism with $r$ in degree $1$. In that case, $L_{R/\overline{A}} \simeq
    R(1)[1]$ since $R$ is obtained from $\overline{A}$ by quotienting by a regular
    element of degree $1$. So, the Hodge--Tate comparison theorem shows that
    \[
      \gr_j^\conj \Prismbar_{R/A} = \L\Omega^j_{R/\overline{A}}[-j]\simeq R(j).
    \]
    We also have that $a=\frac{r}{d}$ is in degree $1$, and hence the basis element $\prod_u \delta^u(a)^{e_u}$ is in degree $\sum p^ue_u$. It follows that $\gr_j \G$ is also $R(j)$. Both $\G_{\leq j}$ and $\F^{\conj}_{\leq j}$ therefore agree with the degree $\leq j$ part of $\Prismbar_{R/A}$, and in particular agree.
\end{proof}

\begin{remark}
    Below, in Corollary~\ref{cor:f-uv-mapto-a-uv}, we will see another way of seeing that the elements $\delta^i(a)$ must be in
    $\F_{\leq p^i}^\conj\Prismbar_{R/A}$.
\end{remark}

\subsection{The $\delta$-ring structure on $\bigoplus\N^{\geq i}\Prismhat_{R/A}^{(1)}$}

\begin{definition}
    Let $(A,I)$ be a bounded prism and let $R$ be a commutative $\overline{A}$-algebra. The
    Frobenius-twisted prismatic cohomology of $A$, denoted by $\Prism^{(1)}_{R/A}$, is the derived
    $(p,I)$-completion of $\Prism_{R/A}\otimes_{A,\varphi}A$. It is equipped with an $A$-linear
    relative Frobenius map $\Prism_{R/A}^{(1)}\xrightarrow{\varphi_{R/A}}\Prism_{R/A}$
    and there is a $\varphi_A$-semilinear map
    $w\colon\Prism_{R/A}\rightarrow\Prism_{R/A}^{(1)}$. The composition
    $w\circ\varphi_{R/A}$ is the Frobenius endomorphism of $\Prism_{R/A}^{(1)}$ and the
    composition $\varphi_{R/A}\circ w$ is the Frobenius endomorphism of $\Prism_{R/A}$.
\end{definition}

\begin{remark}[The Nygaard filtration]\label{rem:nygaard}
    Let $(A,I)$ be a bounded prism and let $R$ be a commutative $\overline{A}$-algebra. There is a
    decreasing multiplicative
    exhaustive filtration $\N^{\geq\star}\Prism_{R/A}^{(1)}$, the Nygaard filtration, on
    $\Prism_{R/A}^{(1)}$ studied in~\cite{bms2,prisms,bhatt-lurie-apc}.
    The Nygaard filtration has the following properties:
    \begin{enumerate}
        \item[(i)] $\N^{\geq 0}\Prism_{R/A}^{(1)}=\N^{\geq
            -1}\Prism_{R/A}^{(1)}=\cdots=\Prism_{R/A}^{(1)}$, or
            equivalently $\gr^i_\N\Prism_{R/A}^{(1)}=0$ for $i<0$;
        \item[(ii)] the relative Frobenius
            $\varphi_{R/A}\colon\Prism_{R/A}^{(1)}\rightarrow\Prism_{R/A}$ promotes to a natural map
            $\N^{\geq\star}\Prism_{R/A}^{(1)}\rightarrow I^\star\Prism_{R/A}$ of filtered
            $\bE_\infty$-rings;
        \item[(iii)] the map $\gr^i_\N\Prism_{R/A}^{(1)}\rightarrow\Prismbar_{R/A}\{i\}$ induced
            from (ii) yields an equivalence
            $\gr^i_\N\Prism_{R/A}^{(1)}\we\F_{\leq i}^\conj\Prismbar_{R/A}\{i\}$ for all $i\in\bZ$;
        \item[(iv)] if $\L_{R/A}$ has $p$-complete Tor-amplitude in $[1,1]$, then the
            Nygaard-completed Frobenius-twisted prismatic cohomology $\Prismhat_{R/A}^{(1)}$ is discrete, the natural induced relative Frobenius
            map $\varphi_{R/A}\colon\Prismhat_{R/A}^{(1)}\rightarrow\Prism_{R/A}$ is injective, and
            $\N^{\geq i}\Prismhat_{R/A}^{(1)}$ consists of those $x$ such that $\varphi_{R/A}(x)\in
            I^i\Prism_{R/A}$. (For the last part, in this generality, see~\cite[Cor.~6.13]{akn-delta}.)
    \end{enumerate}
\end{remark}

\begin{remark}
    Either property (iii) or property (iv) of Definition~\ref{rem:nygaard} can be used to define the
    filtration general~\cite[Thm.~15.3]{prisms}. If $R$ is $p$-completely smooth over $\overline{A}$, then one lets
    $\N^{\geq\star}\Prism^{(1)}_{R/A}$ be $\L\eta_I\Prism_{R/A}$, where $\L\eta_I$
    denotes the d\'ecalage functor~\cite[Prop.~5.8]{bms2}, and then left Kan extends this
    filtration to all commutative $\overline{A}$-algebras. On the other hand, if $R$ is relatively quasiregular
    semiperfectoid over $\overline{A}$, then one defines the Nygaard filtration by letting
    $x\in\N^{\geq i}\Prism_{R/A}^{(1)}$ if and only if $\varphi_{R/A}(x)\in I^i\Prism_{R/A}$. Then,
    this can be descended to give a definition for all relatively quasisyntomic
    $\overline{A}$-algebras.
\end{remark}

For the most part, we are interested only in the Nygaard-completion of $\Prism_{R/A}^{(1)}$ when $R$ is a
Koszul-regular quotient of $\overline{A}$.

\begin{lemma}\label{lem:twisted_envelope}
    Let $(A,d)$ be a bounded, oriented prism and let $R$ be an $\overline{A}$-algebra of the form $A/(d,r_1,\ldots,r_c)$, where the images of $r_1,\ldots,r_c$ form a $p$-completely Koszul-regular sequence in $A/d$. By~\cite[Lem.~7.7]{prisms}, there is an equivalence
    \[
    \Prism_{R / A} \simeq A\left\{\frac{r_1}  d, \ldots , \frac{r_c} d\right\}^\wedge_{(p,d)}.
    \]
    The relative Frobenius
    \[
      \varphi_{R/A}\colon \Prism^{(1)}_{R/A}\to \Prism_{R/A}
    \]
    induces an injective map
    \[
      \varphi_{R/A}\colon \Prismhat^{(1)}_{R/A}\to \Prism_{R/A}
    \]
    exhibiting the Nygaard completion as the completed sub-$\delta$-ring generated by
    the elements $\frac{\varphi(r_j)}{\varphi(d)} = \varphi\left(\frac{r_j}{d}\right)$,
    i.e. as
    \[
    A\left\{\frac{\varphi(r_1)}{\varphi(d)},\ldots,\frac{\varphi(r_c)}{\varphi(d)}\right\}^\wedge_{(p,N)} \subseteq 
      A\left\{\frac{r_1}{d}, \ldots , \frac{r_c}{d}\right\}^\wedge_{(p,d)},
    \]
    with Nygaard filtration given by the restriction of the $d$-adic filtration.
\end{lemma}

\begin{proof}
  As $\Prism_{R/A}$ is generated, as a $\delta$-ring over $A$, by elements $a_j =
    \frac{r_j}{d}$, one has that $\H_0(\Prism_{R/A}^{(1)})$ is generated as a $\delta$-ring over $A$
    by their images $a_j\otimes 1$. Under the relative Frobenius, these map to
    $\varphi(a_j) = \frac{\varphi(r_j)}{\varphi(d)}$. The Frobenius twist
    $\Prism_{R/A}^{(1)}$ is not necessarily discrete, but by
    Remark~\ref{rem:nygaard}(iii) and discreteness of the associated graded terms
    $\L\Omega^j_{R/\overline{A}}[-j]$ of the conjugate filtration, the Nygaard associated
    graded terms are discrete and embed into the $d$-adic associated graded terms. It
    follows that the Nygaard completion is discrete and obtained as closure of the image
    of $\Prism_{R/A}^{(1)}$ in $\Prism_{R/A}$ with respect to the $(p,d)$-adic topology
    and that the Nygaard filtration on the Nygaard completion can be obtained by
    restricting the $d$-adic filtration.
\end{proof}

\begin{remark}
  As suggested in Lemma \ref{lem:twisted_envelope}, one may think of $\Prism_{R/A}^{(1)}$ as a prismatic envelope in its own right, adjoining fractions $\frac{\varphi(r_j)}{\varphi(d)}$. Indeed, the pushout diagram
    \eqref{eqn:envelope_pushout} at the beginning of the section base changes to an analogous pushout. The issue is that in general, both $\Prism_{R/A}^{(1)} = \Prism_{R/A}\otimes_{A,\varphi} A$ and this pushout diagram have to be interpreted in a derived way, even if $\Prism_{R/A}$ is discrete. However, in many cases, the $\varphi(r_j)$ also form a $p$-completely Koszul-regular sequence in $\overline{A}$, in which case the prismatic envelope is discrete, admits a description as in the preceding section, and agrees with $\Prism_{R/A}^{(1)}$.
\end{remark}

The problem with this description is that the Nygaard filtration is rather inexplicit.
We would like to find an explicit description (in terms of generators and relations) of the Nygaard filtered pieces $\N^{\geq i} \Prismhat^{(1)}_{R / A}$. 

\begin{example}
    The Nygaard filtration on prismatic envelopes is rarely complete, although it will be separated
    in all cases considered in this paper. For example, neither
    $\N^{\geq\star}\Prism^{(1)}_{(\bZ/p^n)/\bZ_p\llbracket z_0\rrbracket}$ nor
    $\N^{\geq\star}\Prism^{(1)}_{\bZ_p/\bZ_p\llbracket z_0,z_1\rrbracket}$ is complete.
    Indeed, the elements $f_j\in\N^{\geq p^j}\Prismhat^{(1)}_{(\bZ/p^n)/\bZ_p\llbracket
    z_0\rrbracket}$ we construct in
    Proposition~\ref{prop:f_filtered_statement} lift to $\N^{\geq
    p^j}\Prism^{(1)}_{(\bZ/p^n)/\bZ_p\llbracket z_0\rrbracket}$, but a sum such as
    $\sum_{j\geq 0}f_j$ exists in the Nygaard-completion, but does not exist before Nygaard-completion.
\end{example}

\begin{construction}
    Let $(A,d)$ be an oriented prism and let $R$ be a commutative $\overline{A}$-algebra with
    $\L_{R/\overline{A}}$ having $p$-complete Tor-amplitude in $[1,1]$. Consider the graded ring
    $\bigoplus_{i\geq 0}\N^{\geq i}\Prismhat_{R/A}^{(1)}$. Define functions
    $$\varphitilde_{i}\colon\N^{\geq i}\Prismhat_{R/A}^{(1)}\rightarrow\N^{\geq
    pi}\Prismhat_{R/A}^{(1)}\quad\text{by}\quad\varphitilde_{i}(x)=\tfrac{\varphi(x)}{\varphi(d)^i}d^{pi},$$
    $$\deltatilde_{i}\colon\N^{\geq i}\Prismhat_{R/A}^{(1)}\rightarrow\N^{\geq
    pi}\Prismhat_{R/A}^{(1)}\quad\text{by}\quad\deltatilde_{i}(x)=\delta(x)-\delta(d^i)\tfrac{\varphi(x)}{\varphi(d)^i}.$$
    Using part (iv) of Remark~\ref{rem:nygaard},
    we also have the $A$-linear $d$-divided relative Frobenius maps $$\varphi_{i,R/A}\colon\N^{\geq
    i}\Prismhat_{R/A}^{(1)}\rightarrow\Prism_{R/A}\quad\text{defined
    by}\quad\varphi_{i,R/A}(x)=\tfrac{\varphi_{R/A}(x)}{d^i}$$
    and the $\varphi(d)$-divided Frobenius maps $$\varphi_i\colon\N^{\geq
    i}\Prismhat_{R/A}^{(1)}\rightarrow\Prismhat_{R/A}^{(1)}\quad\text{defined
    by}\quad\varphi_i(x)=\tfrac{\varphi(x)}{\varphi(d)^i}.$$
\end{construction}

\begin{notation}
    We will write $\deltatilde=\bigoplus_{i\geq 0}\deltatilde_i$ and $\varphitilde=\bigoplus_{i\geq
    0}\varphitilde_i$ as graded functions on $\bigoplus_{i\geq 0}\N^{\geq
    i}\Prismhat_{R/A}^{(1)}$.
\end{notation}

\begin{proposition}[The graded $\delta$-ring structure]\label{prop:graded_delta}
    Let $(A,d)$ be a bounded oriented prism and let $R$ be a commutative $\overline{A}$-algebra with
    $\L_{R/\overline{A}}$ having $p$-complete Tor-amplitude in $[1,1]$.
    The operation $\deltatilde$ makes $\bigoplus_{\star\geq 0}\N^{\geq \star}\Prismhat_{R/A}^{(1)}$ into a graded
    $\delta$-ring with graded Frobenius morphism $\varphitilde$; the relative divided Frobenius
    map
    $$\bigoplus_{\star\geq
    0}\N^{\geq\star}\Prismhat_{R/A}^{(1)}\xrightarrow{\bigoplus\varphi_{i,R/A}}\Prism_{R/A}$$ is a
    $\delta$-ring morphism.
\end{proposition}

\begin{proof}
    Let $x\in\N^{\geq i}\Prismhat_{R/A}^{(1)}$.
    While $\varphitilde_{i}(x)\in\N^{\geq pi}\Prismhat_{R/A}^{(1)}$ follows because $d^{pi}\in\N^{\geq
    pi}\Prismhat_{R/A}^{(1)}$, it is less obvious that $\deltatilde_{i}$ preserves the gradings in the
    claimed way. Nevertheless,
    $$x^p+p\deltatilde_{i}(x)=x^p+p\left(\delta(x)-\delta(d^i)\tfrac{\varphi(x)}{\varphi(d)^i}\right)=\frac{(x^p+p\delta(x))\varphi(d)^i-p\delta(d^i)\varphi(x)}{\varphi(d)^i}=\tfrac{\varphi(x)}{\varphi(d)^i}d^{pi}=\varphitilde_i(x).$$
    so that $p\deltatilde_{i}(x)$ is in weight $pi$. By reduction to a universal case using base change, we
    can assume that the Nygaard-filtered pieces $\N^{\geq i}\Prismhat_{R/A}^{(1)}$ are $p$-torsion
    free and hence that $\deltatilde_{i}(x)$ is in weight $pi$, as desired.
    As $\varphitilde_{i}$ is a graded ring endomorphism by definition, it also follows
    that it is a lift of Frobenius with associated $\delta$-ring structure given by
    $\deltatilde$.

    Again, using reduction to a $p$-torsion free case, to check that the relative
    divided Frobenius maps $\varphi_{i,R/A}$ assemble into a $\delta$-ring morphism, it
    is enough to check that they commute with the Frobenius endomorphisms.
    But, ~e have
    $$\varphi_{pi,R/A}(\varphitilde_{i}(x))=\varphi_{R/A}\left(\tfrac{\varphi(x)}{\varphi(d)^i}d^{pi}\right)\tfrac{1}{d^{pi}}=\tfrac{\varphi_{R/A}(\varphi(x))}{\varphi(d)^i}=\varphi\left(\tfrac{\varphi_{R/A}(x)}{d^i}\right)=\varphi(\varphi_{i,R/A}(x)),$$
    using $A$-linearity and $\varphi(d)$-torsion freeness of $\Prism_{R/A}$ to commute $d^{pi}$ and $\varphi(d)^i$ through $\varphi_{R/A}$
    and using that $\varphi_{R/A}$ is a map of $\delta$-rings to say that
    $\varphi(\varphi_{R/A}(x))=\varphi_{R/A}(\varphi(x))$.
\end{proof}

\begin{warning}\label{warning:not_filtered}
    In the oriented case, the $\delta$-ring structure $\deltatilde$ on $\bigoplus_{\star_{\geq
    0}}\N^{\geq\star}\Prismhat^{(1)}_{R/A}$ does not descend to a filtered $\delta$-ring structure on
    $\N^{\geq\star}\Prismhat_{R/A}^{(1)}$. Indeed, if it did, in the case of
    $\Prismhat_{\bF_p/\bZ_p}$, one
    would obtain on $\gr^0_\N\Prismhat_{\bF_p/\bZ_p}\we\bF_p$ a $\delta$-ring structure, which is
    absurd. This is explained by the fact that $\deltatilde$ does not commute with the transition maps $\N^{\geq
    i+1}\Prismhat_{R/A}^{(1)}\rightarrow\N^{\geq i}\Prismhat_{R/A}^{(1)}$.
\end{warning}

\subsection{Generators for the Nygaard filtration}\label{sec:gen_nygaard}

We are now ready to define the generators for the Nygaard filtration.

\begin{construction}[The $f_{u,v}$ generators]\label{const:fuv}
    Fix a bounded, oriented prism $(A,d)$ and a sequence of elements $r_1,\ldots,r_c\in A$ that
    define a $p$-completely Koszul-regular sequence in $\overline{A}$; let $R=A/(d,r_1,\ldots,r_c)$.
    We let $\widetilde{r}_v$ denote the image of $r_v$ in $\N^{\geq 1}\Prismhat_{R/A}^{(1)}$
    under
    \[
      \N^{\geq 1}\Prismhat_{R/A}^{(1)}\subseteq \Prismhat_{R/A}^{(1)}\cong
      (\Prism_{R/A}\otimes_{A,\varphi} A)^{\wedge}_{p,\N}.
    \]
    For $1\leq v\leq c$ and $u\geq 0$, we define $f_{u,v}\in \N^{\geq p^u}\Prismhat^{(1)}_{R/A}$ by
    \[
        f_{u,v}=\deltatilde^u(r_v).
    \]
    Additionally, we let 
    \[
      a_v=\varphi_{1,R/A}(f_{0,v})\in\Prism_{R/A}.
    \]
\end{construction}

We first remark that there is no conflict in notation between the $a_v$ from
Construction~\ref{const:fuv} and those from Proposition~\ref{prop:envelope_regular}.

\begin{lemma}
    If $(A,d)$ is a bounded, oriented prism and $r_1,\ldots,r_c$ is a sequence of elements whose
    image in $\overline{A}$ is $p$-completely Koszul-regular, then $a_v\in\Prism_{R/A}$ satisfies
    $da_v=r_v$ for $1\leq v\leq c$. Moreover, the induced map
    $$A\{\tfrac{r_1}{d},\ldots,\tfrac{r_c}{d}\}_{(p,d)}^\wedge\rightarrow\Prism_{R/A}$$ is an
    equivalence.
\end{lemma}

\begin{proof}
    By definition, $a_v=\varphi_{1,R/A}(f_{0,v})=\tfrac{\varphi_{R/A}(\widetilde{r}_v)}{d}$, so
    $da_v=\varphi_{R/A}(\widetilde{r}_v)$. But, since $\varphi_{R/A}$ is $A$-linear,
    $\varphi_{R/A}(\widetilde{r}_v)=r_v$, so we have that $da_v=r_v$.
    Now, there is an induced map
    $$A\{\tfrac{r_1}{d},\ldots,\tfrac{r_c}{d}\}_{(p,d)}^\wedge\rightarrow\Prism_{R/A}$$ by
    definition of the prismatic envelope. But, this map is the map arising in the proof of the fact that the prismatic envelope of an
    lci quotient computes the prismatic cohomology in~\cite[Ex.~7.9]{prisms}; in particular, it is
    an equivalence.
\end{proof}

\begin{corollary}
  \label{cor:f-uv-mapto-a-uv}
    If $(A,d)$ is a bounded, oriented prism and $r_1,\ldots,r_c$ is a sequence of elements whose
    image in $\overline{A}$ is $p$-completely Koszul-regular, then for every $1\leq v\leq c$ and
    every $u\geq 0$, one has $\varphi_{p^u,R/A}(f_{u,v})=\delta^u(a_v)$ and the image of $\delta^ua_v$ in $\Prismbar_{R/A}$ is in $\F^\conj_{\leq
    p^u}\Prismbar_{R/A}$.
\end{corollary}

\begin{proof}
    The first statement follows by the compatibility
    of $\deltatilde$ and $\delta$ under $\varphi_{\star, R/A}$ established
    in Proposition~\ref{prop:graded_delta}. The second statement follows from the fact
    that the relative Frobenius induces an 
    equivalence $\gr^{p^u}_\N\Prismhat^{(1)}_{R/A}\we\F_{p^u}^\conj\Prismbar_{R/A}$ (see
    for example~\cite[Rem.~5.1.2]{bhatt-lurie-apc}).
\end{proof}

\begin{lemma}[Relations for the $f_{u,v}$]\label{lem:explicitrelations}
  Let $\widetilde{d}\in \N^{\geq 1}\Prismhat^{(1)}_{R/A}$ be $d$, regarded as element of
    Nygaard filtration $\geq 1$. The elements $f_{u,v}$ satisfy relations
  \[
    f_{u,v}^p = (-p + \lambda_u d^{p^{u+1}}) f_{u+1,v} + \widetilde{d}^{p^{u+1}} R_{u,v}',
  \]
  in $\N^{\geq p^u} \Prismhat^{(1)}_{R/A}$,
  where $\lambda_u$ is given by $\lambda_0 = -\frac{1}{\delta(d)}$ and
  \[
    \lambda_{u+1} = \frac{\lambda_u^p}{1-\delta(d^{p^{u+1}}\lambda_u)}
  \]
  and the $R_{u,v}'\in \N^{\geq 0}\Prismhat^{(1)}_{R/A}$ are given by $R_{0,v}' = \frac{\delta(r_v)}{\delta(d)}$ and
  \[
      R_{u+1,v}' = \frac{1}{1-\delta(d^{p^{u+1}}\lambda_u)} (\delta(R_{u,v}') +
      w_1(\lambda_u f_{u+1,v}, R_{u,v}')).
  \]
\end{lemma}

\begin{remark}
    Note that in the statement of Lemma~\ref{lem:explicitrelations}, $\delta$ does not preserve the Nygaard filtration in any way.
\end{remark}

\begin{proof}[Proof of Lemma~\ref{lem:explicitrelations}]
    Applying Lemma \ref{lem:relations} to the graded $\delta$-ring $\bigoplus_{\star\geq
    0}\N^{\geq \star}\Prismhat^{(1)}_{R/A}$ and the relation $df_{0,v} = \widetilde{d}r$, we obtain relations
  \[
    f_{u,v}^p = (-p + \lambda_u d^{p^{u+1}}) f_{u+1,v} + R_{u,v},
  \]
    using $\deltatilde(\widetilde{d})=0$,
  with $R_{0,v} = \frac{\deltatilde(\widetilde{d} r_v)}{\delta(d)} = \widetilde{d}^p R_{0,v}'$.
    Now $R_{u,v}$ satisfies the recursion
  \[
    R_{u+1,v} = \frac{1}{1-\delta(d^{p^{u+1}\lambda_u})} (\deltatilde(R_{u,v}) + w_1(d^{p^{u+1}}\lambda_u f_{u+1,v}, R_{u,v}).
  \]
  We may rewrite $d^{p^{u+1}}f_{u+1,v}$ as $\widetilde{d}^{p^{u+1}} \cdot f_{u+1,v}$ in
    $\N^{\geq p^{u+1}}\Prismhat_{R/A}^{(1)}$. It inductively follows that
  \[
    R_{u+1,v} = \widetilde{d}^{p^{u+1}} R_{u+1,v}',
  \]
  proving the claimed relation.
\end{proof}

\begin{lemma}[Structure maps on the $f_{u,v}$]
  \label{lem:f_uv_structure}
  The divided Frobenius and $\delta$-structure maps are determined on the elements $f_{u,v}$ by
  \begin{enumerate}
      \item[{\em (1)}] $\varphi_{p^u}(f_{u,v}) = \lambda_u f_{u+1,v} + R_{u,v}'$,
      \item[{\em (2)}] $\varphi(f_{u,v}) = \varphi(d)^{p^u} \cdot (\lambda_u f_{u+1,v} +
          R_{u,v}')$, and
      \item[{\em (3)}] $\delta(f_{u,v}) =(1+\delta(d^{p^u})\lambda_u)f_{u+1,v} +  \delta(d^{p^u}) R_{u,v}'$
  \end{enumerate}
\end{lemma}
\begin{proof}
  We have
  \[
    \varphitilde(f_{u,v}) = f_{u,v}^p + pf_{u+1,v} = \lambda_u d^{p^{u+1}} f_{u+1,v} + \widetilde{d}^{p^{u+1}} R_{u,v}'. 
  \]
    Using $\varphitilde(f_{u,v}) = \widetilde{d}^{p^{u+1}}\varphi_{p^u}(f_{u,v})$ and rewriting $d^{p^{u+1}}f_{u+1,v} = \widetilde{d}^{p^{u+1}} f_{u+1,v}$, this implies
  \[
    \varphi_{p^u}(f_{u,v}) = \lambda_u f_{u+1,v} + R_{u,v}',
  \]
  since $\widetilde{d}$ acts as nonzerodivisor on the Nygaard filtration: indeed, on the
    associated graded pieces of the Nygaard filtration, $\widetilde{d}$ fits into a commutative diagram
  \[
    \begin{tikzcd}
      \gr^{j-1}_\N\Prismhat^{(1)}_{R/A} \dar{\simeq}\rar{\widetilde{d}}& \gr^{j}_\N\Prismhat^{(1)}_{R/A} \dar{\simeq}\\
      \F^{\conj}_{\leq j-1} \Prismbar_{R/A} \rar{\can} & 
      \F^{\conj}_{\leq j} \Prismbar_{R/A},
    \end{tikzcd}
  \]
    where the last map has discrete cofiber $\widehat{\L\Omega}^j_{R/\overline{A}}[-j]$ by assumption.
  Since $\varphi(f_{u,v}) = \varphi(d)^{p^u} \varphi_{p^u}(f_{u,v})$, the second claim follows.

  Finally, observe that
  \[
    f_{u,v}^p = (-p + \lambda_u d^{p^{u+1}})f_{u+1,v} + d^{p^{u+1}}R'_{u,v},
  \]
  and, since $\varphi(d^{p^u}) = d^{p^{u+1}} + p \delta(d^{p^u})$,
  \[
    \varphi(f_{u,v}) = (d^{p^{u+1}}+p\delta(d^{p^u})) \cdot (\lambda_u f_{u+1,v} + R_{u,v}').
  \]
taking the difference,
  \[
      \varphi(f_{u,v})-f_{u,v}^p=p\delta(f_{u,v}) = p(1+\delta(d^{p^u})\lambda_u)f_{u+1,v} + p \delta(d^{p^u}) R_{u,v}',
  \]
  which we may divide by $p$ to prove the last claim, since it suffices to prove these relations in the universal case, where everything is $p$-torsion free.
\end{proof}

\begin{remark}[Translating between the $f_{u,v}$ and $\delta^u(a_v)$]
  Since the relative divided Frobenius satisfies $\varphi_{p^u,R/A}(f_{u,v}) =
    \delta^u(a_v)$, part (1) of Lemma \ref{lem:f_uv_structure} implies that
  \[
    w(\delta^u(a_v)) = \lambda_u f_{u+1,v} + R_{u,v}',
  \]
  expressing the image of our generators of $\Prism_{R/A}$ in the Frobenius twist
    $\Prismhat^{(1)}_{R/A}$ in terms of the $f_{u,v}$. As the leading coefficient here
    is a unit, this in fact provides a way to translate between the $f_{u,v}$ and the
    $\delta^u(a_v)$. One may use this to obtain a generators-and-relations description
    of $\Prism^{(1)}_{R/A}$ in terms of the $f_{u,v}$ and the above relations. We prove
    this below only for the Nygaard completion $\N^{\geq\star}\Prismhat^{(1)}_{R/A}$ in
    a different way and will not need to directly translate between these two sets of generators.
\end{remark}

We now want to express the fact that, as a filtered ring, $\N^{\geq\star}\Prismhat^{(1)}_{R/A}$ is generated by the $f_{u,v}$ subject to the relations from Lemma \ref{lem:explicitrelations}. The problem is that the $R_{u,v}'$ are only defined as elements of $\N^{\geq\star}\Prismhat^{(1)}_{R/A}$ rather than abstract polynomials in the $f_{u,v}$, and since the recursion involves the $\delta$-ring structure on $\Prismhat^{(1)}_{R/A}$, it is not immediately obvious how to lift them to polynomials. The following construction accomplishes that.

\begin{construction}
  \label{cons:Ru_recursion}
  On the polynomial ring $A[f_{u,v} \mid u\geq 0, 1\leq v\leq c]^\wedge_{p,d}$, we introduce a $\delta$-ring structure and elements $R'_{u,v}$ through $R_{0,v}' = \frac{\delta(r_v)}{\delta(d)}$ and the mutual recursions
  \begin{gather*}
    \delta(f_{u,v}) = (1+\delta(d^{p^u})\lambda_u) f_{u+1,v} + \delta(d^{p^u}) R_{u,v}'\\
    R_{u+1,v}' = \frac{1}{1-\delta(d^{p^{u+1}}\lambda_u)} (\delta(R_{u,v}') +w_1(\lambda_u f_{u+1,v}, R_{u,v}')).
  \end{gather*}
  This works since inductively, $R_{u,v}'$ is a polynomial in $f_{0,v},\ldots, f_{u,v}$, and $\delta(f_{u,v})$ is a polynomial in $f_{0,v},\ldots,f_{u+1,v}$.

  Also note that with respect to the ascending filtration $\G_{\leq \star}
    A[f_{u,v}]^\wedge_{p,d}$ where $f_{u,v}$ has filtration $\leq p^u$, one inductively
    sees that $R_{u,v}\in \G_{<p^{u+1}}A[f_{u,v}]^\wedge_{p,d}$ and $\delta
    \left(\G_{\leq\star} A[f_{u,v}]^\wedge_{p,d}\right)\subseteq \G_{\leq p\star}
    A[f_{u,v}]^\wedge_{p,d}$. In particular, $R_{u,v}'$ involves $f_{u,v}$ with
    exponents strictly smaller than $p$.

  Finally, note that the map $A[f_{u,v}]^\wedge_{p,d}\to \Prismhat^{(1)}_{R/A}$ is a map of $\delta$-rings and takes $R_{u,v}'$ to the element of the same name.
\end{construction}

\begin{remark}
    \label{rem:vanishingofRu_prime}
    If $\delta(r)=0$, observe that the $R_u'$ are all inductively zero.
\end{remark}

\begin{proposition}
  \label{prop:nygaard_graded_statement}
    If $(A,d)$ is a bounded, oriented prism and $r_1,\ldots,r_c$ is a sequence of elements whose
    image in $\overline{A}$ is $p$-completely Koszul-regular,
    then, as a graded commutative $R$-algebra, $\bigoplus_{\star\geq
    0}\gr^\star_\N\Prismhat_{R/A}^{(1)}$
    is a $p$-completely Koszul-regular quotient of the $p$-complete graded polynomial ring
    $R[\widetilde{d},f_{u,v}\mid u\geq 0,1\leq
    v\leq c]_p^\wedge$, where $\widetilde{d}$ is in weight $1$ and corresponds to
    $d\in\gr^1_\N\Prismhat_{R/A}^{(1)}$ and $f_{u,v}$ is in weight $p^u$, by the relations
  \[
    f_{u,v}^p = (-p + \lambda_u d^{p^{u+1}}) f_{u+1,v} + \widetilde{d}^{p^{u+1}} R_{u,v}',
  \]
    As a graded $R$-module, $\bigoplus_{\star\geq
    0}\gr^\star_\N\Prismhat_{R/A}^{(1)}$ is $p$-completely free on monomials $\widetilde{d}^k\prod
    f_{u,v}^{e_{u,v}} \in \gr_\N^{k+\sum p^u e_{u,v}}\Prismhat_{R/A}$ with $e_{u,v}<p$.
\end{proposition}
\begin{proof}
  Recall that we have commutative diagrams
  \[
    \begin{tikzcd}
      \gr^{j-1}_\N\Prismhat^{(1)}_{R/A} \dar{\simeq}\rar{\widetilde{d}}& \gr^{j}_\N\Prismhat^{(1)}_{R/A} \dar{\simeq}\\
      \F^{\conj}_{\leq j-1} \Prismbar_{R/A} \rar{\can} & 
      \F^{\conj}_{\leq j} \Prismbar_{R/A},
    \end{tikzcd}
  \]
  so $\widetilde{d}$ acts injectively on $\gr^\star_{\N}\Prismhat^{(1)}_{R/A}$.
    The relations hold by Lemma~\ref{lem:explicitrelations}.
    To check that they form a $p$-completely regular sequence, it suffices to know that
    they do so modulo $\widetilde{d}$. This reduces to the claim that the elements
    $f_{u,v}^p + pf_{u+1,v}$ for $u\geq 0$ and $1\leq v\leq c$ form a $p$-completely
    regular sequence in $R[f_{u,v}\mid u\geq 0, 1\leq v\leq c]^\wedge_p$, which is
    clear. This also shows that the monomials $\widetilde{d}^k \prod f_{u,v}^{e_{u,v}}$
    form a basis of the quotient of $R[\widetilde{d},f_{u,v}\mid u\geq 0,1\leq v\leq
    c]^\wedge_p$ by the indicated relations. The divided relative Frobenius
    $\varphi_{i,R/A}\colon \gr^i_\N
    \Prismhat^{(1)}_{R/A} \xto{\simeq} \F^\conj_{\leq i}\Prismbar_{R/A}$ takes the
    $\widetilde{d}^{i-\sum p^u e_{u,v}}\prod f_{u,v}^{e_{u,v}}$ to $\prod
    \delta^u(a_v)^{e_{u,v}}$, so $\gr^\star_\N\Prismhat^{(1)}_{R/A}$ is also free on the
    same set of monomials.
\end{proof}

Recall that in our convention is that in $\N^{\geq\star}\Prismhat_{R/A}^{(1)}$ the
element $\tau$ denotes $1$, but viewed as being in $\N^{\geq -1}\Prismhat_{R/A}^{(1)}$.
In the corollary below, $\tau f_{0,v}$ is then in $\N^{\geq 0}\Prismhat_{R/A}^{(1)}$.

\begin{corollary}\label{cor:nygaard_filtered_statement}
    If $(A,d)$ is a bounded, oriented prism and $r_1,\ldots,r_c$ is a sequence of elements whose
    image in $\overline{A}$ is $p$-completely Koszul-regular with quotient $R$,
    then, as a complete filtered commutative $d^\star
    A\iso\N^{\geq\star}\Prismhat_{\overline{A}/A}^{(1)}$-algebra, $\N^{\geq\star}\Prismhat_{R/A}^{(1)}$
    is a $p$-completely Koszul-regular quotient of the $p$-complete complete filtered
    polynomial ring
    $d^\star A[f_{u,v}\mid  u\geq 0,1\leq
    v\leq c]_{p,\N}^\wedge$, where $f_{u,v}$ is in weight $p^u$, by relations
    $\tau f_{0,v}-r_v$ of weight $0$ and the relations
  \[
    f_{u,v}^p = (-p + \lambda_u d^{p^{u+1}}) f_{u+1,v} + \widetilde{d}^{p^{u+1}} R_{u,v}',
  \]
    $\delta_{d,p^{u-1}}\cdots\delta_{d,1}(df_{0,v}-\widetilde{d}r_v)$ for $u\geq 1$ of weight $p^u$.
\end{corollary}
\begin{proof}
    There is a filtered map $d^\star A[f_{uv}\mid  u\geq 0,1\leq v\leq
    c]_{p,\N}^\wedge\rightarrow\N^{\geq\star}\Prismhat_{R/A}^{(1)}$ taking the elements $f_{uv}$ to
    the corresponding elements of $\N^{\geq\star}\Prismhat_{R/A}^{(1)}$ defined in
    Construction~\ref{const:fuv}.
    By construction, $\tau f_{0v}$ maps to $r_v$ in $\N^{\geq \star}\Prismhat_{R/A}^{(1)}$, so it
    follows that the map factors through the quotient by the indicated relations. We may
    check whether that map is an equivalence after passing to the associated graded ring, which corresponds to reducing mod $\tau$. This turns the $\tau f_{0v}-r_v$ into $-r_v$, which reduces the claim to 
  Proposition \ref{prop:nygaard_graded_statement}.
\end{proof}

\begin{remark}
    Corollary~\ref{cor:nygaard_filtered_statement} is almost expressing $\bigoplus
    \N^{\geq \star}\Prismhat^{(1)}_{R/A}$ as the prismatic envelope $A[\widetilde{d}]\{\frac{\widetilde{d}r_v}{d}\mid 1\leq v\leq c\}$.
    However, this is just a graded and not a filtered object. One might want to take a
    prismatic envelope in filtered objects, but the problem is that $\deltatilde$ is not
    compatible with $\tau$. (See Warning~\ref{warning:not_filtered}.) One valid description along those lines is
  \[
    \bigoplus_{i\in \bZ} \N^{\geq i} \Prismhat^{(1)}_{R/A} \cong
    \frac{(A[\widetilde{d}]\{\frac{\widetilde{d}r_v}{d}\})[\tau]}{(\tau\widetilde{d} -d,
    \tau\frac{\widetilde{d}r_v}{d}-r_v)},
  \]
  which however is not a derived quotient, as the relations $\tau \widetilde{d}-d$ and
    $\tau\frac{\widetilde{d}r_v}{d}-r_v$ are redundant with $d\cdot
    \frac{\widetilde{d}r_v}{d} - \widetilde{d}r_v$.
\end{remark}

\subsection{The $\F$-filtration on prismatic envelopes}\label{sec:f-filtration}

We now specialize to the following situation. We take as our prism $A$ a multivariable Breuil--Kisin
prism $A=W(k)\llbracket z_0,\ldots,z_s\rrbracket$ with Eisenstein polynomial $E(z_0)$. We consider
this as a filtered $\delta$-ring with $z_0,\ldots,z_s$ in filtration weight $1$. We consider
$R=\Oscr_K$ or $\Oscr_K/\varpi^n$ as a filtered commutative $A$-algebra by letting $z_0,\ldots,z_s$
map to $\varpi$. As explained in~\cite{akn-delta}, in this case the prismatic cohomology
$\Prism_{R/A}$ obtains a secondary filtrations compatible with all structure, which we will denote
everywhere by $\F^\star$. In particular, we have the $\F$-filtrations
$\F^\star\Prism_{R/A}$ and
$\F^\star\N^{[a,b]}\Prismhat_{R/A}^{(1)}$.

\begin{proposition}\label{prop:f_filtered_statement_for_prism}
  \leavevmode
  Let $\Oscr_K$ be a discrete valuation ring with residue field $k$ perfect of characteristic $p$
    and with uniformizer $\varpi$. Let $A=W(k)\llbracket z_0,\ldots,z_s\rrbracket$ be the filtered
    $\delta$-ring defined by letting each of $z_0,\ldots,z_s$ have weight $1$ and letting
    $\delta(z_0)=\cdots=\delta(z_s)=0$. We view $\Oscr_K$ and $\Oscr_K/\varpi^n$ as $A$-algebras by
    letting $z_0,\ldots,z_s$ map to $\varpi$.
  \begin{enumerate}
      \item[{\em (1)}] For $R=\Oscr_K$, we have elements
          $\delta^u(b_1),\ldots,\delta^u(b_s)\in\F^{\geq p^u}\Prism_{R/A}$ for $u\geq 0$ corresponding to the
          relations $z_1-z_0,\ldots,z_s-z_0$. A $W(k)$-basis for $\F^{[a,b]}\Prism_{R/A}$
          is given by the monomials $$z_0^k\prod_{u\geq 0,v\in[1,s]} \delta^u(b_v)^{e_{uv}}$$ with $0\leq e_{uv}<p$ for all $u,v$
          and $k+\sum p^ue_{uv}\in [a,b]$.
      \item[{\em (2)}] For $R = \Oscr_K/\varpi^n$, in addition to
          $\delta^u(b_1),\cdots,\delta^u(b_s)$ above,
          there are elements $\delta^u(a)\in\F^{\geq np^u}\Prism_{R/A}$ for $u\geq
          0$ corresponding to the relation $z^n=0$.
          A $W(k)$-basis for $\F^{[a,b]}\Prism_{R/A}$ is given by the elements
          $$z_0^k\prod_{u\geq 0}\delta^u(a)^{e_u}\prod_{u\geq 0,v\in[1,s]}\delta^u(b_v)^{e_{uv}}$$
          with (i) $0\leq e_u<p$ for all $u\geq 0$, (ii) $0\leq e_{uv}<p$ for all $u,v$,
          (iii) $k<n$, and (iv) $k+n\sum_{u\geq 0}p^ue_{u}+\sum_{v\in[1,s]} p^u e_{uv}\in [a,b]$.
  \end{enumerate}    
\end{proposition}

\begin{proof}
    We give the proof of (1). In this case, Proposition~\ref{prop:envelope_regular} presents
    $\Prism_{R/W(k)\llbracket z_0,\ldots,z_s\rrbracket}$ as the prismatic envelope
    $$W(k)\llbracket
    z_0,\ldots,z_s\rrbracket\left\{\frac{z_1-z_0}{E(z_0)},\ldots,\frac{z_s-z_0}{E(z_0)}\right\}_{p,E(z_0)}^\wedge.$$
    We let $b_v=\tfrac{z_v-z_0}{E(z_0)}$, which has weight $1$ and hence the
    $\delta$-iterates $\delta^u(b_v)$ have weight $p^u$. By
    Proposition~\ref{prop:prismbar-basis}, the monomials with $k=0$ give
    $p$-complete generators of
    $\Prismbar_{R/W(k)\llbracket z_0,\ldots,z_s\rrbracket}$ as a module over $R\iso W(k)\llbracket z_0\rrbracket/E(z_0)$. It follows that the
    monomials with $k=0$ give $(p,E(z_0))$-complete generators of the prismatic
    cohomology as $W(k)\llbracket z_0\rrbracket$-module. Thus, the monomials where $k$
    is now allowed to vary give an $\F$-complete basis for the prismatic envelope as a
    $p$-complete $W(k)$-module. It remains to argue that each monomial is in the given
    weight and not in a possibly higher weight. Equivalently, we must check that the
    given elements form a basis for the $\F$-associated graded ring, which follows
    from~\cite[Prop.~10.41]{akn-delta}. The proof of (2) proceeds analogously.
\end{proof}

\begin{proposition}\label{prop:f_filtered_statement}
  \leavevmode
  Let $\Oscr_K$ be a discrete valuation ring with residue field $k$ perfect of characteristic $p$
    and with uniformizer $\varpi$. Let $A=W(k)\llbracket z_0,\ldots,z_s\rrbracket$ be the filtered
    $\delta$-ring defined by letting each of $z_0,\ldots,z_s$ have weight $1$ and letting
    $\delta(z_0)=\cdots=\delta(z_s)=0$. We view $\Oscr_K$ and $\Oscr_K/\varpi^n$ as $A$-algebras by
    letting $z_0,\ldots,z_s$ map to $\varpi$.
  \begin{enumerate}
      \item[{\em (1)}] For $R=\Oscr_K$, we have elements $g_{u,1},\ldots,g_{u,s}\in\F^{\geq p^u}\N^{\geq
          p^u}\Prismhat_{R/A}^{(1)}$ for $u\geq 0$ corresponding to the
          relations $z_1-z_0,\ldots,z_s-z_0$. A $W(k)$-basis for $\F^{[a,b]}\N^{\geq i}
          \Prismhat_{R/A}^{(1)}$ is given by the monomials $$\widetilde{d}^j
          z_0^k\prod_{u\geq 0,v\in[1,s]} g_{u,v}^{e_{u,v}}$$ with $0\leq e_{u,v}<p$ for all $u,v$,
          $j+\sum p^{u}e_{u,v}= i$ or $j=0$ and $\sum p^ue_{u,v}\geq i$, and $k+\sum p^ue_{u,v}\in [a,b]$.
      \item[{\em (2)}] For $R = \Oscr_K/\varpi^n$, in addition to $g_{u,1},\ldots,g_{u,s}$ above,
          there are elements $f_{u}\in\F^{\geq np^u}\N^{\geq p^u}\Prismhat_{R/A}^{(1)}$ for $u\geq
          0$ corresponding to the relation $z^n=0$.
          A $W(k)$-basis for $\F^{[a,b]}\N^{\geq i}
          \Prismhat_{R/A}^{(1)}$ is given by the elements
          $$\widetilde{d}^j
          z_0^k\prod_{u\geq 0}f_u^{e_u}\prod_{u\geq 0,v\in[1,s]} g_{u,v}^{e_{u,v}}$$
          such that (i) $0\leq e_u<p$ for all $u\geq 0$,
          (ii) $0\leq e_{u,v}<p$ for all $u\geq 0$ and $1\leq v\leq s$
          (iii) $k<n$, (iv) $j+\sum p^ue_u+\sum p^ue_{u,v}=i$ or $j=0$ and $\sum p^u e_u+\sum p^{u}e_{u,v}\geq i$,
          and (v) $k+n\sum_{u\geq 0}p^ue_{u}+\sum_{v\in[1,s]} p^u e_{u,v}\in [a,b]$.
  \end{enumerate}    
\end{proposition}

\begin{proof}
  We prove the second statement, the first statement is easier and completely analogous. It suffices to
    check that those elements form a basis for the $\F$-associated graded, i.e. by crystalline degeneration we may reduce to the case of $R = k[z]/z^n$. Now
    Proposition~\ref{prop:nygaard_graded_statement}
    tells us that the elements $\widetilde{d}^j \prod f_{u,v}^{e_{u,v}}$ with $e_{u,v}<p$ and $j+\sum p^u e_{u,v}\geq i$ form a $k[z]/z^n$-basis for the Nygaard-associated graded $\bigoplus_{i'\geq
    i}\gr^{i'}_\N\Prismhat_{R/A}^{(1)}$. So, the elements $\widetilde{d}^j z^k
    \prod f_{u,v}^{e_{u,v}}$ with $k<n$, $e_{u,v}<p$, and $j+\sum p^ue_{u,v}\geq i$ form a $k$-basis, and hence
    the elements $\widetilde{d}^j z^k \prod f_{u,v}^{e_{u,v}}$ with $k<n$, $e_{u,v}<p$ and $j+\sum p^u e_{u,v}=i$ or $j=0$ and $\sum p^ue_{u,v}\geq i$ form a $k[\widetilde{d}]$-basis of $\bigoplus_{i'\geq i}\gr^{i'}_\N\Prismhat_{R/A}^{(1)}$. Analyzing their filtration, those elements with $k+n\sum p^u e_u + n\sum p^u e_{u,v} \in [a,b]$ form a $k[\widetilde{d}]$-basis for $\F^{[a,b]}\bigoplus_{i'\geq i}\gr^{i'}_\N\Prismhat_{R/A}^{(1)}$.  So they form a $W(k)$-basis of $\F^{[a,b]}\N^{\geq i}\Prismhat^{(1)}_{R/A}$ before passing to the associated graded.
\end{proof}

\begin{remark}
We can summarize the interaction between the generators given in
Proposition~\ref{prop:f_filtered_statement} and the Nygaard and $\F$-filtrations on $\Prismhat^{(1)}_{R/W(k)\llbracket z_0\rrbracket}$ as follows.
    The bigraded is concentrated below a staircase of slope $n$, with $\widetilde{d}^j
    z^k \prod f_{u}^{e_{u}}$ of Nygaard-filtration $j+\sum p^u e_u$ and $\F$-filtration $k+n\sum p^u e_u$. The situation is summarized in Figure~\ref{fig:slope}. Note that in the $\F$-associated graded, the copies of $k$ generated by $\widetilde{d}^j z^k\prod f_{u}^{e_u}$ for varying $j$ form a copy of $W(k)$ generated by $z^k \prod f_u^{e_u}$.
\end{remark}

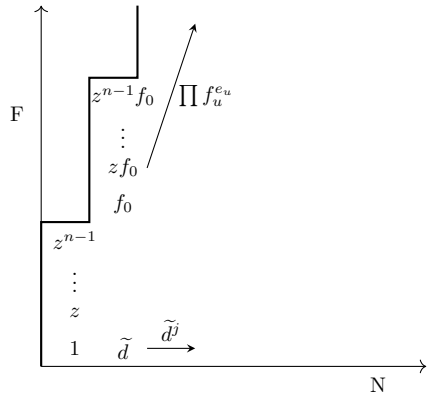
\begin{figure}[h]
    \centering
    \begin{tikzpicture}[scale=0.8, every node/.style={transform shape}, x=0.8cm,y=0.6cm]

    \begin{scope}[shift={(-0.7,-0.5)}]
    \draw[->](0,0) -- (8,0);
    \draw[->](0,0) -- (0,10);
    \node at (7,-0.5) {$\N$};
    \node at (-0.5,7) {$\F$};
    \end{scope}
    \node at (0,0) {$1$};
    \node at (0,1) {$z$};
    \node at (0,2) {$\vdots$};
    \node at (0,3) {$z^{n-1}$};
    \node at (1,4) {$f_0$};
    \node at (1,5) {$zf_0$};
    \node at (1,6) {$\vdots$};
    \node at (1,7) {$z^{n-1}f_0$};
      \draw[-stealth] (1.5,5) -- node[right] {$\prod f_u^{e_u}$} (2.5,9);
      \node at (1,0) {$\widetilde{d}$};
      \draw[-stealth] (1.5,0) -- node[above] {$\widetilde{d}^j$} (2.5,0);
      \draw[thick, shift={(-0.7,-0.5)}] (0,0) -- (0,4) -- (1,4) -- (1,8) -- (2,8) -- (2,10);
    \end{tikzpicture}
    \caption{Note that all monomials $\prod f_u^{e_u}$ lie on the line of slope $n$ through the origin, and everything is concentrated below the staircase depicted bold in the figure. It follows that the
        filtration induced by $\F^{\geq\star}$ on each $\N^{[a,b]}\Prismhat_{R/A}^{(1)}$
        is bounded above, and that the filtration induced by the Nygaard filtration on each
        $\F^{[a,b]}\Prismhat_{R/A}^{(1)}$ is complete (by $d$-completeness). See
        Section~\ref{sec:prismatic} for more on the interaction of the
        $\F$-filtration and the Nygaard filtration.}
    \label{fig:slope}
\end{figure}

\section{Relative-to-absolute descent}\label{sec:bkdescent}

In this section, we discuss in detail the structure of descent from the multivariable
Breuil--Kisin prismatic cohomology of $\Oscr_K/\varpi^n$ to its absolute prismatic cohomology.
Recall the augmented cosimplicial $\delta$-ring $A^\bullet$
\[
    W(k)\rightarrow W(k)\llbracket z_0\rrbracket  \stack{3} W(k)\llbracket z_0,z_1\rrbracket  \stack{5} \ldots,
\]
which is the completion of the Amitsur complex of
$W(k)\rightarrow W(k)\langle z_0\rangle$ at the augmented cosimplicial ideal
$$0\rightarrow(z_0)\stack{3}(z_0,z_1)\stack{5}\cdots.$$
We view $(A^\bullet,\Oscr_K)$ and $(A^\bullet,\Oscr_K/\varpi^n)$ as augmented cosimplicial
$\delta$-pairs by letting $z_0,\ldots,z_s$ map to $\varpi$.

Using the prismatic envelope machinery of Section~\ref{sec:envelopes}, we can compute the prismatic
package of $\Oscr/\varpi^n$ relative to each $A^s$ individually.
Since the multivariable Breuil--Kisin prisms are orientable, one can trivialize away the Breuil--Kisin twist by
Lemma~\ref{lem:bk-orientation-existence} in computing the relative syntomic complexes
$\bZ_p(i)((\Oscr/\varpi^n)/A^\bullet)$ of Construction~\ref{const:relative_syntomic}, as
we explain in Section~\ref{sec:orienting}. But, the transition maps in the diagram are not compatibly maps of prisms, so in
the limit the Breuil--Kisin twist is non-trivial and we must account for it.
We explain how Breuil--Kisin orientations transform in Section~\ref{sec:bko} and study the
Hopf algebroid associated to the descent diagram in the remainder of
Section~\ref{sec:bkdescent}.

\subsection{Breuil--Kisin orientations}\label{sec:bko}

For a prism $(B,I)$, the Breuil--Kisin twist $B\{1\}$ admits a canonical
description as a limit
\begin{equation}\label{eq:bk}
    B\{1\} = \lim_r I_r/I_r^2=\lim\left(\cdots\rightarrow
    I_3/I_3^2\xrightarrow{\tfrac{\can}{p}}I_2/I_2^2\xrightarrow{\tfrac{\can}{p}}I_1/I_1^2\right)
\end{equation}
where $I_r = I\cdot\varphi^*_B(I)\cdots(\varphi^{r-1}_B)^*(I)$.
If $B$ is transversal, the transition map $I_r/I_r^2\rightarrow I_{r-1}/I_{r-1}^2$ for $r\geq 2$
is characterized as the unique map
which, after multiplication with $p$, agrees with the canonical map $I_r/I_r^2
\to I_{r-1}/I_{r-1}^2$. See~\cite[Sec.~2.2]{bhatt-lurie-apc} for
details.

\begin{proposition}\label{prop:twist_mod}
    If $(B,I)$ is a prism, the natural map $B\{1\}/I_r\rightarrow I_r/I_r^2$ is an
    isomorphism.
\end{proposition}

\begin{proof}
    This is~\cite[Prop.~2.2.12]{bhatt-lurie-apc}.
\end{proof}

\begin{remark}
    \label{rem:bk_general}
    In~\cite[Sec.~2.2]{bhatt-lurie-apc}, the authors construct $B\{1\}$ as a limit as
    in~\eqref{eq:bk} under the
    assumption that $B$ is transversal and they prove Proposition~\ref{prop:twist_mod}
    in that case as well. They go on to define $B\{1\}$ in general by base change from
    the transversal case by showing that the category of maps to a given prism from
    transversal ones is sifted. However, Proposition~\ref{prop:twist_mod} holds for
    general $B$ by base change again, since $I_r$ and the map involved are also obtained by base change
    from transversal prisms. Moreover, the limit description of $B\{1\}$ also holds in
    the non-transversal case, once the diagram has been obtained by base change from the
    transversal case. That is, for any prism $B$ there are natural maps
    $\tfrac{\can}{p}\colon I_r/I_{r}^2\rightarrow I_{r-1}/I_{r-1}^2$ for $r\geq 2$ and a
    natural equivalence identifying $B\{1\}$ as the limit of the tower of~\eqref{eq:bk}.
    To see this, it is enough to check it after derived modding out by $p$, where one can use the
    equivalences $I_{r+1}/p\we I^{[r+1]_p}/p$. Now, the result follows from
    $(p,I)$-completeness.
\end{remark}

In this section, we discuss choices of \emph{orientations} on $B\{1\}$, i.e.,
$B$-module isomorphisms $B\cong B\{1\}$. Such isomorphisms are in bijection to choices of
generators $s\in B\{1\}$ as a $B$-module.

\begin{definition}[Breuil--Kisin orientations]
    For a ring $R$ and an invertible $R$-module $M$, we write
    \[
    \Or_R(M)
    \]
    for the $R^{\times}$-torsor (possibly empty) of $R$-module isomorphisms $R\to M$.
    For a prism $(B,I)$, we refer to $\Or_{B}(I)$ as the set of
    orientations of $B$
    and $\Or_{B}(B\{1\})$ as the set of Breuil--Kisin orientations of $B$.
    The prism $(B,I)$ is orientable if $\Or_B(I)$ is nonempty and is Breuil--Kisin
    orientable if $\Or_B(B\{1\})$ is nonempty.
\end{definition}

\begin{remark}
    We will see in Lemma \ref{lem:bk-orientation-existence} that the existence of a
    Breuil--Kisin orientation is equivalent to the existence of an orientation in
    the usual sense (i.e. of the ideal $I$). However, when it comes to choices of
    orientations, the more natural direction is to go from a Breuil--Kisin
    orientation to an ordinary orientation.
\end{remark}

Recall the following construction from~\cite[Const.~2.2.14]{bhatt-lurie-apc}.

\begin{construction}[The Frobenius on the Breuil--Kisin twist]\label{const:frobenius_bk}
    The Breuil--Kisin twisted Frobenius for a prism $B$ arises from a natural isomorphism of $B$-modules
    \[
    B\{1\}\otimes_{B,\varphi} B \xto{\cong} I^{-1}\otimes_{B}B\{1\}
    \]
    whose reduction mod $\varphi^*_B(I_r)$ is given by the isomorphism
    \[
    \begin{tikzcd}
        B\{1\}/I_r \otimes_{B,\varphi} B \rar{\cong}\dar{\cong}&  I^{-1}\otimes_{B} B\{1\}/\varphi_B^*(I_r)\dar{\cong}\\
    I_r/I_r^2 \otimes_{B,\varphi} B \rar{\cong}\dar{\cong} & I^{-1}\otimes_{B} I_{r+1}/I_{r+1}\varphi_B^*(I_r)\dar{=}\\
        \varphi_B^*(I_r)/\varphi_B^*(I_r)^2 \rar{\cong} & I^{-1} \otimes_{B} I_{r+1}/I_{r+1}\varphi_B^*(I_r),
    \end{tikzcd}
    \]
    where the top right vertical arrow exists because $\varphi_B^*(I_r)$ contains $I_{r+1}$ so that
    $B\{1\}\otimes_{B}B/\varphi^*_B(I_r)\iso
    B\{1\}\otimes_{B}B/I_{r+1}\otimes_{B/I_{r+1}}B/\varphi^*_B(I_r)\iso
    I_{r+1}/I_{r+1}^2\otimes_{B/I_{r+1}}B/\varphi_B^*(I_r)$ by
    Proposition~\ref{prop:twist_mod} and where the bottom arrow is an isomorphism
    because $\varphi^*_B(I_r)\iso I^{-1}I_{r+1}$.
    Now, $\varphi_{B\{1\}}$ is the $\varphi_B$-semilinear map obtained by
    adjunction from the composition
    \[
    B\{1\}\to B\{1\}\otimes_{B,\varphi} B \xto{\cong} I^{-1}\otimes_{B}B\{1\};
    \]
    it satisfies
    \[
    \varphi_{B\{1\}}(as) = \varphi_B(a) \varphi_{B\{1\}}(s)
    \]
    for $a\in B$ and $s\in B\{1\}$.
\end{construction}

\begin{lemma}
\label{lem:bk-frobenius}
For a prism $B$, a Breuil--Kisin orientation $s\colon B\to B\{1\}$ on $B$ determines a unique orientation $d_s\in \Or(I)$ with
\[
\varphi_{B\{1\}}(as) = d_s^{-1} \varphi_B(a) s.
\]
The corresponding map
\[
\Or_B(B\{1\})\to \Or_B(I)
\]
is equivariant with respect to the homomorphism $B^\times\to B^\times$ given by
\[
u\mapsto \frac{u}{\varphi_B(u)}.
\]
\end{lemma}

\begin{proof}
As the map
\[
B\{1\}\otimes_{B,\varphi} B \xto{\cong} I^{-1}\otimes_{B}B\{1\}, \quad s\otimes b \mapsto b \cdot \varphi_{B\{1\}}(s)
\]
is an isomorphism, $\varphi_{B\{1\}}(s)$ is a uniquely determined generator of $I^{-1}\otimes_{B}
B\{1\}$, so of the form $d^{-1}_s\otimes s$ for a generator $d_s$ of $I$.
We can write $d_s$ symbolically as $\frac{s}{\varphi_{B\{1\}}(s)}$. It is
then clear that
\[
d_{s'} = \frac{us}{\varphi_{B\{1\}}(us)} = \frac{u}{\varphi_B(u)} d_s
\]
if $s'=us$ for a unit $u$ in $R$.
\end{proof}

The next lemma is also proven in~\cite[Rem.~2.5.8]{bhatt-lurie-apc}.

\begin{lemma}
\label{lem:bk-orientation-existence}
    A prism $(B,I)$ is orientable if and only if it is Breuil--Kisin
    orientable.
\end{lemma}

\begin{proof}
    Lemma $\ref{lem:bk-frobenius}$ shows that Breuil--Kisin orientability
    implies orientability. For the other direction, choose an orientation $d$
    of $I$. Then $I_r/I_{r}^2$ is a free $B/I_r$-module on the generator
    $d\cdot\varphi(d)\cdots \varphi^{r-1}(d)$. We now observe that, for $r\geq 1$, there are units
    $u_r$ such that
    \[
    \varphi^r(d) = p\cdot u_r \quad \text{ mod $I_r$}.
    \]
    For $r=1$, this follows from
    \[
    \varphi(d) = p\cdot \delta(d) + d^p
    \]
    since $d$ is distinguished by hypothesis. We proceed by induction and assume that
    \[
    \varphi^r(d) = p\cdot u_r + S_r
    \]
    with $S_r\in I_r$. Then
    \begin{align*}
    \varphi^{r+1}(d) &= \varphi^r(d)^p + p\varphi^r(\delta(d))\\
     &= \varphi^r(d)^{p-1}(p u_r + S_r) + p\varphi^r(\delta(d))\\
     &= p(\varphi^r(\delta(d)) + \varphi^r(d)^{p-1} u_r) + \varphi^r(d)^{p-1} S_r,
    \end{align*}
    with $\varphi^r(d)^{p-1} S_r\in I_{r+1}$ and $u_{r+1} := \varphi^r(\delta(d)) + \varphi^r(d)^{p-1} u_r$ a unit since $\delta(d)$ is a unit and $\varphi^r(d)\in (d,p)$.

    It follows that if $(B,I)$ is transversal, we can identify the unique ``$\frac{1}{p}\can$'' map $I_{r+1}/I_{r+1}^2\to I_r/I_{r}^2$ through the following commutative diagram:
    \[
    \begin{tikzcd}[column sep=2cm]
        B/I_{r+1}\dar{u_r}\rar{d\cdots \varphi^{r}(d)} &
        I_{r+1}/I_{r+1}^2\dar{\tfrac{\can}{p}}\\
    B/I_{r}\rar{d\cdots \varphi^{r-1}(d)} & I_{r}/I_{r}^2.
    \end{tikzcd}
    \]
    In particular, the tower of $I_{r}/I_{r}^2$ whose limit defines $B\{1\}$ is
    identified with the tower of $B/I_{r}$ with transition maps given by $u_r$
    times the canonical maps. This further identifies with the canonical tower
    of $B/I_r$ (using appropriate products of units). We get an identification
    $B\to B\{1\}$ in the transversal case. Since the Breuil--Kisin twist on
    non-transversal prisms is obtained by base change from an
    arbitrary transversal prism $\widetilde{B}\to B$
    by~\cite[Sec.~2.5]{bhatt-lurie-apc} and since an orientable prism
    can always be mapped to from an orientable transversal prism (e.g. the
    universal oriented prism), the claim follows for arbitrary prisms.
\end{proof}

\begin{remark}
    In principle, the above proof determines a canonical way to associate to an
    orientation a Breuil--Kisin orientation. This is however not inverse to the
    construction of Lemma \ref{lem:bk-frobenius}, and leads to cumbersome formulas
    in practice. In what follows, we will instead describe situations in which we
    can find preferred lifts against the construction of Lemma
    \ref{lem:bk-frobenius}, i.e., preferred refinements of orientations to
    Breuil--Kisin orientations.
\end{remark}

\begin{remark}
    Subsequently, we will never implicitly identify $B\{i\}$ with $B$. It is
    therefore harmless to just write $\varphi$ both for $\varphi_B$ and
    $\varphi_{B\{i\}}$, so we will do so. This makes computations with the
    Breuil--Kisin twisted Frobenius notationally very intuitive: for example,
    \[
    \varphi(as^i) = \varphi(a)\varphi(s)^i = d_s^{-i} \varphi(a) s^i
    \]
    describes the Frobenius map $B\{i\}\to I^{-i}\otimes_B B\{i\}$.
\end{remark}

\begin{lemma}\label{lem:crystalline-bk-orientation}
    If $(B,(p))$ is a crystalline prism, then there is a Breuil--Kisin
    orientation $s\colon B\to B\{1\}$ with $d_s = p$ which is uniquely characterized
    by the requirement that $s= p^r\pmod{I_r^2}$ in $B\{1\}/I_r^2\iso
    I_r/I_r^2$.
\end{lemma}

\begin{proof}
    While there is no \emph{unique} map $I_{r+1}/I_{r+1}^2 \to I_r/I_r^2$ whose multiple
    by $p$ is the canonical map in the crystalline case, it is still true that the canonical identifications
    \[
    \begin{tikzcd}
    B\{1\}/I_{r+1} \rar\dar{\cong} & B\{1\}/I_r\dar{\cong}\\
    I_{r+1}/I_{r+1}^2 \rar & I_r/I_r^2
    \end{tikzcd}
    \]
    exhibit $B\{1\}$ as the limit of \emph{some} such maps; see
    Remark~\ref{rem:bk_general}.

    In the crystalline case, the limits of any two systems of such maps are canonically
    identified, since they differ by maps which factor through the $p$-torsion subgroups
    of
    $(I_r/I_r^2)$, which form a pro-nil system. The maps taking $p^{r+1}$ in
    $I_{r+1}/I_{r+1}^2\iso (p^{r+1})/(p^{2r+2})$ to $p^r$ in $I_r/I_r^2\iso
    (p^r)/(p^{2r})$ yield one choice of such maps. We can thus canonically identify $B\{1\}$ with the limit along those maps. The collection of classes $[p^{r}]$ determine a generator $s\in B\{1\}$ in the limit.

    Chasing through the diagram of isomorphisms which defines the Frobenius
    \[
    B\{1\}\otimes_{B,\varphi} B \cong I^{-1}\otimes_B B\{1\},
    \]
    we see that for this Breuil--Kisin orientation, $d_s = p$.
\end{proof}

\begin{definition}[Crystalline Breuil--Kisin orientation]
    We will refer to the Breuil--Kisin orientation constructed in
    Lemma~\ref{lem:crystalline-bk-orientation} as the crystalline
    Breuil--Kisin orientation. Note that
    it is not uniquely characterized by the requirement that $d_s = p$:
    if $B=\bZ_p$, then changing $s$ by a unit $u\in\bZ_p^\times$ changes $d_s$ by
    $\frac{u}{\varphi(u)} = 1$.
\end{definition}

\begin{remark}
    Bhatt and Lurie show in~\cite[Prop.~2.6.1]{bhatt-lurie-apc} that there is a
    Breuil--Kisin orientation $s\in B\{1\}$ where $(B,I)$ is the $q$-de Rham
    prism with $B=\bZ_p\llbracket q-1\rrbracket$, $\varphi(q)=q^p$, and
    $I=[p]_q=1+q+\cdots+q^{p-1}$. Moreover, the associated orientation $d_s$ in
    this case is equal to $[p]_q$ by~\cite[Rem.~2.6.2]{bhatt-lurie-apc}.
    Specializing along $q\mapsto 1$, one obtains a Breuil--Kisin orientation of
    any crystalline prism with associated orientation equal to $p$;
    this gives an alternative proof of
    Lemma~\ref{lem:crystalline-bk-orientation}.
\end{remark}

In~\cite{akn-delta}, we introduced the notion of filtered
prisms $(\F^\star B,I)$. Say that a filtered prism $(\F^\star B,I)$ is crystalline if
$(\gr^0B,I\otimes_B\gr^0B)$ is crystalline.
Crystalline filtered prisms inherit canonical filtered-crystalline Breuil--Kisin orientations
under the additional assumption of completeness.

\begin{lemma}
    If $(\F^{\geq\star} A,I)$ is a complete filtered crystalline prism, then $(A,I)$
    is orientable.
\end{lemma}

\begin{proof}
    The map $I\rightarrow\gr_\F^{0}I$ is surjective, so choose an element $d\in I$
    that reduces modulo $\F^{\geq 1}$ to $p\in \gr^0_\F I=(p)\subseteq\gr^0_\F A$.
    Since $\delta(d)$  is a unit modulo $\F^{\geq 1}$ and since $\F^\star A$ is
    a complete filtered commutative ring, $\delta(d)$ is a unit in $A$. Thus,
    $(A,(d))$ is a prism and $(A,(d))\rightarrow (A,I)$ is a map of prisms.
    By rigidity of prisms~\cite[Prop.~3.5]{prisms}, $I=(d)$.
\end{proof}

\begin{lemma}\label{lem:filtered-crystalline-bk}
    Assume $(B,I)$ is a filtered prism. If $B$ is filtered-crystalline and
    complete, then for any $d\in I$ whose
    image in $\gr^0 I$ is $p$, there exists a unique lift $s\colon B\to B\{1\}$ of
    the crystalline Breuil--Kisin orientation on $\gr^0 B$ with $d_s = d$.
\end{lemma}

\begin{proof}
    Any two lifts of the crystalline Breuil--Kisin orientation differ by
    elements of $(1+ \F^{\geq 1}B)^\times \subseteq B^{\times}$. Similarly, $d_s$ for
    such a lift differs from $d$ by an element of $(1+\F^{\geq 1}B)^\times$. The claim
    now follows from the observation that the map
    \[
    u \mapsto \frac{u}{\varphi(u)}
    \]
    is an automorphism of $(1+\F^{\geq 1}B)^\times$, since it agrees with the
    identity on the associated graded pieces as the Frobenius raises filtration
    weight from $m$ to $pm$.
\end{proof}

This enables us to talk about canonical choices of Breuil--Kisin orientations
on filtered crystalline prisms such as $(z^\star W(k)\llbracket z\rrbracket,E(z))$,
where $E(z)$ is a choice of a (not necessarily monic) Eisenstein
polynomial with $E(0)=p$.

For future reference, we also record how these transform.

\begin{lemma}\label{lem:bk-factor-general}
    Let $(\F^\star B,I)$ be a complete filtered-crystalline prism, and $d,d'$ two generators of
    $I$ whose image in $\gr^0 I$ is $p$. Let $u\in (1+\F^{\geq 1}B)^\times$ with
    $d'=ud$. If $s$ and $s'$ are the canonical filtered-crystalline Breuil--Kisin
    orientations with $d_s = d$ and $d_{s'}=d'$, then
    \[
    s' = \left(\prod_{r\geq 0} \varphi^r(u)\right)\cdot s.
    \]
\end{lemma}

\begin{proof}
We have $s' = v\cdot s$ with $v\in (1+\F^{\geq 1}B)^\times$ by construction. Now
\[
d_{s'} = \frac{v}{\varphi(v)} d_{s},
\]
i.e., $\frac{v}{\varphi(v)} = u$. The unique solution $v$ is given by the product
\[
v = \prod_{r\geq 0} \varphi^r(u),
\]
which converges because the filtration $\F^\star B$ is complete and because
Frobenius raises filtration weight $m$ to $pm$.
\end{proof}

\subsection{Orienting the syntomic complex}\label{sec:orienting}

Suppose that $(A,I)$ is a prism with a Breuil--Kisin orientation $s\in A\{1\}$. Our goal
in this section is to use $s$ to rewrite the relative syntomic complex of a commutative
$A$-algebra $R$
$$\bZ_p(i)(R/A)\we\fib\left(\N^{\geq
i}\Prismhat^{(1)}_{R/A}\{i\}\xrightarrow{\can-\varphi}\Prismhat^{(1)}_{R/A}\{i\}\right)$$
as an isomorphic complex
$$\fib\left(\N^{\geq
i}\Prismhat^{(1)}_{R/A}\xrightarrow{\can-\varphi_i}\Prismhat^{(1)}_{R/A}\right),$$ for
an appropriate divided Frobenius morphism $\varphi_i$.

We briefly review the Nygaard filtration and the construction of the relative syntomic
complexes.

\begin{construction}
    We write $\varphi_A^*$ for the $(p,I)$-completed extension of scalars along the
    Frobenius of $A$.
    Suppose that $(A,I)$ is a prism and recall the $A$-linear isomorphism
    $\varphi_A^*(A\{i\})\iso I^{-i}A\{i\}$ given in
    Construction~\ref{const:frobenius_bk}. Given a commutative $\overline{A}$-algebra
    $R$, there is an induced $A$-linear map
    $$\varphi_A^*(\Prism_{R/A}\{i\})\iso\varphi_A^*(\Prism_{R/A})\otimes_A\varphi_A^*(A\{i\})\iso\varphi_A^*(\Prism_{R/A})\otimes_A
    I^{-i}A\{i\}\xrightarrow{\varphi_{R/A}\otimes\id} \Prism_{R/A}\otimes_A I^{-i}A\{i\}\iso
    I^{-i}\Prism_{R/A}\{i\},$$ where $\varphi_A^*(\Prism_{R/A})\rightarrow\Prism_{R/A}$
    is the relative Frobenius $\varphi_{R/A}$.
\end{construction}

\begin{notation}
    We write $\Prism_{R/A}^{(1)}\{i\}$ for
    $\varphi_A^*(\Prism_{R/A}\{i\})$, despite the possibility for confusion.
\end{notation}

\begin{remark}[The Nygaard filtration]
    The Frobenius $\varphi\colon\Prism^{(1)}_{R/A}\{i\}\rightarrow
    I^{-i}\Prism_{R/A}\{i\}$ refines to a filtered map
    $\N^{\geq\star}\Prism^{(1)}_{R/A}\rightarrow I^{\star-i}\Prism_{R/A}\{i\}$. In fact,
    this can be taken to be the defining property of $\N^{\geq\star}\Prism_{R/A}\{i\}$ in the case when
    $\L_{R/\overline{A}}$ has $p$-complete Tor-amplitude in $[1,1]$. In that case,
    $\N^{\geq j}\Prism_{R/A}^{(1)}\{i\}$ is defined to be
    the non-derived pullback
    $$\xymatrix{
        \N^{\geq j}\Prism_{R/A}^{(1)}\{i\}\ar[r]\ar[d]&I^{j-i}\Prism_{R/A}\{i\}\ar[d]\\
        \Prism_{R/A}^{(1)}\{i\}\ar[r]&I^{-i}\Prism_{R/A}\{i\}
    }$$
    for $j\geq 0$. For details, see~\cite[Thm.~15.2]{prisms}.
\end{remark}

\begin{definition}[The Breuil--Kisin orientation on the Frobenius twist]
    Suppose that $A$ is a prism with a Breuil--Kisin orientation $s$. In this case,
    there is an induced $A$-linear isomorphism $\Prism_{R/A}^{(1)}\xrightarrow{w(s^i)}\Prism_{R/A}^{(1)}\{i\}$.
\end{definition}

\begin{lemma}\label{lem:divided_frobenius}
    Suppose that $A$ is a prism with a Breuil--Kisin orientation $s$ and that $R$ is a
    commutative $A$-algebra. Then, there is a natural commutative diagram
    $$\xymatrix{
        \Prism^{(1)}_{R/A}\ar[d]_\iso^{w(s^i)}\ar[r]^{\varphi_{i,R/A}}&I^{-i}\Prism_{R/A}\ar[d]_\iso^{s^i}\\
    \Prism^{(1)}_{R/A}\{i\}\ar[r]^{\varphi_{R/A}}& I^{-i}\Prism_{R/A}\{i\}
    }$$
    for some $A$-linear map $\varphi_{i,R/A}$. If $x\in\N^{\geq i}\Prism_{R/A}^{(1)}$, then 
    $\varphi_{i,R/A}(x)=\tfrac{\varphi_{R/A}}{d_s^i}$, where $d_s$ is the orientation
    associated to the Breuil--Kisin orientation $s$.
\end{lemma}

\begin{proof}
    Existence follows from the fact that $w(s^i)$ and $s^i$ are isomorphisms.
    Given $x\in\Prismhat_{R/A}^{(1)}$, we have that
    $$\varphi_{R/A}(x w(s^i))=d_s^{-i}\varphi_{R/A}(x)s^i=\varphi_{i,R/A}(x)s^i,$$
    where the first equality is by construction and the second is by commutativity of
    the diagram and the definition of $\varphi_{i,R/A}$.
    As $s^i$ is invertible it follows that $\varphi_{i,R/A}(x)=d_s^{-i}\varphi_{R/A}(x)$
    in $I^{-i}\Prism_{R/A}$. Now, if $x\in\N^{\geq i}\Prism_{R/A}^{(1)}$, then we can
    write $\varphi_{i,R/A}(x)=d_s^{-i}\varphi_{i,R/A}(x)$ in $\Prism_{R/A}$, as desired.
\end{proof}

\begin{definition}\label{def:relative_divided_frobenius}
    Let $s$ be a Breuil--Kisin orientation of a prism $A$ and let $R$ be a commutative
    ring. We define $$\varphi_{i,R/A}(x)=\frac{\varphi_{R/A}(x)}{d_s^{i}}$$ for
    $x\in\N^{\geq i}\Prism_{R/A}^{(1)}$. This is the $i$th relative divided Frobenius.
\end{definition}

\begin{definition}\label{def:divided_frobenius}
    The square
    \begin{equation}\label{eq:w_square}\begin{gathered}\xymatrix{
            \Prism_{R/A}\ar[r]^w\ar[d]^{s^i}&\Prism_{R/A}^{(1)}\ar[d]^{w(s^i)}\\
        \Prism_{R/A}\{i\}\ar[r]^{w}&\Prism_{R/A}^{(1)}\{i\},
    }\end{gathered}\end{equation} is naturally commutative.
    It follows that by composing $\varphi_{i,R/A}$ with
    $w\colon\Prism_{R/A}\rightarrow\Prism_{R/A}^{(1)}$, one obtains a map
    $$\varphi_i\colon\N^{\geq i}\Prism_{R/A}^{(1)}\rightarrow\Prism_{R/A}^{(1)},$$
    semilinear with respect to the Frobenius on $\Prism_{R/A}^{(1)}$ and which satisfies
    $$\varphi_i(x)=\frac{\varphi(x)}{\varphi(d)^i}.$$
    This follows from the fact that $w$ is $\varphi$-semilinear, so $w(d)=\varphi(d)$, and
    $w\circ\varphi_{R/A}=\varphi_{\Prism_{R/A}^{(1)}}$.
\end{definition}

\begin{definition}[Oriented syntomic complexes]
    Both $\can,\varphi_i\colon\N^{\geq i}\Prism_{R/A}^{(1)}\rightarrow\Prism_{R/A}^{(1)}$ 
    descend to the Nygaard-completions to give maps $\can,\varphi_i\colon\N^{\geq
    i}\Prismhat_{R/A}^{(1)}\rightarrow\Prismhat_{R/A}^{(1)}$.
    Fix a prism $A$ and a commutative $\overline{A}$-algebra $R$.
    If $s$ is a Breuil--Kisin orientation of $A$, then the oriented syntomic complex
    $\bZ_p(i)(R/A,s^i)$ is defined as $$\fib\left(\N^{\geq
    i}\Prismhat_{R/A}^{(1)}\xrightarrow{\can-\varphi_i}\Prismhat_{R/A}^{(1)}\right).$$
\end{definition}

\begin{proposition}
    \label{prop:oriented_syntomic}
    If $A$ is a prism, $R$ is a commutative $\overline{A}$-algebra, and $s$ is a
    Breuil--Kisin orientation for $A$, then there is an equivalence
    $$\bZ_p(i)(R/A,s^i)\we\bZ_p(i)(R/A),$$ natural in $R$ and the pair $(A,s)$.
\end{proposition}

\begin{proof}
    The proposition follows from Lemma~\ref{lem:divided_frobenius}, the commutativity
    of~\eqref{eq:w_square}, and the compatibility of the Nygaard filtrations on
    $\Prism_{R/A}^{(1)}$ and $\Prism^{(1)}_{R/A}\{i\}$
    with the Breuil--Kisin twists under the orientations $s^i$.
\end{proof}

\subsection{The descent diagram}

Let $\cO_K$ be a DVR with perfect residue field $k$, uniformizer $\varpi$, and
Eisenstein polynomial $E(z_0) \in W(k)\llbracket z_0\rrbracket $. There is a prism
structure on $W(k)\llbracket z_0\rrbracket $ with $d=E(z_0)$. Let $R = \cO_K/\varpi^n$. Then
\[
\prism_{R / W(k)\llbracket z_0\rrbracket } = W(k)\llbracket z_0\rrbracket
\left\{\frac{z_0^n}{E(z_0)}\right\}.
\]
More generally, we have a prism structure on $W(k)\llbracket z_0,\ldots,z_s\rrbracket $, given (somewhat arbitrarily) by the ideal $(E(z_0))$. We have
\[
\prism_{R / W(k)\llbracket z_0,\ldots,z_s\rrbracket } = W(k)\llbracket z_0,\ldots,z_s\rrbracket  \left\{\frac{z^n_0}{E(z_0)}, \frac{z_1-z_0}{E(z_0)},\ldots,\frac{z_s-z_0}{E(z_0)}\right\}^\wedge,
\]
where we take all $z_j \mapsto \varpi$. These are more natural than we expect
from the description. For example, there should be a map from $\prism_{R/W(k)\llbracket
z_0\rrbracket } \to \prism_{R/W(k)\llbracket z_0,z_1\rrbracket }$ taking $z_0\mapsto z_1$,
even though that does not define a map on the base prisms. The problem is the arbitrary
choice of ideal $(E(z_0))$: The map $z_0\mapsto z_1$ is only compatible with the prism
structure defined by $(E(z_1))$ instead. These different choices do however become
equivalent in $\prism_{R/W(k)\llbracket z_0,\ldots,z_s\rrbracket }$, as one can see from
the prismatic envelope description: $E(z_j) - E(z_0)$ is divisible by $z_j-z_0$, and is
thus being made divisible by $E(z_0)$, since the prismatic envelope contains an element
$\frac{z_j-z_0}{E(z_0)}$. In fact, the various $E(z_j)$ become unit multiples of each other in the
prismatic envelope. Hence the above prismatic envelope also contains elements
$\frac{z_{j'}-z_0}{E(z_{j})}$ for $0\leq j,j'\leq s$ and satisfies the universal property of the prismatic envelope
describing prismatic cohomology relative to the prism structure $(E(z_j))$ on the base.

In \cite{akn-delta}, this idea is used to systematically extend prismatic cohomology to
a functor defined on the category of $\delta$-pairs $(A,R)$, where $A$ is a
$\delta$-ring and $R$ an $A$-algebra in such a way that, for a prism $(A,I)$ and an $A/I$-algebra
$R$, it agrees with derived prismatic cohomology in the sense of
\cite{prisms}, and, for $A=\bZ$, it agrees with absolute prismatic cohomology
in the sense of \cite{bhatt-lurie-apc}. This additional functoriality allows us to make
the following statement.

\begin{lemma}\label{lem:local_to_global}
  For a perfect $\bF_p$-algebra $k$ and a $W(k)\llbracket z_0\rrbracket$-algebra $R$, we have
      \[
        \Prism_R\{i\} \simeq \Prism_{R/W(k)}\{i\}\simeq \Tot\left(\Prism_{R/W(k)\llbracket z_0\rrbracket}\{i\} \stack{3} \Prism_{R/W(k)\llbracket z_0,z_1\rrbracket }\{i\} \stack{5} \ldots\right)\]
    and
    \[
        \Prismhat^{(1)}_{R/W(k)}\{i\} \simeq
    \Tot\left(\Prismhat_{R/W(k)\llbracket z_0\rrbracket }^{(1)}\{i\} \stack{3}
    \Prismhat^{(1)}_{R/W(k)\llbracket z_0,z_1\rrbracket }\{i\} \stack{5} \ldots\right)
\]
for all weights $i\in\bZ$.
\end{lemma}

\begin{proof}
    See~\cite[Thm.~1.2(6)]{akn-delta}.
\end{proof}

We will call the cosimplicial diagrams of Lemma~\ref{lem:local_to_global} the \emph{relative-to-absolute descent diagrams}
for absolute prismatic cohomology and Nygaard-complete absolute prismatic cohomology,
respectively. They give prismatic cohomology analogues of the approach to $\TP$ and
$\TC$ pioneered by Liu and Wang in
\cite{liu-wang}. Note that in the case $R = \cO_K / \varpi^n$, all terms appearing in the
diagrams are discrete, and thus the totalization is represented by the total complex of the
diagram. In this case, we will also call this diagram the relative-to-absolute descent complex.
There is also an $\F$-completed version of the untwisted relative-to-absolute descent
diagram which will play a role in the arguments below.

\begin{notation}
    If $(\F^\star A,\F^\star R)$ is a filtered $\delta$-pair in the sense
    of~\cite{akn-delta}, then $\Prism_{R/A}$ admits an associated $\F$-filtration
    $\F^\star\Prism_{R/A}$. We will write
    $\Prismhat_{R/A}$ for the completion of $\Prism_{R/A}$ with respect to $\F^\star$.
    If $A=\bZ_p$ with the trivial filtration, we write $\Prismhat_R$ for
    $\Prismhat_{R/\bZ_p}$.
\end{notation}

\begin{warning}
    Recall that for rings $R$ such as $\Oscr_K$ or $\Oscr_K/\varpi^n$, the Nygaard
    completion $\Prismhat_{R/W(k)\llbracket z_0,\ldots,z_s\rrbracket}^{(1)}$ agrees with
    the $\F$-completion of $\Prism_{R/W(k)\llbracket z_0,\ldots,z_s\rrbracket}^{(1)}$ by
    Theorem~\ref{thm:n_v_f}.
    In~\cite{bms2,bhatt-lurie-apc}, the Nygaard filtration is constructed on absolute
    prismatic cohomology $\Prism_R$ and one typically writes $\Prismhat_R$ for the
    completion with respect to the Nygaard filtration.
    In general, this conflicts with our notation. However, for rings $R$ such as
    $\Oscr_K$ or $\Oscr_K/\varpi^n$, the two completions agree using an argument similar
    to that of Theorem~\ref{thm:n_v_f}. In this paper,
    $\Prismhat_{R/A}$ will always denote the completion of $\Prism_{R/A}$ with respect
    to the $\F$-filtration induced by a filtration on the $\delta$-pair $(A,R)$ while
    $\Prismhat_{R/A}^{(1)}$ will always denote the completion of $\Prism_{R/A}^{(1)}$ with
    respect to the Nygaard filtration.
\end{warning}

\begin{definition}
    By Lemma~\ref{lem:local_to_global} for a filtered commutative $z_0^\star W(k)\llbracket
    z_0\rrbracket$-algebra $R$, there is an induced equivalence
      \[
        \Prismhat_R\{i\} \simeq \Prismhat_{R/W(k)}\{i\}\simeq
        \Tot\left(\Prismhat_{R/W(k)\llbracket z_0\rrbracket}\{i\} \stack{3}
        \Prismhat_{R/W(k)\llbracket z_0,z_1\rrbracket }\{i\} \stack{5} \ldots\right).\]
    We will refer to this as the {\em $\F$-completed relative-to-absolute descent
    diagram} for
    absolute prismatic cohomology.
\end{definition}

The relative-to-absolute descent diagrams are the cobar complexes associated to Hopf
algebroids, as we now explain.

\begin{definition}[Hopf algebroids]
    A Hopf algebroid consists of a pair of commutative rings $(\Gamma_0,\Gamma_1)$ with maps
    \begin{enumerate}
      \item $\eta_L,\eta_R\colon \Gamma_0\to \Gamma_1$,
      \item $\varepsilon\colon \Gamma_1\to \Gamma_0$,
      \item $\iota\colon \Gamma_1\to \Gamma_1$, and
      \item $\Delta\colon \Gamma_1\to \Gamma_1\otimes_{\eta_R,\Gamma_0,\eta_L} \Gamma_1$
    \end{enumerate}
    fulfilling various identities which make the pair $(\Spec\Gamma_0,\Spec\Gamma_1)$ into a
    groupoid object in the category of affine schemes. See~\cite[App.~1]{ravenel-green} for
    details.
\end{definition}

\begin{definition}[Comodules]
    For a Hopf algebroid $(\Gamma_0,\Gamma_1)$, a (right) \emph{comodule} is given by a
    right $\Gamma_0$-module $M$ together with a coaction map $M \to M\otimes_{\Gamma_0}
    \Gamma_1$, which is a right $\Gamma_0$-module map and which satisfies
    counit and coassociative identities as in~\cite[App.~1]{ravenel-green}.
\end{definition}

\begin{definition}
    The cobar complex of a Hopf algebroid $(\Gamma_0,\Gamma_1)$ is a cosimplicial
    commutative ring $\cB(\Gamma_0,\Gamma_1)$ of the form
    \[
      \Gamma_0 \stack{3} \Gamma_1 \stack{5} \Gamma_1\otimes_{\Gamma_0} \Gamma_1 \cdots.
    \]
    To a comodule $M$ over $(\Gamma_0,\Gamma_1)$, there is an associated cobar complex
    $\cB(M)$, taking the form
    \[
      M \stack{3} M\otimes_{\Gamma_0} \Gamma_1 \stack{5} M\otimes_{\Gamma_0} \Gamma_1 \otimes_{\Gamma_0}\Gamma_1\cdots,
    \]
    which is a cosimplicial module over $\cB(\Gamma_0,\Gamma_1)$.
    See~\cite[Def.~A1.2.11]{ravenel-green} for a reference to the reduced version of
    this construction.
\end{definition}

\begin{remark}
    Analogous definitions make sense in (complete) filtered commutative rings and modules.
\end{remark}

\begin{construction}\label{const:hopf_algebroid}
    Consider the Hopf algebroid $(W(k)\llbracket z_0\rrbracket,W(k)\llbracket
    z_0,z_1\rrbracket)$
    in commutative $W(k)$-algebras,
    where
    \begin{enumerate}
        \item[(1)] $\eta_L(z_0)=z_0$ and $\eta_R(z_0)=z_1$,
        \item[(2)] $\epsilon(z_0)=\epsilon(z_1)=z_0$,
        \item[(3)] $\iota(z_0)=z_1$ and $\iota(z_1)=z_0$, and
        \item[(4)] $\Delta(z_0)=z_0$ and $\Delta(z_1)=z_2$ under the identification
            $$W(k)\llbracket z_0,z_1\rrbracket\otimes_{\eta_R,W(k)\llbracket
            z_0\rrbracket,\eta_L}W(k)\llbracket z_0,z_1\rrbracket\iso W(k)\llbracket
            z_0,z_1,z_2\rrbracket.$$
    \end{enumerate}
    Write $W(k)\llbracket z_\bullet\rrbracket$ for the cobar complex of this Hopf
    algebroid, which agrees with the (completed) descent complex of $W(k)\rightarrow W(k)\llbracket
    z_0\rrbracket$. If $R$ is a commutative $W(k)\llbracket z_0\rrbracket$-algebra, we
    can consider the pair of pairs $$((R/W(k)\llbracket z_0\rrbracket),(R/W(k)\llbracket
    z_0,z_1\rrbracket)),$$ which is a Hopf algebroid in the category of $\delta$-pairs.
\end{construction}

\begin{lemma}\label{lem:prismatic_hopf_algebroid}
    Suppose that $R$ is a commutative $W(k)\llbracket z_0\rrbracket$-algebra such that
    $\L_{R/W(k)\llbracket z_0\rrbracket}$ has $p$-complete Tor-amplitude in $[0,1]$.
  \begin{enumerate}
      \item[{\em (a)}] The pair $(\Prism_{R/W(k)\llbracket z_0\rrbracket}, \Prism_{R/W(k)\llbracket
          z_0,z_1\rrbracket})$ forms a complete filtered Hopf algebroid (with respect to
          the Hodge--Tate filtration) and the
          local-to-global descent complex $\Prism_{R/W(k)\llbracket z_\bullet\rrbracket}$ identifies with its cobar complex.
      \item[{\em (b)}] For each integer $i$, the $i$th Breuil--Kisin twist
          $\Prism_{R/W(k)\llbracket z_0\rrbracket}\{i\}$ is a comodule over the Hopf
          algebroid of part (a) and the descent complex $\Prism_{R/W(k)\llbracket z_\bullet\rrbracket}\{i\}$ identifies with its cobar complex.
  \end{enumerate}
\end{lemma}

\begin{proof}
    We can apply prismatic cohomology (in the sense of~\cite{akn-delta}) pointwise to the cobar complex of the Hopf algebroid in
    $\delta$-rings of Construction~\ref{const:hopf_algebroid}. This produces the
    local-to-global descent complex computing absolute prismatic cohomology of
    Lemma~\ref{lem:local_to_global}. For each $i\in\bZ$, $s\geq 2$, and $0\leq j\leq s$, the natural map
  \[
    \Prism_{R/W(k)\llbracket z_0,\ldots,z_j\rrbracket}\{i\} \otimes_{\Prism_{R/W(k)\llbracket z_j\rrbracket}} \Prism_{R/W(k)\llbracket z_j,\ldots,z_s\rrbracket} \to \Prism_{R/W(k)\llbracket z_0,\ldots,z_s\rrbracket}\{i\}
  \]
  is an equivalence by symmetric monoidality of prismatic cohomology for prismatic
    $\delta$-pairs, which follows for example from the compatibility of relative
    (derived) Cartier--Witt
    stacks with limits of schemes and the affineness of all such stacks in question
    (see~\cite[Thm.~7.17]{bhatt-lurie-prism}), and by functoriality of the Breuil--Kisin twists~\cite[Prop.~2.5.1]{bhatt-lurie-apc}.
    Taking $s=2$ and $j=1$ produces the comultiplication $\Delta$.
    For $s>2$, one learns that the cobar complex is indeed isomorphic to the local-to-global descent
    complex. This completes the proof.
\end{proof}

\begin{remark}\label{rem:completed_hopf_algebroid}
    The $\F$-completion $(\F^\star\Prismhat_{R/W(k)\llbracket
    z_0\rrbracket},\F^\star\Prismhat_{R/W(k)\llbracket z_0,z_1\rrbracket})$ of the Hopf
    algebroid of Lemma~\ref{lem:prismatic_hopf_algebroid} is a Hopf algebroid in complete
    $p$-complete filtered complexes (with respect to the $\F$-filtration) and the
    $\F$-completed relative-to-absolute descent diagram is its cobar complex.
    A result analogous to Lemma~\ref{lem:prismatic_hopf_algebroid}(2) holds for the Breuil--Kisin twists.
\end{remark}

The analogous statements hold with Frobenius twists, and the proof is similar.
For the relevant symmetric monoidality of the Nygaard-completed Frobenius twists, one
can reduce to the associated graded pieces of the Nygaard filtration and hence to the
conjugate filtered pieces of Hodge--Tate cohomology where it follows using symmetric
monoidality of $p$-completed derived differential forms $\widehat{\L\Omega}^*_{-/-}$.

\begin{lemma}
  \begin{enumerate}
      \item[{\em (a)}] The pair $(\Prismhat_{R/W(k)\llbracket z_0\rrbracket}^{(1)},
          \Prismhat_{R/W(k)\llbracket z_0,z_1\rrbracket}^{(1)})$ forms a complete
          filtered Hopf algebroid (with respect to the Nygaard filtration), and
          the descent complex $\Prismhat_{R/W(k)\llbracket z_\bullet\rrbracket}^{(1)}$
          identifies with its cobar complex.
      \item[{\em (b)}] For each integer $i$, the $i$th Breuil--Kisin twist
          $\Prismhat_{R/W(k)\llbracket z_0\rrbracket}^{(1)}\{i\}$ is a comodule over the
          Hopf algebroid of part (a) and the local-to-global descent complex $\Prismhat_{R/W(k)\llbracket z_\bullet\rrbracket}^{(1)}\{i\}$ identifies with its cobar complex.
  \end{enumerate}
\end{lemma}

\subsection{Structure maps of the descent complex}

We now describe the structure maps in detail for $R= \cO_K/\varpi^n$. From now on, we
will be using $\F$-completed prismatic cohomology $\Prismhat_{R/A}$. In this case,
\[
  \Prismhat_{R/W(k)\llbracket z_0\rrbracket} = W(k)\llbracket z_0\rrbracket \left\{\frac{z_0^n}{E(z_0)}\right\}^\wedge_{p,E(z_0)}
\]
is an $\F$-completely free $W(k)$-module on monomials $z_0^k \prod_{u}
\delta^u(a)^{e_u}$ for $k<n$ and $e_u<p$, where $a = \frac{z_0^n}{E(z_0)}$
by Proposition~\ref{prop:f_filtered_statement_for_prism}. On
\[
  \Prismhat_{R/W(k)\llbracket z_0,z_1\rrbracket} = W(k)\llbracket z_0\rrbracket \left\{\frac{z_0^n}{E(z_0)}, \frac{z_1-z_0}{E(z_0)}\right\}^\wedge_p
\]
we similarly have a $W(k)$-basis of monomials $z_0^k \prod_{u} \delta^u(a)^{e_u} \prod_u \delta^u(b)^{e'_u}$ with $k<n$ and $e_u,e'_u<p$, where $a = \frac{z_0^k}{E(z_0)}$ and $b=\frac{z_1-z_0}{E(z_0)}$.

Observe that the two orientations $E(z_1)$ and $E(z_0)$ of $\Prismhat_{R/W(k)\llbracket z_0,z_1\rrbracket}$ differ by the factor
\[
  u = \frac{E(z_1)}{E(z_0)} = 1 + \frac{E(z_1)-E(z_0)}{z_1-z_0} b.
\]
By Lemma~\ref{lem:bk-factor-general}, the corresponding filtered-crystalline
Breuil--Kisin orientations $s_0$ and $s_1$ compatible with the left and right unit maps
from $\Prismhat_{R/W(k)\llbracket z_1\rrbracket}$ differ by the factor $v = \frac{s_1}{s_0}
= \prod_{r\geq 0} \varphi^r(u)$.

\begin{lemma}\label{lem:untwisted_hopf_algebroid}
    The structure maps in the Hopf algebroid $(\Prismhat_{R/W(k)\llbracket
    z_0\rrbracket},\Prismhat_{R/W(k)\llbracket z_0,z_1\rrbracket})$ are determined, in terms of
    the prismatic envelope description of
    Proposition~\ref{prop:f_filtered_statement_for_prism}, by noting that
  \begin{enumerate}
      \item[{\em (i)}] $\eta_L\colon\Prismhat_{R/W(k)\llbracket z_0\rrbracket}\to
          \Prismhat_{R/W(k)\llbracket z_0,z_1\rrbracket}$ takes $z_0 \mapsto z_0$ and $a\mapsto a$;
      \item[{\em (ii)}] $\eta_R\colon \Prismhat_{R/W(k)\llbracket z_0\rrbracket}\to
          \Prismhat_{R/W(k)\llbracket z_0,z_1\rrbracket}$ takes 
      \[
        z_0 \mapsto z_1 = z_0 + E(z_0)\cdot b
      \]
      and 
      \[
        a\mapsto \frac{z_1^n}{E(z_1)} =  \frac{(z_0 + E(z_0)b)^n}{E(z_0)} \cdot u^{-1};
      \]
  \item[{\em (iii)}] $\varepsilon\colon \Prismhat_{R/W(k)\llbracket
      z_0,z_1\rrbracket}\to \Prismhat_{R/W(k)\llbracket z_0\rrbracket}$ takes $z_0\mapsto z_0$, $a\mapsto a$, and $b \mapsto 0$;
  \item[{\em (iv)}] $\iota\colon \Prismhat_{R/W(k)\llbracket z_0,z_1\rrbracket}\to
      \Prismhat_{R/W(k)\llbracket z_0,z_1\rrbracket}$ takes $z_0 \mapsto \eta_R(z_0)=z_1$, $a\mapsto \eta_R(a)$, and
      \[
        b \mapsto \frac{z_0-z_1}{E(z_1)} = -b u^{-1};
      \]
  \item[{\em (v)}] $\Delta\colon \Prismhat_{R/W(k)\llbracket z_0,z_1\rrbracket}
      \to \Prismhat_{R/W(k)\llbracket z_0,z_1,z_2\rrbracket} \cong \Prismhat_{R/W(k)\llbracket
          z_0,z_1\rrbracket} \otimes_{\eta_R,\Prism_{R/W(k)\llbracket
          z_0\rrbracket},\eta_L} \Prismhat_{R/W(k)\llbracket z_0,z_1\rrbracket}$
    takes $z_0\mapsto z_0\otimes 1$, $a\mapsto a\otimes 1$, and
        \[
          b = \frac{z_1-z_0}{E(z_0)}\mapsto \frac{z_2-z_0}{E(z_0)} =
          \frac{z_1-z_0}{E(z_0)} + \frac{z_2-z_1}{E(z_1)}\left(1 +
          \frac{E(z_1)-E(z_0)}{z_1-z_0} b\right) = b \otimes 1 + u \otimes b.
        \]
  \end{enumerate}
  The values on the remaining generators are determined by compatibility of the
    structure maps with the $\delta$-ring structures.
\end{lemma}
\begin{proof}
  This follows from the description as prismatic envelope.
\end{proof}

For the Frobenius-twisted term, we have a similar description. Recall from
Proposition~\ref{prop:f_filtered_statement} that
\[
  \Prismhat^{(1)}_{R/W(k)\llbracket z_0\rrbracket} = W(k)\llbracket z_0\rrbracket \left\{\frac{z_0^{pn}}{\varphi(E(z_0))}\right\}^\wedge
\]
is $\F$-completely $W(k)$-free on monomials $z_0^k \prod f_u^{e_u}$ where $k<n$ and $e_u<p$, and
\[
  \Prismhat^{(1)}_{R/W(k)\llbracket z_0,z_1\rrbracket} = W(k)\llbracket z_0,z_1\rrbracket \left\{\frac{z_0^{pn}}{\varphi(E(z_0))}, \frac{z_1^p-z_0^p}{\varphi(E(z_0))}\right\}^\wedge
\]
is $\F$-completely $W(k)$-free on monomials $z_0^k \prod f_u^{e_u} \prod g_u^{e'_u}$ where $k<n$, $e_u<p$, $e'_u<p$.
Here $f_0$ is $z_0^n$, viewed as element of $\N^{\geq 1}\Prismhat^{(1)}$, $g_0$ is $z_1-z_0$, viewed as element of $\N^{\geq 1}\Prismhat^{(1)}$, and $f_u$ and $g_u$ arise from iterating $\widetilde{\delta}: \N^{\geq \star}\Prismhat^{(1)}\to \N^{\geq p\star}\Prismhat^{(1)}$.

\begin{lemma}\label{lem:hopf_algebroid}
    The structure maps in the Hopf algebroid $(\Prismhat^{(1)}_{R/W(k)\llbracket
    z_0\rrbracket},\Prismhat^{(1)}_{R/W(k)\llbracket z_0,z_1\rrbracket})$ are determined, in terms of
    the prismatic envelope description of Proposition~\ref{prop:f_filtered_statement}, by
    noting that
  \begin{enumerate}
      \item[{\em (a)}] $\eta_L$ takes $z_0\mapsto z_0$ and $f_0\mapsto f_0$;
      \item[{\em (b)}] $\eta_R$ takes $z_0 \mapsto z_1 = z_0 + g_0$ and $f_0 \mapsto (z_0+g_0)^n$;
      \item[{\em (c)}] $\varepsilon$ takes $z_0 \mapsto z_0$, $f_0\mapsto f_0$ and $g_0\mapsto 0$;
      \item[{\em (d)}] $\iota$ takes $z_0 \mapsto \eta_R(z_0)=z_0+g_0$, $f_0\mapsto
          \eta_R(f_0)=(z_0+g_0)^n$ and $g_0 \mapsto z_0 - z_1 = -g_0$;
      \item[{\em (e)}] $\Delta$ takes $z_0\mapsto z_0$, $f_0\mapsto f_0$, and $g_0 \mapsto z_2 - z_0 = (z_1 - z_0) + (z_2-z_1) = g_0\otimes 1 + 1\otimes g_0$.
  \end{enumerate}
  The values on the other generators are determined by compatibility of the structure
    maps with the $\delta$-ring structures.
\end{lemma}
 Note that the maps in Lemma~\ref{lem:hopf_algebroid} are not compatible with the operations $\widetilde{\delta}$, since
 they depend on the orientation, so one has to use the formulas for $\delta$ on the
 $f_u$ and $g_u$ generators from Section~\ref{sec:gen_nygaard}.
 Alternatively, one can consider different
 versions of $\widetilde{\delta}$ and how to translate them into each other. In the
 algorithm we have implemented, we use an alternative approach to compute the map $\nabla$ from
 Definition~\ref{def:nabla}; see Section~\ref{sec:connection}. For the proof of Lemma
 \ref{lem:hopf_algebroid}, it
 will suffice to determine the structure maps on the associated graded of the
 $\F$-filtration, which we do below in Lemma~\ref{lem:hopf_algebroid_graded}.

Finally, we describe the comodule structure on $\Prismhat_{R/W(k)\llbracket
z_0\rrbracket}\{i\}$ and $\Prismhat_{R/W(k)\llbracket z_0\rrbracket}^{(1)}\{i\}$. Recall
from Lemma~\ref{lem:filtered-crystalline-bk}
that $\Prismhat_{R/W(k)\llbracket z_0\rrbracket}\{i\}$ is a free
$\Prismhat_{R/W(k)\llbracket z_0\rrbracket}$-module on an element $s^i$ with $s$ determined
as the unique filtered-crystalline Breuil--Kisin orientation with $\varphi(s) = \frac{s}{E(z_0)}$.
Note that this implies that
$\Prismhat_{R/W(k)\llbracket z_0\rrbracket}^{(1)}\{i\}$ is a free
$\Prismhat_{R/W(k)\llbracket z_0\rrbracket}^{(1)}$-module on $w(s^i)$. 

\begin{lemma}\label{lem:coaction_unit}
    The coaction 
  \[
    \Prismhat_{R/W(k)\llbracket z_0\rrbracket}\{i\}  \to 
    \Prismhat_{R/W(k)\llbracket z_0,z_1\rrbracket}\{i\}  \simeq 
    \Prismhat_{R/W(k)\llbracket z_0\rrbracket}\{i\} \otimes_{\Prismhat_{R/W(k)\llbracket
    z_0\rrbracket}} \Prismhat_{R/W(k)\llbracket z_0,z_1\rrbracket}
  \]
  takes $s^i \mapsto v^i s^i$ where
  \[
    v =\prod_{r\geq 0} \varphi^r(u).
  \]
\end{lemma}
\begin{proof}
  The map is induced by $z_0\mapsto z_1$. So it takes the filtered-crystalline
    Breuil--Kisin orientation $s$ with $d_s = E(z_0)$ to the filtered-crystalline
    Breuil--Kisin orientation $s_R$ with $d_{s_R} = E(z_1)$. Since
  \[
    u = \frac{E(z_1)}{E(z_0)}
  \]
  we have $d_{s_R} = u d_s$, and so
  \[
    s \mapsto v s,
  \]
    with $v =\prod_{r\geq 0} \varphi^r(u)$ by Lemma~\ref{lem:bk-factor-general}.
\end{proof}

\begin{lemma}\label{lem:coaction_unit_twist}
    The coaction 
  \[
    \Prismhat_{R/W(k)\llbracket z_0\rrbracket}^{(1)}\{i\}  \to 
    \Prismhat_{R/W(k)\llbracket z_0,z_1\rrbracket}^{(1)}\{i\}  \simeq 
    \Prismhat_{R/W(k)\llbracket z_0\rrbracket}^{(1)}\{i\}
    \otimes_{\Prismhat_{R/W(k)\llbracket z_0\rrbracket}^{(1)}} \Prismhat_{R/W(k)\llbracket z_0,z_1\rrbracket}^{(1)}
  \]
    takes $w(s^i)$ to $w(v^i) \cdot w(s^i)$.
\end{lemma}
\begin{proof}
    This follows by applying $w$ to the result of Lemma~\ref{lem:coaction_unit}.
\end{proof}

\subsection{The connection}\label{sec:connection}

In the case of $R = \cO_K$ or $\cO_K/\varpi^n$, one can show using the Hodge--Tate
comparison theorem in absolute prismatic cohomology~\cite{bhatt-lurie-apc} that
$\Prismhat_R$ and $\Prismhat_R^{(1)}$ are cohomologically
$1$-dimensional. This means that, in these cases, the relative-to-absolute descent complexes are quasi-isomorphic to the $2$-term complexes
\begin{equation}\label{eq:kernel}
    \begin{gathered}
        \Prismhat_{R/W(k)\llbracket z_0\rrbracket}\{i\}  \to \ker\left(\Prismhat_{R/W(k)\llbracket
    z_0,z_1\rrbracket}\{i\} \to \Prismhat_{R/W(k)\llbracket
    z_0,z_1,z_2\rrbracket}\{i\}\right),\\
  \Prismhat^{(1)}_{R/W(k)\llbracket z_0\rrbracket}\{i\}  \to
    \ker\left(\Prismhat^{(1)}_{R/W(k)\llbracket z_0,z_1\rrbracket}\{i\} \to
    \Prismhat^{(1)}_{R/W(k)\llbracket z_0,z_1,z_2\rrbracket}\{i\}\right)
    \end{gathered}
\end{equation}
obtained as their good truncations $\tau^{\leq 1}$.
In this section, we want to produce a more explicit description of the second term, reproducing the statement about the cohomological dimension along the way.

It follows from Proposition~\ref{prop:f_filtered_statement_for_prism}
that $\Prismhat_{R/W(k)\llbracket z_0,z_1\rrbracket}$ is $\F$-completely free as a
$\Prismhat_{R/W(k)\llbracket z_0\rrbracket}$-module, on generators $\prod
\delta^u(b)^{e_u}$ with $e_u<p$. Similarly, $\Prismhat^{(1)}_{R/W(k)\llbracket
z_0,z_1\rrbracket}$ is $\F$-completely free as a $\Prismhat^{(1)}_{R/W(k)\llbracket z_0\rrbracket}$-module, on generators $\prod g_u^{e_u}$ with $e_u<p$.

\begin{definition}\label{def:nabla}
  We define maps
  \begin{gather*}
    \theta\colon \F^{\geq \star}\Prismhat_{R/W(k)\llbracket z_0,z_1\rrbracket}\{i\}\to
      \F^{\geq \star-1}\Prismhat_{R/W(k)\llbracket z_0\rrbracket}\{i-1\}=: \F^{\geq
      \star} \Prismhat_{R/W(k)\llbracket z_0\rrbracket}^\nabla\{i\}\\
    \theta\colon \F^{\geq \star}\N^{\geq \star}\Prismhat^{(1)}_{R/W(k)\llbracket z_0,z_1\rrbracket}\{i\}\to \F^{\geq \star-1}\N^{\geq \star-1}\Prismhat^{(1)}_{R/W(k)\llbracket z_0\rrbracket}\{i-1\}=: \F^{\geq \star}\N^{\geq \star}\Prismhat^{(1),\nabla}_{R/W(k)\llbracket z_0\rrbracket}\{i\}
  \end{gather*}
  as the $\Prismhat_{R/W(k)\llbracket z_0\rrbracket}$-linear map taking $bs^i\mapsto
    s^{i-1}$ and all other basis monomials $\prod \delta^u(b)^{e_u}s^i\mapsto 0$ for the
    first line, and as the $\Prismhat^{(1)}_{R/W(k)\llbracket z_0\rrbracket}$-linear map
    taking $g_0 \cdot w(s^i)\mapsto w(s^{i-1})$ and all other basis
    monomials $\prod g_u^{e_u} \cdot w(s^i)\mapsto 0$ for the second line. We
    then define maps
  \begin{gather*}
    \nabla\colon \F^{\geq \star}\Prismhat_{R/W(k)\llbracket z_0\rrbracket}\{i\}\to
      \F^{\geq \star} \Prismhat_{R/W(k)\llbracket z_0\rrbracket}^\nabla\{i\}\\
    \nabla\colon \F^{\geq \star}\N^{\geq \star}\Prismhat^{(1)}_{R/W(k)\llbracket z_0\rrbracket}\{i\}\to  \F^{\geq \star}\N^{\geq \star}\Prismhat^{(1),\nabla}_{R/W(k)\llbracket z_0\rrbracket}\{i\}
  \end{gather*}
  as the composite of the first differential $\eta_R-\eta_L$ of the descent complexes with $\theta$.
\end{definition}

\begin{remark}
    Note that we define the $\Prism^\nabla$ terms with shifted $\F$ and Nygaard filtrations.
    Additionally, this where the $\F$-completion of relative prismatic is necessary as
    the map $\theta$ is defined only with respect to an $\F$-complete basis.
\end{remark}

In Sections~\ref{sec:hopf_operations} and~\ref{sec:proof_of_nabla}, we will prove the following lemma.

\begin{lemma}
    There are commutative diagrams
  \label{lem:nablaterm}
  \[
    \begin{tikzcd}
      \F^{\geq \star}\Prismhat_{R/W(k)\llbracket z_0\rrbracket}\{i\}\rar\dar{\id} &
        \F^{\geq \star}\Prismhat_{R/W(k)\llbracket z_0,z_1\rrbracket}\{i\} \rar\dar{\theta} & \ldots\\
      \F^{\geq \star}\Prismhat_{R/W(k)\llbracket z_0\rrbracket}\{i\}\rar{\nabla} &
        \F^{\geq \star}\Prismhat^{\nabla}_{R/W(k)\llbracket z_0\rrbracket}\{i\} \rar & 0\\
    \end{tikzcd}
  \]
  and
  \[
    \begin{tikzcd}
      \F^{\geq \star}\N^{\geq \star}\Prismhat^{(1)}_{R/W(k)\llbracket z_0\rrbracket}\{i\}\rar\dar{\id} &  \F^{\geq \star}\N^{\geq \star}\Prismhat^{(1)}_{R/W(k)\llbracket z_0,z_1\rrbracket}\{i\} \rar\dar{\theta} & \ldots\\
      \F^{\geq \star}\N^{\geq \star}\Prismhat^{(1)}_{R/W(k)\llbracket
        z_0\rrbracket}\{i\}\rar{\nabla} &  \F^{\geq \star}\N^{\geq
        \star}\Prismhat^{(1),\nabla}_{R/W(k)\llbracket z_0\rrbracket}\{i\} \rar & 0,\\
    \end{tikzcd}
  \]
  where the top rows in each case are the cochain complexes associated to the
    relative-to-absolute descent diagrams, and
  the vertical map induces a quasi-isomorphism between the rows.
\end{lemma}
Note that the upper horizontal map is not $\Prismhat^{(1)}_{R/W(k)\llbracket
z_0\rrbracket}$-linear, so the bottom horizontal map is also not a linear map.
Additionally, the next
differential in the upper complex is not linear, so while the above lemma implies that
the restriction of $\theta$ defines an isomorphism between the kernel term in Equation~\ref{eq:kernel} and $\Prismhat^{(1),\nabla}_{R/W(k)\llbracket z_0\rrbracket}$,
that isomorphism is not a module map with respect to any a priori defined module
structure on the kernel. In fact, the whole identification is non-canonical in the sense
that if we switch the roles of $z_0$ and $z_1$, we obtain different maps $\theta$ and
$\nabla$. It is an interesting problem to give $\nabla$ and its codomain a more invariant
interpretation, but we will not pursue this here. A more canonical version of this type
of connection exists in the $q$-de~Rham setting, see \cite{gros-stum-quiros}.

\begin{corollary}
  \label{cor:syntomicsquare}
  The syntomic complex $\bZ_p(i)(R)$ with its $\F$-filtration identifies with the total fiber of a square
  \[
    \begin{tikzcd}
      \F^{\geq \star} \N^{\geq i} \Prismhat^{(1)}_{R/W(k)\llbracket z_0\rrbracket}\{i\}\rar{\N^{\geq i}\nabla}\dar{\can-\varphi} & \F^{\geq\star}\N^{\geq i}\Prismhat^{(1),\nabla}_{R/W(k)\llbracket z_0\rrbracket}\{i\} \dar{\can - \varphi^\nabla}\\
      \F^{\geq \star}\Prismhat^{(1)}_{R/W(k)\llbracket z_0\rrbracket}\{i\}\rar{\nabla} & \F^{\geq\star}\Prismhat^{(1),\nabla}_{R/W(k)\llbracket z_0\rrbracket}\{i\}\\
    \end{tikzcd}
  \]
for some map $\varphi^\nabla$. (Compare Definition \ref{def:nablamaps} below.)
\end{corollary}

\begin{proof}
  By definition and the fact that prismatic cohomology is $1$-dimensional, the syntomic
    complex identifies with the total fiber of
  \[
    \begin{tikzcd}
      \F^{\geq \star}\N^{\geq i}\Prismhat^{(1)}_{R/W(k)\llbracket z_0\rrbracket}\{i\}\rar \dar{\can-\varphi} &  \F^{\geq \star}\N^{\geq \star} K\dar{\can-\varphi}\\
      \F^{\geq \star}\Prismhat^{(1)}_{R/W(k)\llbracket z_0\rrbracket}\{i\}\rar &  \F^{\geq \star}K,
    \end{tikzcd}
  \]
  where $K$ denotes the kernel of the second differential in the descent complex. By
    Lemma \ref{lem:nablaterm}, the restriction of $\theta$ to $K$ is an isomorphism. By
    transporting the maps $\can$ and $\varphi$ along $\theta|_K$, we thus obtain the
    claim. Note that although $\theta$ is not a module map, it commutes with $\can$ (see
    Lemma~\ref{lem:can_commutes}), so
    $\can$ appears unchanged in the square above.
\end{proof}

\begin{remark}
    The map $\varphi^\nabla$ can in fact be characterized uniquely by commutativity of
    the square above, since we will see that $\nabla$ is $\F$-completely a rational
    isomorphism, and all involved terms are torsion-free. In the algorithmic application
    of Corollary \ref{cor:syntomicsquare}, we can therefore avoid explicitly computing
    $K$ and $\theta$.
\end{remark}

\subsection{Generalities about Hopf algebroid operations}\label{sec:hopf_operations}

In this section, let $\Gamma=(\Gamma_0,\Gamma_1)$ be a Hopf algebroid where $\Gamma_1$ is free as a left $\Gamma_0$-module.

\begin{definition}
  For a left $\Gamma_0$-linear map $\theta\colon \Gamma_1\to \Gamma_0$, we obtain a transformation
  \[
    \nabla_\theta\colon M \to M
  \]
  natural in right $\Gamma$-comodules $M$, as the composite
  \[
      M\to M\otimes_{\Gamma_0} \Gamma_1 \xto{\id\otimes \theta} M.
  \]
\end{definition}

\begin{lemma}
  The above construction determines an equivalence between the set of left linear homomorphisms $\Gamma_1\to \Gamma_0$ and the set of natural endomorphisms
  \[
    M\to M
  \]
  of the forgetful functor from right $\Gamma$-comodules to abelian groups.
\end{lemma}
\begin{proof}
  We may evaluate a natural endomorphism $\nabla$ on $\Gamma_1$ to obtain a map
  \[
    \nabla\colon \Gamma_1\to \Gamma_1
  \]
  which is a \emph{left} $\Gamma_0$-module homomorphism by naturality, as the left
    $\Gamma_0$-action on $\Gamma_1$ is through right comodule morphisms. We may
    postcompose with $\varepsilon\colon \Gamma_1\to \Gamma_0$ to obtain a left linear
    homomorphism $\theta\colon \Gamma_1\to \Gamma_0$. We claim those constructions are
    inverse to each other. Starting with a left linear homomorphism $\theta\colon \Gamma_1\to \Gamma_0$, passing to $\nabla_\theta$ and forming the composite as above, we recover $\theta$, since the diagram
  \[
    \begin{tikzcd}
      \Gamma_1 \ar[rd]\ar[rrd,"\nabla_\theta"]\ar[rdd,"\id",swap] & & \\
         &
         \Gamma_1\otimes_{\Gamma_0}\Gamma_1\rar["\id\otimes\theta",swap]\dar{\varepsilon\otimes\id} & \Gamma_1\dar{\varepsilon}\\
         & \Gamma_1\rar{\theta} & \Gamma_0
    \end{tikzcd}
  \]
  commutes. Conversely, if we start with a natural endomorphism $\nabla$ and obtain a left linear
    $\theta\colon \Gamma_1\to \Gamma_0$ as above, $\nabla=\nabla_\theta$, since the following diagram commutes:
  \[
    \begin{tikzcd}
      M \rar{\nabla}\dar & M\dar\ar[rdd,"\id"] & \\
      M\otimes_{\Gamma_0} \Gamma_1 \rar{\nabla =
        \id\otimes\nabla}\ar[rrd,"\id\otimes\theta",swap] & M\otimes_{\Gamma_0}\Gamma_1
        \ar[rd,"\id\otimes\varepsilon"]& \\
      & & M.
    \end{tikzcd}
  \]
  Here $\nabla$ on $M\otimes_{\Gamma_0}\Gamma_1$ agrees with $\id_M\otimes \nabla$ by naturality of $\nabla$, since $(-)\otimes_{\Gamma_0} \Gamma_1$ is a colimit-preserving functor from right $\Gamma_0$-modules to right $\Gamma$-comodules.
\end{proof}

\begin{example}
    For $b_i$ a left $\Gamma_0$-basis of $\Gamma_1$ (possibly in some completed sense), $\theta^i\colon \Gamma_1\to \Gamma_0$ the dual basis homomorphisms, and $x\in \Gamma_1$, we have
  \[
    \Delta(x) = \sum_i \nabla_{\theta^i}(x) \otimes b_i.
  \]
  Similarly, for $x\in \Gamma_0$, we have
  \[
    \eta_R(x) = \sum_i \nabla_{\theta^i}(x) \cdot b_i.
  \]
\end{example}

\begin{example}
  If $\Gamma_1$ is a free divided power algebra over $\Gamma_0$ on a generator $y$
    which is primitive (meaning that $\Delta(y)=y\otimes 1 + 1\otimes y$), then 
  \[
    \Delta(y^{[n]}) = \sum_i y^{[i]} \otimes y^{[n-i]},
  \]
  so $\nabla_\theta(y^{[n]}) = y^{[n-1]}$ when $\theta\colon \Gamma_1\to \Gamma_0$
    is the dual basis element to $y$ with respect to the divided power basis, which is
    to say the map
    that takes $y\mapsto 1$ and  $y^{[n]}\mapsto 0$ for $n\neq 1$. In particular,
    $\varepsilon \circ \nabla_{\theta}^n\colon \Gamma_1 \to \Gamma_0$ is the dual basis
    element to $y^{[n]}$, hence after a suitable completion, \emph{every} $\nabla$ may
    be expressed as a power series in the chosen $\nabla_\theta$. We also have
  \[
    \eta_R(x) = \sum_n \nabla_{\theta}^n(x) \cdot y^{[n]},
  \]
  for $x\in \Gamma_0$, and evaluating $\eta_R(xy)$ two ways shows a Leibniz rule
  \[
    \nabla_\theta(xy) = \nabla_\theta(x) y + x\nabla_\theta(y).
  \]
  Since the sequence
    \begin{equation}\label{eq:nabla}
0\to \Gamma_0\xto{\eta_L} \Gamma_1\xto{\nabla_\theta} \Gamma_1\to 0
    \end{equation}
  is a short exact sequence of left $\Gamma$-comodules (by direct inspection), applying
  the two-side cobar complex
    $\cB_\Gamma(M,-)$ to~\eqref{eq:nabla} yields a fiber sequence
    \[\cB_\Gamma(M,\Gamma_0)\to
    M \xto{\nabla_\theta} M.
  \]
\end{example}

\begin{remark}
    We will see below that the situation of the descent Hopf algebroid behaves as a
    deformation of this divided power situation: the $\nabla$ from Lemma~\ref{lem:nablaterm}
    is indeed of the form discussed in this section, and while we do not have a divided
    power structure, it still generates all operations in a suitably completed sense, and
    identifies the limit of the cobar complex of $M$ with the fiber of $\nabla$ on $M$
    (Lemma~\ref{lem:nablaterm} is a special case of this). Since we work with the bases
    $\prod \delta^u(b)^{e_u}$ or $\prod g_u^{e_u}$ instead of a divided power basis,
    and the relations for $g_u^p$ for example may involve terms linear in $g_0$, the analogue of
    the Leibniz rule breaks down and instead involves higher correction terms. The precise
    behaviour is of course highly sensitive to the choice of basis, and we currently do not
    know whether there is a better choice of basis that leads to closed-form formulas.
\end{remark}

\subsection{Proof of Lemma \ref{lem:nablaterm}}\label{sec:proof_of_nabla}

We need two preliminary lemmas. The first gives a set of generators of the
$\F$-associated graded pieces of the relevant prismatic envelopes
which are related to, but not the same as, those given in Section~\ref{sec:f-filtration}.

\begin{lemma}\label{lem:hopf_algebroid_graded}
  \begin{enumerate}
      \item[{\em (i)}] In the Hopf algebroid $(\gr^\star_\F \Prism_{R/W(k)\llbracket z_0\rrbracket}, \gr^\star_\F \Prism_{R/W(k)\llbracket z_0,z_1\rrbracket})$,
          the element $b$ satisfies $\Delta(b) = b\otimes 1 + 1\otimes b$, and admits divided powers.
          Furthermore, the divided powers $b^{[k]}$ provide a basis of
          $\gr^\star_\F \Prism_{R/W(k)\llbracket z_0,z_1\rrbracket}$ as a $\gr^\star_\F\Prism_{R/W(k)\llbracket z_0\rrbracket}$-module.
      \item[{\em (ii)}] In the Hopf algebroid $(\gr^\star_\F \Prismhat^{(1)}_{R/W(k)\llbracket z_0\rrbracket},\gr^\star_\F \Prismhat^{(1)}_{R/W(k)\llbracket z_0,z_1\rrbracket})$, the element $g_0$ satisfies $\Delta(g_0) = g_0\otimes 1 + 1\otimes g_0$, and admits divided powers.
          Furthermore, the divided powers $g_0^{[k]}$ provide a basis of
          $\gr^\star_\F\Prismhat_{R/W(k)\llbracket z_0,z_1\rrbracket}^{(1)}$ as a $\Prismhat^{(1)}_{R/W(k)\llbracket z_0\rrbracket}$-module.
  \end{enumerate}
\end{lemma}

\begin{proof}
  We established that $\Delta(g_0) = g_0\otimes 1 + 1\otimes g_0$ and
  \[
      \Delta(b) = b\otimes 1 + u\otimes b = b\otimes 1 + 1\otimes b \pmod{\F^{\geq
      2}\Prism_{R/W(k)\llbracket z_0,z_1,z_2\rrbracket}}
  \]
    in Lemmas~\ref{lem:untwisted_hopf_algebroid} and~\ref{lem:hopf_algebroid}.
 For the statements about divided powers, we first do the Frobenius-twisted case. Recall
    from Lemma~\ref{lem:explicitrelations} that 
  \[
    g_u^p = (-p+d^{p^{u+1}}\lambda_u) g_{u+1} + d^{p^{u+1}} R'_{u}.
  \]
  As $R_0' = \frac{\delta(z_1-z_0)}{\delta(d)}$ and
  \[
      \delta(z_1-z_0) = \frac{z_1^p - z_0^p - (z_1-z_0)^p}{p} = \frac{(z_0+g_0)^p - z_0^p - g_0^p}{p},
  \]
  we have that $R_0'$ is divisible by $g_0$, so in particular it is a polynomial in the $g_u$ without constant term.
    The explicit formula for $\delta(g_0)$ then shows that $\delta(g_0)$ is a polynomial
    in the $g_u$ without constant term, and then the formula for $R_1'$ shows that $R_1'$
    is a polynomial in the $g_u$ without constant term. Inductively, this shows that all
    $R'_u$ and $\delta(g_u)$ are polynomials in the $g_u$ without constant terms. So
    $d^{p^{u+1}-1} R'_u$ is still in $\N^{\geq p^{u+1}}\Prismhat^{(1)}_{R/W(k)\llbracket z_0,z_1\rrbracket}$, and we have
  \[
    g_u^p = (-p + p^{p^{u+1}}\lambda_u) g_{u+1} + p d^{p^{u+1}-1} R_u'
  \]
  in $\gr^{p^{u+1}}_\F \N^{\geq p^{u+1}} \Prismhat^{(1)}_{R/W(k)\llbracket
    z_0,z_1\rrbracket}$. So $g_u$ has a $p$th divided power in the $\F$-associated
    graded ring (unique because of torsion-freeness), and it agrees with
  \[
    g_u^{[p]} = (-1 + p^{p^{u+1}-1}\lambda_u) g_{u+1} + d^{p^{u+1}-1} R'_u.
  \]
  As $R'_u$ is a polynomial in $g_0,\ldots,g_u$, this equation also proves that $g_u^{[p]}$ again has a $p$th divided power, and we deduce inductively that $g_0$ has all divided powers. We also see that $g_0^{[p^u]}$ differs from a unit multiple of $g_{u}$ by some polynomial in $g_0,\ldots,g_{u-1}$, and so it follows that the $g_0^{[p^u]}$ form a basis as claimed.

  We obtain the statement for the non-Frobenius twisted version by applying the relative divided Frobenius.
\end{proof}

\begin{remark}
    Because the relative-to-absolute Hopf algebroid is compatible with the $\F$-filtration and because the
    associated graded pieces are crystalline prisms, the existence of the divided power
    structure on the Frobenius twists follows from the crystalline comparison theorem.
\end{remark}

\begin{lemma}
  \label{lem:resolution}
  If $(\Gamma_0,\Gamma_1)$ denotes either the non-Frobenius twisted or Frobenius twisted descent
    Hopf algebroid and $\theta\colon \Gamma_1\to \Gamma_0$ as in Lemma \ref{lem:nablaterm}, then
    the sequence
  \[
    0\to   \Gamma_0 \xto{\eta_L} \Gamma_1 \xto{\nabla_\theta} \Gamma_1 \to 0
  \]
  is $\F$-completely exact.
\end{lemma}
\begin{proof}
  It suffices to check this after passage to the $\F$-associated graded. In that case, we let $y=b$ in the non-Frobenius twisted case or $y=g_0$ in the Frobenius twisted case, and recall that $y$ is primitive and has divided powers. We have
  \[
    \Delta(y^{[k]}) = \sum_j y^{[j]}\otimes y^{[k-j]}
  \]
  and $\theta(y)=1$, and so
  \[
      \nabla_\theta(y^{[k]}) = y^{[k-1]} + \sum_{k\geq 2} y^{[j]} \cdot\theta(y^{[k-j]}).
  \]
  We also have $\theta(1)=0$ and thus $\nabla_\theta(1)=0$. Since $\nabla_\theta$ on
    $\Gamma_1$ is a left $\Gamma_0$ module map, and the above shows that with respect to
    the $y^{[k]}$-basis it is strictly upper triangular with ones on the first off
    diagonal, the kernel is precisely $\Gamma_0\cdot 1$, as claimed.
\end{proof}

\begin{proof}[Proof of Lemma \ref{lem:nablaterm}]
    The claim now follows by applying the two-sided coBar construction
    $\cB_\Gamma(M,-)$ with the trivial right comodule $\Gamma_0$ to the short exact sequence from Lemma \ref{lem:resolution}. More precisely, consider the map
  \[
    \begin{tikzcd}
        M\rar\dar & M\otimes_{\Gamma_0} \Gamma_1 \rar\dar{1\otimes\theta} &
        M\otimes_{\Gamma_0}\Gamma_1\otimes_{\Gamma_0}\Gamma_1\rar\dar & \ldots\\
      M\rar{\nabla_\theta} & M\rar & 0
    \end{tikzcd}
  \]
  of complexes, natural in a right $\Gamma$-comodule $M$. It is a quasi-isomorphism if $M$ is of the form $N\otimes_{\Gamma_0} \Gamma_1$ for a flat right $\Gamma_0$-module $N$, since then both complexes are augmented by $N$, and the augmentation is a quasi-isomorphism for the top one since the cobar complex for $\Gamma_1$ has an extra degeneracy, and for the bottom one cause of Lemma \ref{lem:resolution}. Since we may resolve $\Gamma_0\{i\}$ as right $\Gamma$-comodule by such induced comodules, the statement then follows for $M=\Gamma_0\{i\}$ as well, which is what we wanted to prove.
\end{proof}

\subsection{Properties of $\nabla$}

We now analyze compatibility of $\nabla$ with the structure maps. Observe that Lemma
\ref{lem:nablaterm} asserts that $\theta$ determines an isomorphism
\[
  \ker\big(\F^{\geq\star} \Prismhat_{R/W(k)\llbracket z_0,z_1\rrbracket} \{i\}\to
  \F^{\geq\star} \Prismhat_{R/W(k)\llbracket z_0,z_1,z_2\rrbracket}\{i\}\big) \cong
  \F^{\geq\star} \Prismhat^\nabla_{R/W(k)\llbracket z_0\rrbracket}\{i\},
\]
and analogously for the Frobenius-twisted version, in which case it is also compatible with the Nygaard filtration (including the shift on the $\nabla$-term).

\begin{definition}
  \label{def:nablamaps}
For $R=\cO_K/\varpi^n$ or $\cO_K$ and the natural structure maps $\can, \varphi,
    \varphi_{R/W(k)\llbracket z_0\rrbracket}, w$, where $w\colon \Prismhat_{R/A}\to
    \Prismhat_{R/A}^{(1)}$ is the canonical $\varphi$-semilinear map, we let 
  \begin{gather*}
    \can^\nabla\colon \F^{\geq\star}\N^{\geq j}\Prismhat^{(1),\nabla}_{R/W(k)\llbracket
      z_0\rrbracket}\{i\}\to \F^{\geq\star}\Prismhat^{(1),\nabla}_{R/W(k)\llbracket
      z_0\rrbracket}\{i\},\\
    \varphi^\nabla\colon \F^{\geq\star}\N^{\geq
      i}\Prismhat^{(1),\nabla}_{R/W(k)\llbracket z_0\rrbracket}\{i\}\to \F^{\geq
      p\cdot\star}\Prismhat^{(1),\nabla}_{R/W(k)\llbracket z_0\rrbracket}\{i\},\\
    \varphi_{R/W(k)\llbracket z_0\rrbracket}^\nabla\colon \F^{\geq\star}\N^{\geq
      i}\Prismhat^{(1),\nabla}_{R/W(k)\llbracket z_0\rrbracket}\{i\}\to
      \F^{\geq\star}\Prismhat^{\nabla}_{R/W(k)\llbracket z_0\rrbracket}\{i\},\\
    w^\nabla\colon \F^{\geq\star}\Prismhat^{\nabla}_{R/W(k)\llbracket z_0\rrbracket}\{i\}\to \F^{\geq p\cdot\star}\Prismhat^{(1),\nabla}_{R/W(k)\llbracket z_0\rrbracket}\{i\}
  \end{gather*}
  denote the corresponding maps obtained by transporting them along the above isomorphism $\theta$.
\end{definition}

\begin{lemma}\label{lem:can_commutes}
  For $R=\cO_K/\varpi^n$ or $\cO_K$, $\nabla$ commutes with $\can$. In other words, the following diagram commutes:
  \[
    \begin{tikzcd}
      \F^{\geq\star}\N^{\geq j} \Prismhat^{(1)}_{R/W(k)\llbracket z_0\rrbracket}\{i\}\dar{\can} \rar{\N^{\geq j}\nabla} & \F^{\geq\star}\N^{\geq j}\Prismhat^{(1),\nabla}_{R/W(k)\llbracket z_0\rrbracket}\{i\}\dar{\can^\nabla}\rar[equal] & \F^{\geq\star-1}\N^{\geq j-1} \Prismhat^{(1)}_{R/W(k)\llbracket z_0\rrbracket}\{i-1\}\dar{\can}\\
      \F^{\geq\star} \Prismhat^{(1)}_{R/W(k)\llbracket z_0\rrbracket}\{i\} \rar{\nabla}
        &\F^{\geq\star} \Prismhat^{(1),\nabla}_{R/W(k)\llbracket
        z\rrbracket}\{i\}\rar[equal] & \F^{\geq\star-1}\Prismhat^{(1)}_{R/W(k)\llbracket
        z\rrbracket}\{i-1\}.
    \end{tikzcd}
  \]
\end{lemma}
\begin{proof}
  The left square commutes by definition. For the right square, commutativity unwinds to the claim that $\theta$ commutes with $\can$, which can be checked on basis elements.
\end{proof}

\begin{lemma}
  For $R=\cO_K/\varpi^n$ or $\cO_K$, and $R' = \cO_K/\varpi^m$, with $R\to R'$ the
    reduction map, $\nabla$ commutes with the induced map on prismatic cohomology, i.e., the following diagram commutes:
  \[
    \begin{tikzcd}
      \F^{\geq\star}\N^{\geq j} \Prismhat_{R/W(k)\llbracket z_0\rrbracket}\{i\}\dar{\red}
        \rar{\N^{\geq j}\nabla} & \F^{\geq\star}\N^{\geq j}\Prismhat^{\nabla}_{R/W(k)\llbracket
        z_0\rrbracket}\{i\}\dar{\red^\nabla}\rar[equal] & \F^{\geq\star-1}\N^{\geq j-1}
        \Prismhat_{R/W(k)\llbracket z_0\rrbracket}\{i-1\}\dar{\red}\\
      \F^{\geq\star} \Prismhat_{R/W(k)\llbracket z_0\rrbracket}\{i\} \rar{\nabla}
        &\F^{\geq\star} \Prismhat^{\nabla}_{R/W(k)\llbracket z_0\rrbracket}\{i\}\rar[equal] &
        \F^{\geq\star-1}\Prismhat_{R/W(k)\llbracket z_0\rrbracket}\{i-1\}.
    \end{tikzcd}
  \]
  The analogous statement holds for the Frobenius-twisted version.
\end{lemma}
\begin{proof}
  This again reduces to checking that $\theta$ commutes with the induced maps, which
    follows since it preserves the $\prod g_u^{e_u}$-basis expressing
    $\Prismhat_{R/W(k)\llbracket z_0,z_1\rrbracket}$ as an $\F$-completely free
    $\Prismhat_{R/W(k)\llbracket
    z_0\rrbracket}$-module.
\end{proof}

\begin{lemma}
  \label{lem:nablafrobenius}
  For $R=\cO_K/\varpi^n$ or $\cO_K$, $\nabla$ commutes with the relative Frobenius,
    i.e., the following diagram commutes:
  \[
    \begin{tikzcd}
      \F^{\geq\star}\N^{\geq i} \Prismhat^{(1)}_{R/W(k)\llbracket z_0\rrbracket}\{i\}\dar{\varphi_{R/W(k)\llbracket z_0\rrbracket}} \rar{\nabla} & \F^{\geq\star}\N^{\geq i}\Prismhat^{(1),\nabla}_{R/W(k)\llbracket z_0\rrbracket}\{i\}\dar{\varphi_{R/W(k)\llbracket z_0\rrbracket}^\nabla}\rar[equal] & \F^{\geq\star}\N^{\geq i-1} \Prismhat^{(1)}_{R/W(k)\llbracket z_0\rrbracket}\{i-1\}\dar{\varphi_{R/W(k)\llbracket z_0\rrbracket}}\\
      \F^{\geq\star} \Prismhat_{R/W(k)\llbracket z_0\rrbracket}\{i\} \rar{\nabla}
        &\F^{\geq\star} \Prismhat^{\nabla}_{R/W(k)\llbracket z_0\rrbracket}\{i\}\rar[equal]
        & \F^{\geq\star}\Prismhat_{R/W(k)\llbracket z_0\rrbracket}\{i-1\}.
    \end{tikzcd}
  \]
\end{lemma}
\begin{proof}
  Again, the left square commutes by definition. For the right square, we again check
    that $\theta$ commutes with the relative Frobenius, which comes from the fact that
    the divided Frobenius takes $g_u\mapsto \delta^u b$, i.e., it respects the monomial bases involved in the definition of $\theta$ on the Frobenius-twisted and the non Frobenius-twisted versions of relative prismatic cohomology.
\end{proof}

\begin{remark}
    The corresponding statement for the absolute Frobenius (and the canonical map
    $w\colon \Prismhat \to \Prismhat^{(1)}$) is false: $\varphi^{\nabla}$ and $w^\nabla$ do not agree with $\varphi$ and $w$.
\end{remark}

\subsection{Isogeny properties}

In this section, we determine the action of $\can$, $\nabla$ and the map induced by $\cO_K\to \cO_K/\varpi^n$ on the $\F$-associated graded. We obtain that various maps are rationally invertible, and explicit formulas for their determinants. We note that these results are very similar in nature to \cite{sulyma} and could in fact be reduced to this using the idea of crystalline degeneration.

Recall that $\gr^j_\F\Prismhat^{(1)}_{(\cO_K/\varpi^n)/W(k)\llbracket z_0\rrbracket}$ is
a free $W(k)$-module on a generator $z_0^k \prod f_u^{e_u}$ where $\sum p^u e_u =
\lfloor \frac{j}{n}\rfloor$ with all $e_u<p$, and $k$ is the remainder of $j$ mod $n$.
Similarly, $\gr^j\N^{\geq i}_\F\Prismhat^{(1)}_{(\cO_K/\varpi^n)/W(k)\llbracket
z_0\rrbracket}$ is a free $W(k)$-module on a generator $\widetilde{d}^lz_0^k \prod
f_u^{e_u}$ where $\sum p^u e_u = \lfloor \frac{j}{n}\rfloor$ with all $e_u<p$, $k$ is
the remainder of $j$ mod $n$, and $l=\max(i - \sum e_up^u, 0)$. Analogously, $\gr^k_\F
\Prismhat^{(1)}_{\cO_K/W(k)\llbracket z_0\rrbracket}$ and $\gr^k_\F \N^{\geq
i}\Prismhat^{(1)}_{\cO_K/W(k)\llbracket z_0\rrbracket}$ are free on $z_0^k$ and
$\widetilde{d}^iz_0^k$ respectively. Generators for the Breuil--Kisin twisted terms are obtained by formally multiplying with the orientation $w(s^i)$. On the $\nabla$-terms, we have the same generators with a shift in indexing.

\begin{lemma}
For $R=\cO_K/\varpi^n$, the $\can$ and $\can^\nabla$ maps are described on the $\F$-associated graded as follows with respect to the chosen bases:
  \begin{enumerate}
      \item[{\em (1)}]  \[
        \gr^j_\F\N^{\geq i}\Prismhat^{(1)}_{R/W(k)\llbracket z_0\rrbracket}\to 
       \gr^j_\F\N^{\geq i}\Prismhat^{(1)}_{R/W(k)\llbracket z_0\rrbracket}
      \]
      is multiplication with $p^{\max(i - \lfloor \frac{j}{n}\rfloor,0)}$;
  \item[{\em (2)}]  \[
        \gr^j_\F\N^{\geq i}\Prismhat^{(1),\nabla}_{R/W(k)\llbracket z_0\rrbracket}\to 
       \gr^j_\F\N^{\geq i}\Prismhat^{(1),\nabla}_{R/W(k)\llbracket z_0\rrbracket}
      \]
      is multiplication with $p^{\max(i-1 - \lfloor \frac{j-1}{n}\rfloor,0)}$.
  \end{enumerate}
  The analogous statements hold with Breuil--Kisin twists.
\end{lemma}
\begin{proof}
  For the first statement, the basis element is $\widetilde{d}^lz_0^k \prod f_u^{e_u}$
    with $l = \max(i - \sum p^ue_u,0)$, $\sum p^ue_u = \lfloor \frac{j}{n} \rfloor$, and $k$ the residue of $j$ mod $n$. Since the basis element in the target is $z_0^k \prod f_u^{e_u}$ and $\widetilde{d} = E(z_0)$ is $p$ in the $\F$-associated graded, the claim follows. 

  The second statement follows from the first by taking into account the indexing shift from the definition of the target of $\nabla$.
\end{proof}

\begin{lemma}
  The reduction maps induced by $\cO_K\to \cO_{K/\varpi^n}$ are described on the $\F$-associated graded as follows with respect to the chosen bases:
  \begin{enumerate}
      \item[{\em (1)}]
      \[
       \gr^j_\F\Prismhat^{(1)}_{\cO_K/W(k)\llbracket z_0\rrbracket}\to 
       \gr^j_\F\Prismhat^{(1)}_{(\cO_K/\varpi^n)/W(k)\llbracket z_0\rrbracket}
      \]
      is multiplication with a unit multiple of $\lfloor \frac{j}{n}\rfloor!$ and
  \item[{\em (2)}]
      \[
       \gr^j_\F\Prismhat^{(1),\nabla}_{\cO_K/W(k)\llbracket z_0\rrbracket}\to 
       \gr^j_\F\Prismhat^{(1),\nabla}_{(\cO_K/\varpi^n)/W(k)\llbracket z_0\rrbracket}
      \]
      is multiplication with a unit multiple of $\lfloor \frac{j-1}{n}\rfloor!$.
  \end{enumerate}
  The analogous statements hold with Breuil--Kisin twists.
\end{lemma}
\begin{proof}
  On the $\F$-associated graded, $d=p$, and so $f_u^p = (-p + \lambda_u
    p^{p^{u+1}}f_{u+1})$. This exhibits $f_u^p$ as unit multiple of $pf_{u+1}$, so $z_0^n=f_0$ has all divided powers,
    and the basis monomials $f_u^{e_u}$ are unit multiples of the divided powers
    $(z_0^n)^{[s]}$. It follows that $z_0^j$ is taken to $z_0^k s! (z_0^n)^{[s]}$ where $k$ is the remainder of $j$ modulo $n$ and $s = \lfloor \frac{j}{n}\rfloor$.
    The second statement follows from the first by taking into account the indexing shift from the definition of the target of $\nabla$.
\end{proof}

\begin{lemma}
  For $\cO_K$, the map $\nabla$ acts on the $\F$-associated graded with respect to the chosen bases as follows:
  \begin{enumerate}
      \item[{\em (1)}] \[
 \nabla\colon  \gr^j_\F\Prismhat^{(1)}_{\cO_K/W(k)\llbracket z_0\rrbracket}\{i\}\to 
    \gr^j_\F\Prismhat^{(1),\nabla}_{\cO_K/W(k)\llbracket z_0\rrbracket}\{i\}
  \]
  acts by multiplication with a unit multiple of $j$ and
    \item[{\em (2)}] \[
        \nabla\colon  \gr^j_\F\N^{\geq i}\Prismhat^{(1)}_{\cO_K/W(k)\llbracket z_0\rrbracket}\{i\}\to 
      \gr^j_\F\N^{\geq i}\Prismhat^{(1),\nabla}_{\cO_K/W(k)\llbracket z_0\rrbracket}\{i\}
  \]
      for $i\geq 1$ acts by multiplication with a unit multiple of $p^{\epsilon(i,j)}j$, where 
      \[
        \epsilon = \begin{cases}
          1 \text{ if $\lfloor \frac{j}{n}\rfloor < \lceil \frac{j}{n}\rceil \leq i$}\\
          0 \text{ otherwise.}
        \end{cases}
      \]
  \end{enumerate}
\end{lemma}
\begin{proof}
  We have
  \[
      \nabla(z_0^jw(s^i)) = \theta(((z_0+g_0)^j w(v^i) - z_0^j)w(s^i)),
  \]
    i.e. $w(s^{i-1})$ times the coefficient of $g_0$ in $(z_0+g_0)^jw(v^i) - z_0^j$. Working in the $\F$-associated graded, the unit $v$ is $1$, so
  \[
      \nabla(z_0^j w(s^i)) = jz_0^{j-1}w(s^{i-1}) + \sum_{m=2}^{j} \binom{j}{m} z_0^{j-m} \theta(g_0^ms_i).
  \]
  For weight reasons, each of the summands is a multiple of $z_0^{j-1}$.
  As observed in Lemma~\ref{lem:hopf_algebroid_graded}, $g_0$ has divided powers. So each of the summands for $m\geq 2$ is divisible by $\binom{j}{m} m!$. By Kummer's theorem, this has $p$-valuation at least
  \[
    v_p(j) - v_p(m) + v_p(m!),
  \]
  which is strictly bigger than $v_p(j)$ unless $m=2$. If $p\neq 2$,
    $\theta(g_0^2s^i)=0$. If $p=2$, we deduce that $\theta(g_0^2w(s^i))$ is divisible by $d^2$ since
  \[
    g_0^2 = (-2 + \lambda_0 d^2) g_1 + d^2 R_0',
  \]
  and since $d=p$ in the $\F$-associated graded we conclude in all cases that each of
    the $m\geq 2$ summands is divisible by $pjz_0^{j-1}w(s^{i-1})$. So we conclude that
    $\nabla(z_0^jw(s^i))$ is a unit multiple of $jz_0^{j-1}w(s^{i-1})$.

  For the second statement, we use the commutative square
  \[
    \begin{tikzcd}
      \gr^j_\F\N^{\geq i}\Prismhat^{(1)}_{R/W(k)\llbracket z_0\rrbracket}\{i\}\rar{\nabla}\dar{\can}&
      \gr^j_\F\N^{\geq i}\Prismhat^{(1),\nabla}_{R/W(k)\llbracket z_0\rrbracket}\{i\}\dar{\can} \\
      \gr^j_\F\Prismhat^{(1)}_{R/W(k)\llbracket z_0\rrbracket}\{i\}\rar{\nabla} & 
      \gr^j_\F\Prismhat^{(1),\nabla}_{R/W(k)\llbracket z_0\rrbracket}\{i\},
    \end{tikzcd}
  \]
  observing that, up to units, the bottom horizontal map acts by $j$, the left vertical
    map acts by $p^{\max(i - \lfloor \frac{j}{n}\rfloor,0)}$, and the right vertical map acts by $p^{\max(i-1-\lfloor \frac{j-1}{n}\rfloor,0)}$. The difference of these exponents is at most $1$, and only is $1$ if $i> \lfloor\frac{j}{n}\rfloor$ and $i - \lfloor \frac{j}{n}\rfloor > i-1 - \lfloor \frac{j-1}{n}\rfloor$. This happens if and only if $j$ is not divisible by $n$ and $\lceil \frac{j}{n}\rceil \leq i$, i.e. if $\epsilon(i,j)\neq 1$. So the above diagram shows that the top horizontal map acts by a unit multiple of $p^{\epsilon(i,j)} j$.
\end{proof}

\begin{lemma}
  For $R=\cO_K/\varpi^n$, the map $\nabla$ acts on the $\F$-associated graded with respect to the chosen bases as follows:
  \begin{enumerate}
      \item[{\em (1)}] \[
 \nabla\colon  \gr^j_\F\Prismhat^{(1)}_{R/W(k)\llbracket z_0\rrbracket}\{i\}\to 
    \gr^j_\F\Prismhat^{(1),\nabla}_{R/W(k)\llbracket z_0\rrbracket}\{i\}
  \]
      acts by multiplication with a unit multiple of $\{j,n\}$, where
      \[
        \{j,n\} = \begin{cases} n \text{ if $n\mid j$}\\
           j\text{ otherwise;}
        \end{cases}
      \]
  \item[{\em (2)}] \[
        \nabla\colon  \gr^j_\F\N^{\geq i}\Prismhat^{(1)}_{R/W(k)\llbracket z_0\rrbracket}\{i\}\to 
      \gr^j_\F\N^{\geq i}\Prismhat^{(1),\nabla}_{R/W(k)\llbracket z_0\rrbracket}\{i\}
  \]
      for $i\geq 1$ acts by multiplication with a unit multiple of $p^{\epsilon(i,j)}\{j,n\}$.
  \end{enumerate}
\end{lemma}
\begin{proof}
 We use the commutative diagram
  \[
    \begin{tikzcd}
      \gr^j_\F \Prismhat^{(1)}_{\cO_K/W(k)\llbracket z_0\rrbracket}\{i\}\rar{\nabla}\dar{\red} &
      \gr^j_\F \Prismhat^{(1)\nabla}_{\cO_K/W(k)\llbracket z_0\rrbracket}\{i\}\dar{\red^\nabla}\\
      \gr^j_\F \Prismhat^{(1)}_{R/W(k)\llbracket z_0\rrbracket}\{i\}\rar{\nabla} & 
      \gr^j_\F \Prismhat^{(1)\nabla}_{R/W(k)\llbracket z_0\rrbracket}\{i\}
    \end{tikzcd}
  \]
  The top horizontal map acts by multiplication with a unit multiple of $j$. The left vertical map acts by $\lfloor\frac{j}{n} \rfloor!$, and the right vertical map by $\lfloor \frac{j-1}{n}\rfloor!$ due to the indexing shift. We obtain that the bottom map acts by a unit multiple of
  \[
    j \cdot \left\lfloor \frac{j-1}{n}\right\rfloor! \cdot \left(\left\lfloor \frac{j}{n}\right\rfloor!\right)^{-1}.
  \]
  The quotient of the two factorial terms is $1$ if $n\nmid j$, and $\left(\frac{j}{n}\right)^{-1}$ if $n\mid j$. In total we obtain that the bottom map acts by a unit multiple of $\{j,n\}$ as claimed.

  For the second statement, we proceed analogously with $\N^{\geq i}\Prismhat^{(1)}$.
\end{proof}

In particular, we may conclude from this that all maps in the syntomic square as well as the reduction maps are rational isomorphisms on any finite part $\F^{[a,b]}$ of the $\F$-filtration:

\begin{lemma}
  \label{lem:can_isogeny}
  For $R=\cO_K/\varpi^n$ or $R=\cO_K$, the map
  \[
    \can\colon \N^{\geq i}\Prismhat^{(1)}_{R/W(k)\llbracket z_0,\ldots,z_s\rrbracket} \{i\}
    \to \Prismhat^{(1)}_{R/W(k)\llbracket z_0,\ldots,z_s\rrbracket} \{i\}
  \]
  is a rational isomorphism on $\F^{[a,b]}$ for any integers $a\leq b$. Its cokernel has order 
  \[
      \prod_{j=a}^{\min(b,in-1)} q^{i- \lfloor \frac{j}{n}\rfloor},
  \]
  where $q=p^f$ is the order of the residue field $k$.
\end{lemma}
\begin{proof}
  This follows directly from the description of $\can$ on the $\F$-associated graded.
\end{proof}

\begin{lemma}
  \label{lem:reduction_isogeny}
  The reduction map
  \[
      \red\colon\Prismhat^{(1)}_{\cO_K/W(k)\llbracket z_0 \rrbracket}\{i\} \to
      \Prismhat^{(1)}_{(\cO_K/\varpi^n)/W(k)\llbracket z_0 \rrbracket}\{i\}
    \]
  is a rational isomorphism on $\F^{[a,b]}$ for any integers $a\leq b$. Its cokernel has
    order $\prod_{j=a}^b q^{v_p(\lfloor \frac{j}{n}\rfloor!)}$.
\end{lemma}
\begin{proof}
  This follows directly from the description of $\red$ on the $\F$-associated graded.
\end{proof}

\begin{lemma}
    \label{lem:isogeny}
    \leavevmode
    \begin{enumerate}
        \item [{\em (1)}]
  For $R=\cO_K/\varpi^n$ or $\cO_K$, the map
  \[
      \nabla\colon \Prismhat^{(1)}_{R/W(k)\llbracket z_0 \rrbracket}\{i\} \to
      \Prismhat^{(1),\nabla}_{R/W(k)\llbracket z_0 \rrbracket}\{i\}
    \]
        is a rational isomorphism on $\F^{[a,b]}$ for any integers $a\leq b$ with $a\geq 1$.
            Its cokernel has order $\prod_{j=a}^b q^{v_p\{j,n\}}$ for $R=\cO_K/\varpi^n$
            and $\prod_{j=a}^b q^{j}$ for $R=\cO_K$.
    \item[{\em (2)}]
  For $R=\cO_K/\varpi^n$ or $\cO_K$ and $i\geq 1$, the map
  \[
    \nabla\colon \N^{\geq i}\Prismhat^{(1)}_{R/W(k)\llbracket z_0 \rrbracket}\{i\} \to
        \N^{\geq i}\Prismhat^{(1),\nabla}_{R/W(k)\llbracket z_0 \rrbracket}\{i\}
  \]
        is a rational isomorphism on $\F^{[a,b]}$ for any integers $a\leq b$ with $a\geq 1$.
            Its cokernel has order $\prod_{j=a}^b q^{\epsilon(i,j)+v_p\{j,n\}}$ for
            $R=\cO_K/\varpi^n$ and $\prod_{j=a}^b q^{\epsilon(i,j)+v_p(j)}$  for $R=\cO_K$.
    \end{enumerate}
\end{lemma}
\begin{proof}
  This follows directly from the description of $\nabla$ on the $\F$-associated graded.
\end{proof}

\begin{remark}
    Lemma~\ref{lem:isogeny} can be used to derive the Angeltveit quotient result of
    Proposition~\ref{prop:angeltveit}.
\end{remark}

\begin{remark}\label{rem:complete_the_squares}
  These results together imply the following for the structure maps on $\F^{[a,b]}$, with $a\geq 1$.
\begin{enumerate}
  \item From knowledge of the map
    \[
      \nabla\colon\F^{[a,b]} \Prismhat^{(1)}_{\cO_K / W(k)\llbracket z_0\rrbracket}\{i\}\to \F^{[a,b]}\Prismhat^{(1),\nabla}_{\cO_K / W(k)\llbracket z_0\rrbracket}\{i\}
    \]
    and the maps induced by the reduction $\cO_K\to \cO_K/\varpi^n$, we may recover the map
    \[
      \nabla\colon \F^{[a,b]}\Prismhat^{(1)}_{(\cO_K/\varpi^n) / W(k)\llbracket z_0\rrbracket}\{i\}\to \F^{[a,b]}\Prismhat^{(1),\nabla}_{(\cO_K/\varpi^n) / W(k)\llbracket z_0\rrbracket}\{i\}
    \]
    (analogously for the non-Frobenius twisted version).
  \item From knowledge of the map
    \[
      \nabla\colon \F^{[a,b]}\Prismhat^{(1)}_{(\cO_K/\varpi^n) / W(k)\llbracket z_0\rrbracket}\{i\}\to \F^{[a,b]}\Prismhat^{(1),\nabla}_{(\cO_K/\varpi^n) / W(k)\llbracket z_0\rrbracket}\{i\}
    \]
    and the can maps, we may recover the map
    \[
      \nabla\colon \F^{[a,b]}\N^{\geq i}\Prismhat^{(1)}_{(\cO_K/\varpi^n) /
        W(k)\llbracket z_0\rrbracket}\{i\}\to \F^{[a,b]}\N^{\geq
        i}\Prismhat^{(1),\nabla}_{(\cO_K/\varpi^n) / W(k)\llbracket z_0\rrbracket}\{i\}.
    \]
  \item From knowledge of the maps
    \begin{gather*}
      \nabla\colon \F^{[a,b]}\Prismhat^{(1)}_{(\cO_K/\varpi^n) / W(k)\llbracket
        z\rrbracket}\{i\}\to \F^{[a,b]}\Prismhat^{(1),\nabla}_{(\cO_K/\varpi^n) /
        W(k)\llbracket z_0\rrbracket}\{i\},\\
      \nabla\colon \F^{[a,b]}\N^{\geq i}\Prismhat^{(1)}_{(\cO_K/\varpi^n) / W(k)\llbracket z_0\rrbracket}\{i\}\to \F^{[a,b]}\N^{\geq i}\Prismhat^{(1),\nabla}_{(\cO_K/\varpi^n) / W(k)\llbracket z_0\rrbracket}\{i\},
    \end{gather*}
    and the Frobenius map
    \[
      \varphi\colon \F^{[a,b]}\N^{\geq i}\Prismhat^{(1)}_{(\cO_K/\varpi^n) / W(k)\llbracket z_0\rrbracket}\{i\}\to \F^{[a,b]}\Prismhat^{(1)}_{(\cO_K/\varpi^n) / W(k)\llbracket z_0\rrbracket}\{i\},
    \]
    we may recover the map $\varphi^\nabla$.
\end{enumerate}
  So in practice, all structure maps appearing in the syntomic square (Corollary
    \ref{cor:syntomicsquare}) are determined by the  $\can$ and $\varphi$ maps for
    $\cO_K/\varpi^n$, the reduction map for $\cO_K\to \cO_K/\varpi^n$, and the map
    $\nabla$ for $\cO_K$. One advantage is that the latter depends on $\cO_K$ and the
    choice of uniformizer, but not on $n$. The other advantage is that each of these
    pieces of data requires computations in
    $\Prismhat^{(1)}_{(\cO_K/\varpi^n)/W(k)\llbracket z_0\rrbracket}$ or
    $\Prismhat^{(1)}_{\cO_K / W(k)\llbracket z_0,z_1\rrbracket}$, but never
    $\Prismhat^{(1)}_{(\cO_K/\varpi^n)/W(k)\llbracket z_0,z_1\rrbracket}$, which is
    significantly bigger.
\end{remark}

\section{The even vanishing theorem}\label{sec:evenvanishing}

In this section, we prove the following fact: $\K_{2i-2}(\cO/\varpi^n)$ vanishes for $i\gg 0$.
This comes as a complete surprise, since, in the
previously known range of $\K_*(\bZ/p^n)$ due to Angeltveit~\cite{angeltveit}, while
$\K_{2i-2}(\bZ/p^n)=0$ for all $2\leq i<p$, one has $\K_{2p-2}(\bZ/p^n)\iso
\bZ/p$. In the characteristic $p$ case, where $\Oscr_K/\varpi^n\iso\bF_p[z]/z^n$, the
even degree positive $\K$-groups all vanish due to a result of Hesselholt--Madsen
(see Corollary \ref{cor:hesselholtmadsen}, which we can obtain as a special case of
Theorem \ref{thm:evenvanishing}). Note that the vanishing of the even degree
$\K$-groups in sufficiently large degrees determines the orders of the odd degree $\K$-groups in
sufficiently large degrees by Proposition~\ref{prop:angeltveit}.

\begin{definition}[$p$-analogues]
    Recall that the $p$-analogue of an integer $j$ is
    \[[j]_p=\frac{p^j-1}{p-1}=1+p+p^2+\cdots+p^{j-1}.\]
\end{definition}

\begin{theorem}
\label{thm:evenvanishing}
If
  \[
    i-1\geq \frac{p}{p-1} \left( p[j]_p - p^j j + p^j\frac{n}{e}\right)
  \]
  with $j=\lceil \frac{n}{e}\rceil$,
then $\K_{2i-2}(\cO_K/\varpi^n) = 0$.
\end{theorem}

\begin{remark}
    The number $j=\lceil\frac{n}{e}\rceil$ is the exponent of the abelian group
    $\Oscr_K/\varpi^n$.
\end{remark}

\begin{remark}
  \label{rem:simplerbound}
  As $\frac{n}{e} - j \leq 0$, a slightly weaker but cleaner bound is given by
  \[
    i-1 \geq \frac{p}{p-1} \cdot p[j]_p = \frac{p^2}{(p-1)^2} \cdot (p^j-1).
  \]
  This version has the same asymptotic behaviour, but does not imply Corollary \ref{cor:hesselholtmadsen} below.
\end{remark}

The following corollary was originally discovered in~\cite[Thm.~A]{hesselholt-madsen-truncated}.

\begin{corollary}[Hesselholt-Madsen]\label{cor:hesselholtmadsen}
    For all $i\geq 2$,
    \[
    \K_{2i-2}(\bF_q[z]/z^n) = 0.
    \]
\end{corollary}
\begin{proof}
    We can write $\bF_q[z]/z^n$ as a quotient of
    rings of the form $\Oscr_K$ where $K$ has arbitrarily large ramification index $e$; for example
    $\bF_q[z]/z^n\iso W(\bF_q)[p^{1/e}] / p^{n/e}$.
    Thus, we can choose $e$ arbitrarily large in Theorem~\ref{thm:evenvanishing}. For large $e$, we have $j=1$ and the inequality simplifies to
    \[
     i-1\geq  \frac{p^2n}{(p-1)e},
    \]
    which is satisfied for any $i\geq 2$ for $e$ sufficiently large.
\end{proof}

\begin{corollary}
    If
  \[
    i-1\geq \frac{p}{p-1} \left( p[j]_p - p^j j + p^j\frac{n}{e}\right)
  \]
  with $j=\lceil \frac{n}{e}\rceil$,
    then $\K_{2i-1}(\cO_K/\varpi^n)$ is a group of order $(q^i-1)\cdot q^{i(n-1)}$, where $q$ is the order of the residue field of $\cO_K$.
\end{corollary}

\begin{proof}
    This follows from Theorem~\ref{thm:evenvanishing},
    Proposition~\ref{prop:angeltveit}, Corollary~\ref{cor:prime_to_p}, and Quillen's calculation in~\cite{Qui} of the $\K$-theory of finite fields.
\end{proof}

\subsection{Nilpotency mod $p$}

In this section, we observe that in $\Prismhat^{(1)}_{R/W(k)\llbracket z\rrbracket}$
and $\N^{\geq \star}\Prismhat^{(1)}_{R/W(k)\llbracket z\rrbracket}$ for $R=\cO_K/\varpi^n$,
various elements become nilpotent mod $p$.
This will feature in our proof of the even vanishing theorem, but also implies
nilpotence of $v_1$. It also implies that, in large $\F$-filtration,
the prismatic envelope $\Prismhat^{(1)}_{R/W(k)\llbracket z\rrbracket}$ admits divided powers.

Below, we write $x\sim y$ for two elements which are unit multiples mod $p$. Recall that
$$f_u^p\sim d^{p^{u+1}}f_{u+1}\sim z^{ep^{u+1}}f_{u+1}$$ in $\N^{\geq
p^{u+1}}\Prismhat_{R/W(k)\llbracket z_0\rrbracket}^{(1)}$. This follows from the explicit
relation for $f_u^p$ given in Lemma~\ref{lem:explicitrelations}, the fact that the
$R_u'$ vanish for the relation $z^n$ by
Remark~\ref{rem:vanishingofRu_prime}, and the fact that $d=E(z)\sim z^e$.

\begin{lemma}
  \label{lem:nygaardpowers}
  Let $R=\cO_K/\varpi^n$.
      In $\N^{p^{j-1}-1}\Prismhat^{(1)}_{R/W(k)\llbracket z\rrbracket}$ for $j\leq \lceil \frac{n}{e}\rceil$, we have
      \[
        z^{p^{j-1}(n-je) + [j]_p e - 1} d^{p^{j-1}-1} \sim z^{n-1} f_0^{p-1}\cdots f_{j-2}^{p-1}
      \]
\end{lemma}
\begin{proof}
  For $j=1$, this is the tautology $z^{n-1} = z^{n-1}$. We now proceed by induction. Assuming the relation for $j<\frac{n}{e}$, first observe that
  \[
    z^{p^{j-1}(n-je) + [j]_p e - 1} d^{p^{j-1}} \sim z^{n-1} f_0^{p-1}\cdots
    f_{j-2}^{p-1}d.
  \]
  Multiplying both sides with $z$ and using $z^nd \sim z^e f_0$ and $f_u^p \sim
    z^{ep^{u+1}} f_{u+1}$, we learn that
  \[
    z^{p^{j-1}(n-je) + [j]_p e} d^{p^{j-1}} \sim z^{[j]_p e}f_{j-1}.
  \]
  As $n-je>0$, we may use this relation $p-1$ times to deduce that
  \[
    z^{(p-1)p^{j-1}(n-je) + [j]_p e} d^{p^j-1} \sim z^{[j]_pe} f_{j-1}^{p-1}
    d^{p^{j-1}-1};
  \]
    (note that the exponent on the left-hand side is $p^{j}-1$, not $p^{j-1}$).
  Finally, multiplying with $z^{p^{j-1}(n-je)-1}$ and using the inductive assumption once more, we obtain
  \[
    z^{p^{j}(n-je) + [j]_p e -1} d^{p^j-1} \sim z^{n-1} f_0^{p-1}\cdots f_{j-1}^{p-1}.
  \]
  Since $p^j (n-je) + [j]_p e -1 = p^j(n-(j+1)e) + [j+1]_p e - 1$, this completes the inductive step.
\end{proof}

\begin{lemma}
  \label{lem:prismaticpowers}
  Let $R=\cO_K/\varpi^n$.
      In $\Prismhat^{(1)}_{R/W(k)\llbracket z\rrbracket}$ for $j\leq \lceil \frac{n}{e}\rceil$, we have
      \[
        z^{p^j(n-je) + p[j]_p e - 1} \sim z^{n-1} f_0^{p-1}\cdots f_{j-1}^{p-1}
      \]
\end{lemma}
\begin{proof}
  The $j<\lceil \frac{n}{e}\rceil$ case can be deduced directly from Lemma
    \ref{lem:nygaardpowers}, but the $j=\lceil \frac{n}{e}\rceil$ case requires special
    manipulations either way, so we proceed by an analogous induction. For $j=0$, we
    again have the tautology $z^{n-1} = z^{n-1}$. Given the relation for some
    $j<\frac{n}{e}$, multiplication by $z$ yields
    \begin{equation}\label{eq:rel_1}
        z^{p^j(n-je) + p[j]_p e} \sim z^{p[j]_pe} f_j.
    \end{equation}
      As $n-je>0$, we may multiply~\eqref{eq:rel_1} with $z^{(p-2)p^j(n-je)}$ and
      simplify using the relation~\eqref{eq:rel_1} to obtain
      \begin{equation}\label{eq:rel_2}
        z^{(p-1)p^j(n-je) + p[j]_p e} \sim z^{p[j]_pe} f_j^{p-1}.
      \end{equation}
      Multiplying~\eqref{eq:rel_2} with $z^{p^j(n-je)-1}$ and using the inductive
      assumption yields the relation
      \[
        z^{p^{j+1}(n-je) + p[j]_p e - 1} \sim z^{n-1} f_0^{p-1}\cdots f_j^{p-1}.
      \]
      As $p^{j+1}(n-je) + p[j]_p e - 1 = p^{j+1}(n-(j+1)e) + p[j+1]_p e - 1$, this
      completes the induction.
\end{proof}

\begin{lemma}
  Let $R=\cO_K/\varpi^n$ with $n\geq 2$. For $j = \lceil \frac{n}{e}\rceil$,
  \label{lem:znilpotence}
  \begin{enumerate}
      \item[{\em (1)}] $z^{p[j]_p e - p^j(je-n)} = 0$ in $\Prismhat^{(1)}_{R/W(k)\llbracket z\rrbracket}/p$, but $z^{p[j]_p e - p^j(je-n) - 1} \neq 0$;
      \item[{\em (2)}] for any $i\geq p^j-1$,
      \[
        z^{[j]_p e - p^{j-1}(je-n)} d^i = 0
      \]
      in $\N^{\geq i}\Prismhat^{(1)}_{R/W(k)\llbracket z\rrbracket} / p$, but $z^{[j]_p e - p^{j-1}(je-n)- 1} d^i \neq 0$.
  \end{enumerate}
\end{lemma}

\begin{proof}
  We start with the relation
    \[
      z^{p^j(n-je) + p[j]_p e - 1} \sim z^{n-1} f_0^{p-1}\cdots f_{j-1}^{p-1}
    \]
  from Lemma \ref{lem:prismaticpowers};
  these terms are not zero mod $p$ since the right-hand term is one of our standard
    basis elements of the prismatic cohomology. Multiplying with $z$, we get
  \[
    z^{p[j]_p e - p^j(je-n)} \sim z^{p[j]_p e} f_j.
  \]
  As $je-n\geq 0$, the exponent on the right is at least as big as on the left. Thus we may iterate, obtaining elements of arbitrarily high Nygaard filtration. This proves the first statement.

  For the second statement, it suffices to treat $i=p^j-1$. We start with the relation
  \[
    z^{p^{j-1}(n-je) + [j]_p e - 1} d^{p^{j-1}} \sim z^{n-1} f_0^{p-1}\cdots f_{j-2}^{p-1} d
  \]
    in $\N^{\geq p^{j-1}}\Prismhat_{R/W(k)\llbracket z_0\rrbracket}^{(1)}$ implied by
    Lemma~\ref{lem:nygaardpowers}. The right-hand side of this relation is non-zero as
    it is one of the standard basis elements of the Nygaard filtration (see
    Proposition~\ref{prop:f_filtered_statement}). Multiplying with $z$, we obtain
  \[
    z^{[j]_p e - p^{j-1}(je-n)} d^{p^{j-1}} \sim z^{[j]_pe}f_{j-1}.
  \]
  As $je-n\geq 0$, we may iterate this to obtain
  \[
    z^{[j]_p e - p^{j-1}(je-n)} d^{p^j-1} \sim z^{[j]_pe + (p-2) p^{j-1}(je-n)}f_{j-1}^{p-1} d^{p^{j-1}-1}.
  \]
  Again, using the relation from Lemma \ref{lem:nygaardpowers}, we further learn
  \[
    z^{[j]_p e - p^{j-1}(je-n)} d^{p^j-1} \sim z^{(p-1) p^{j-1}(je-n) + 1} z^{n-1} f_0^{p-1} \cdots f_{j-1}^{p-1}.
  \]
  The right hand side is a multiple of
  \[
    z^n f_0^{p-1} \cdots f_{j-1}^{p-1} \sim z^{p[j]_pe} f_j,
  \]
    which is zero mod $p$ by the first statement.
\end{proof}

\begin{corollary}\label{cor:divided_powers}
  For $j\geq \lceil \frac{n}{e}\rceil$, the element $f_j^p$ agrees with $pf_{j+1}$ up to a unit.
\end{corollary}
\begin{proof}
  Recall that $f_j^p = (-p + \lambda_j d^{p^{j+1}}) f_{j+1}$ where $\lambda_j$ is some
    unit. We have that
  \[
    e(p^{j+1}-1) \geq ep[j]_p,
  \]
  and so
  \[
      d^{p^{j+1}-1} \sim z^{e(p^{j+1}-1)}\sim 0
  \]
    by Lemma~\ref{lem:znilpotence}(1).
  It follows that $d^{p^{j+1}-1} = 0$ modulo $p$, and thus
  \[
    f_j^p = (-1 + \lambda_j r d) pf_{j+1}
  \]
    for some $r$. By $d$-completeness, the factor $(-1+\lambda_jrd)$ is a unit.
\end{proof}

\begin{remark}
    Corollary~\ref{cor:divided_powers} says that ``asymptotically'' $\Prismhat^{(1)}_{(\cO_K/\varpi^n) / W(k)\llbracket z\rrbracket}$ looks like a free divided power algebra.
\end{remark}

\subsection{A prismatic Cartier isomorphism}

Observe the following immediate consequence of relative-to-absolute descent
(Corollary~\ref{cor:syntomicsquare}).
\begin{theorem}\label{thm:prismatic_cartier}
    The commutative square
  \[
    \begin{tikzcd}
      \Prismhat_{R/W(k)\llbracket z_0\rrbracket}\{i\} \dar{w}\rar{\nabla} & \Prismhat^{\nabla}_{R/W(k)\llbracket z_0\rrbracket}\{i\}\dar{w^\nabla}\\
      \Prismhat^{(1)}_{R/W(k)\llbracket z_0\rrbracket}\{i\} \rar{\nabla} &
        \Prismhat^{(1),\nabla}_{R/W(k)\llbracket z_0\rrbracket}\{i\}
    \end{tikzcd}
  \]
  induces a quasi-isomorphism between the rows.
\end{theorem}

\begin{proof}
    This is due to the fact that both horizontal fibers are $\Prismhat_{R/\bZ_p}$ and
    the canonical inclusion $\Prismhat_{R/\bZ_p}\to \Prismhat_{R/\bZ_p}^{(1)}$ is an isomorphism since $\bZ_p$ is perfect.
\end{proof}

\begin{remark}
  We view this as a kind of Cartier isomorphism, since it looks formally similar to the Cartier isomorphism for $k[z]$ as quasi-isomorphism between the rows of
  \[
    \begin{tikzcd}
      \Omega^0_{k[z]/k} \rar{0}\dar& \Omega^1_{k[z]/k}\dar\\
      \Omega^0_{k[z]/k} \rar{d}&\Omega^1_{k[z]/k},
    \end{tikzcd}
  \]
    where the vertical maps are induced by the Frobenius on $k[z]$. (See for
    example~\cite{deligne-illusie}.)
    A meaningful justification of this analogy would require a more conceptual understanding of the connection $\nabla$.
\end{remark}

\begin{remark}\label{rem:divisible_by_p}
    Note that the vertical maps take $\F^{\geq j} \to \F^{\geq pj}$, due to the semilinearity of $w$. It follows in
    particular that $w$ induces an equivalence between the cokernel of
    \[
      \nabla\colon \gr^j_\F\Prism_{R/W(k)\llbracket z_0\rrbracket}\{i\}\to
      \gr^j_\F\Prism^\nabla_{R/W(k)\llbracket z_0\rrbracket}\{i\}
    \]
    and that of
    \[
      \nabla\colon \gr^{pj}_\F\Prismhat^{(1)}_{R/W(k)\llbracket z_0\rrbracket}\{i\}\to \gr^{pj}_\F\Prismhat^{(1),\nabla}_{R/W(k)\llbracket z_0\rrbracket}\{i\},
    \]
    and that 
    \[
      \nabla\colon \gr^{j}_\F\Prismhat^{(1)}_{R/W(k)\llbracket z_0\rrbracket}\{i\}\to \gr^{j}_\F\Prismhat^{(1),\nabla}_{R/W(k)\llbracket z_0\rrbracket}\{i\},
    \]
    is an isomorphism if $p\nmid j$.
\end{remark}

\begin{lemma}
\label{lem:frobeniustwist-cartier}
In the case $R=\cO_K/\varpi^n$, the prismatic Cartier isomorphism fits into a
    commutative diagram
\[
\begin{tikzcd}
  \F^{[1,b]}\N^{\geq i} \prism_{R/W(k)\llbracket z_0\rrbracket }^{(1)}\{i\}
    \rar{\nabla}\dar{\varphi_{R/W(k)\llbracket z_0\rrbracket }} & \F^{[1,b]}\N^{\geq i}
    \prism_{R/W(k)\llbracket z_0\rrbracket }^{(1),\nabla}\{i\}\dar{\varphi^\nabla_{R/W(k)\llbracket z_0\rrbracket }} \\
  \F^{[1,b]} \prism_{R/W(k)\llbracket z_0\rrbracket }\{i\} \rar{\nabla}\dar{w} &  \F^{[1,b]} \prism_{R/W(k)\llbracket z_0\rrbracket }^{\nabla}\{i\}\dar{w^\nabla} \\
  \F^{[1,pb]}\prism_{R/W(k)\llbracket z_0\rrbracket }^{(1)}\{i\} \rar{\nabla} & \F^{[1,pb]}\prism_{R/W(k)\llbracket z_0\rrbracket }^{(1),\nabla}\{i\}
\end{tikzcd}
\]
    inducing quasi-isomorphisms between the rows whenever $b\leq n(i+1) - 1$.
\end{lemma}

\begin{proof}
    The statement about the lower two rows follows from the fact that the prismatic
    Cartier isomorphism of Theorem~\ref{thm:prismatic_cartier} is compatible with filtrations.
    Commutativity of the upper square follows from Lemma~\ref{lem:nablafrobenius}.
    Recall that $\F^{[1,b]}\prism_{R/W(k)\llbracket
    z\rrbracket }$ is a free $W(k)$-module on monomials $z_0^c \prod \delta^u(a)^{e_u}s^i$ with $c<n$,
    $e_u<p$ and total filtration $\leq b$. Similarly, $\F^{[1,b]}\N^{\geq i}
    \prism_{R/W(k)\llbracket z_0\rrbracket }^{(1)}$ is free on monomials $d^{i - \sum
    p^u e_u} z_0^c
    \prod f_u^{e_u}w(s^i)$ with total filtration $\leq b$. The bound $b\leq
    n(i+1)-1$ guarantees that the $i$th divided Frobenius sends generators to
    generators: by
    construction, and Corollary~\ref{cor:f-uv-mapto-a-uv},
    \[
        \varphi_{W(k)\llbracket z_0\rrbracket } \left(d^{i - \sum p^u e_u} z_0^c \prod
    f_u^{e_u}w(s^i)\right)
    = z_0^c d^i\prod \varphi_{p^u}(f_u)^{e_u}\tfrac{s^i}{d^i} = z_0^c \prod \delta^u
    (a)^{e_u}s^i,
    \]
    so the upper left vertical map is an isomorphism. Similarly, the upper right
    vertical map is an isomorphism and hence the upper square induces a quasi-isomorphism
    between the top two rows.
\end{proof}

\subsection{Proof of the even vanishing theorem}

\begin{lemma}
  \label{lem:candivisible}
Let $R=\cO_K/\varpi^n$ and fix $b\geq 1$. If 
  \[
    e\left(i - \left\lfloor \frac{b}{n}\right\rfloor\right) \geq p[j]_p e - p^j (je - n)
  \]
  with $j = \lceil \frac{n}{e}\rceil$, then
  \[
      \can\colon \F^{[0,b]}\N^{\geq i} \Prismhat^{(1)}_{R/W(k)\llbracket
      z_0\rrbracket}\{i\} \to \F^{[0,b]}\Prismhat^{(1)}_{R/W(k)\llbracket
      z_0\rrbracket}\{i\}
  \]
  is divisible by $p$. Similarly, if
  \[
    e\left(i-1 - \left\lfloor \frac{b}{n}\right\rfloor\right) \geq p[j]_p e - p^j (je - n),
  \]
  then
  \[
    \can\colon \F^{[1,b+1]}\N^{\geq i} \Prismhat^{(1),\nabla}_{R/W(k)\llbracket
    z_0\rrbracket}\{i\} \to \F^{[1,b+1]}\Prismhat^{(1),\nabla}_{R/W(k)\llbracket
    z_0\rrbracket}\{i\}
  \]
  is divisible by $p$.
\end{lemma}

\begin{proof}
    We may assume $i=0$.
  A basis for $\F^{[0,b]}\N^{\geq i} \Prismhat^{(1)}_{R/W(k)\llbracket z_0\rrbracket}$ is given by $\widetilde{d}^{i - \sum p^ue_u} z^k \prod f_u^{e_u}$ with $e_u<p$ and $k<n$.
    Here, $k+\sum np^u e_u \leq b$, and so the exponent of $d$ is $\geq i - \lfloor
    \frac{b}{n}\rfloor$. By Lemma \ref{lem:znilpotence}, we have
  \[
    z^{p[j]_p e - p^j(je-n)} = 0
  \]
  mod $p$. As $d\sim z^e$, the claim follows. The claim for the $\nabla$-term follows from the observation that
  \[
    \F^{[1,b+1]}\N^{\geq i} \Prismhat^{(1),\nabla}_{R/W(k)\llbracket z_0\rrbracket} \cong \F^{[0,b]}\N^{\geq i-1} \Prismhat^{(1)}_{R/W(k)\llbracket z_0\rrbracket},
  \]
  compatibly with $\can$.
\end{proof}

\begin{proof}[Proof of Theorem \ref{thm:evenvanishing}]
It is enough to show that the maps into the lower right corner in the diagram
\[
\begin{tikzcd}
  \F^{[1,in-1]} \N^{\geq i} \Prismhat^{(1)}_{R/W(k)\llbracket z_0\rrbracket }\{i\} \rar{\nabla}\dar{\can-\varphi} & \F^{[1,in-1]}\N^{\geq i}\Prismhat^{(1),\nabla}_{R/W(k)\llbracket z_0\rrbracket }\{i\}\dar{\can - \varphi^\nabla}\\
  \F^{[1,in-1]} \Prismhat^{(1)}_{R/W(k)\llbracket z_0\rrbracket }\{i\} \rar{\nabla} & \F^{[1,in-1]} \Prismhat^{(1),\nabla}_{R/W(k)\llbracket z_0\rrbracket }\{i\} 
\end{tikzcd}
\]
  are jointly surjective. In fact, on the $\nabla$ terms, can is already an isomorphism on $\F^{\geq (i-1)n+1}$, and so we may pass to the quasi-isomorphic quotient
\[
\begin{tikzcd}
  \F^{[1,in-1]} \N^{\geq i} \Prismhat^{(1)}_{R/W(k)\llbracket z_0\rrbracket }\{i\} \rar{\nabla}\dar{\can-\varphi} & \F^{[1,(i-1)n]}\N^{\geq i}\Prismhat^{(1),\nabla}_{R/W(k)\llbracket z_0\rrbracket }\{i\}\dar{\can - \varphi^\nabla}\\
  \F^{[1,in-1]} \Prismhat^{(1)}_{R/W(k)\llbracket z_0\rrbracket }\{i\} \rar{\nabla} & \F^{[1,(i-1)n]} \Prismhat^{(1),\nabla}_{R/W(k)\llbracket z_0\rrbracket }\{i\}.
\end{tikzcd}
\]
Joint surjectivity of the maps into the bottom right corner is equivalent to
    surjectivity of
\[
  (\can - \varphi^\nabla)\colon \F^{[1,(i-1)n]}\N^{\geq
    i}\Prismhat^{(1),\nabla}_{R/\bZ_p\llbracket z_0\rrbracket }\{i\} \to \F^{[1,(i-1)n]}
    \Prismhat^{(1),\nabla}_{R/W(k)\llbracket z_0\rrbracket }\{i\} / \im(\nabla).
\]
Since the target is a finitely generated abelian $p$-group, it suffices to check that
\[
  (\can - \varphi^\nabla)\colon \F^{[1,(i-1)n]}\N^{\geq i}\Prismhat^{(1),\nabla}_{R/W(k)\llbracket z_0\rrbracket }\{i\} \to \F^{[1,(i-1)n]} \Prismhat^{(1),\nabla}_{R/W(k)\llbracket z_0\rrbracket }\{i\} / (\im(\nabla),p)
\]
is surjective. Now a map $f\colon A\to B$ is surjective if and only if $A/H\to B/f(H)$ is surjective, for an arbitrary subgroup $H\subseteq A$.

  Applying this to the subgroup $\F^{[\lceil \frac{(i-1)n}{p}\rceil+1, (i-1)n]}\N^{\geq i}\Prismhat^{(1),\nabla}_{R/W(k)\llbracket z_0\rrbracket }\{i\}$, we are
    thus reduced to checking that
\begin{equation}
\label{eqn:surjectivemodp}
  \can-\varphi^\nabla\colon \F^{[1,\lceil \frac{(i-1)n}{p}\rceil ]}\N^{\geq i}\Prismhat^{(1),\nabla}_{R/W(k)\llbracket z_0\rrbracket }\{i\} \to
    \frac{\F^{[1,(i-1)n]} \Prismhat^{(1),\nabla}_{R/W(k)\llbracket z_0\rrbracket
    }\{i\}}{\left(\im(\nabla), p, (\can-\varphi^\nabla) \F^{[\lceil \frac{(i-1)n}{p} \rceil +
    1, (i-1)n]}\right)}
\end{equation}
is surjective. By assumption, 
  \[
    i-1\geq \frac{p}{p-1} \left( p[j]_p - p^j j + p^j\frac{n}{e}\right),
  \]
  so
  \begin{align*}
    e \left(i-1 -\left\lfloor  \frac{\lceil
      \frac{(i-1)n}{p}\rceil-1}{n}\right\rfloor\right) &\geq e \left(i-1 - \frac{\lceil \frac{(i-1)n}{p}\rceil-1}{n}\right) \\
    &\geq e \left(i-1 - \frac{(i-1)n}{pn}\right)\\
      &\geq e\cdot \frac{p-1}{p}\cdot (i - 1)\\
      &\geq p[j]_pe-p^j(je-n).
  \end{align*}
Thus, Lemma \ref{lem:candivisible} applies to show that the $\can$ part of \eqref{eqn:surjectivemodp} is divisible by $p$, thus zero.
    We are reduced to showing that the $\varphi^\nabla$ part of the map
    \eqref{eqn:surjectivemodp} is surjective. We claim this holds even when quotienting
    the target by less, i.e., that the map
\[
  \varphi^\nabla\colon \F^{[1,\lceil \frac{(i-1)n}{p}\rceil ]}\N^{\geq i}\Prismhat^{(1),\nabla}_{R/W(k)\llbracket z_0\rrbracket }\{i\} \to \F^{[1,(i-1)n]} \Prismhat^{(1),\nabla}_{R/W(k)\llbracket z_0\rrbracket }\{i\} / \im(\nabla)
\]
is surjective. This follows directly from the form of the prismatic Cartier isomorphism presented in Lemma \ref{lem:frobeniustwist-cartier}.
\end{proof}

\section{Nilpotence of $v_1$}\label{sec:v1nilpotence}

It has recently been proven in~\cite{bhatt-clausen-mathew, land-mathew-meier-tamme} that
$L_{T(1)}\K(\bZ/p^n)\we 0$ for all $n\geq 1$, where $L_{T(1)}$ denotes the Bousfield localization at
the $T(1)$-equivalence. Recall that a map
of spectra $X\rightarrow Y$ is a $T(1)$-equivalence if
$X/p[v_1^{-1}]\rightarrow Y/p[v_1^{-1}]$ is an equivalence, where $v_1$ is an element of
$\bS/p$ which induces an equivalence after tensoring with $\KU$. For $p$ odd, such an element exists
in $\pi_{2p-2}(\bS/p)$ and is uniquely determined up to units, while, for $p=2$, only
$v_1^4\in\pi_8(\bS/2)$ exists. This story is complicated by the fact that
$\bS/p$ is not an $A_\infty$-ring spectrum, that for $p=3$ multiplication is not
associative, and that
for $p=2$ there is no multiplication; see~\cite[App.~A]{thomason-etale}.

In mod $p$ algebraic $K$-theory, the story is somewhat simpler. For all primes $p$, even and odd,
there is a class $v_1\in\K_{2p-2}(\bZ_p,\bF_p)$, which is the image of $v_1\in\pi_{2p-2}(\bS/p)$
when $p$ is odd and whose fourth power is the image of $v_1^4\in\pi_8(\bS/2)$ for $p=2$.
In this case, $\K(\bZ_p;\bF_p)\we\TC(\bZ_p;\bF_p)$ and $v_1$ arises from a class
$$v_1\in\H^0(\bF_p(p-1)(\bZ_p))$$ in syntomic cohomology for degree reasons (see
Remark~\ref{rem:in_the_limit}).
The sheared down associated graded $$\bigoplus_{i\geq
0}\gr^i_\mot\TC(\bZ_p;\bF_p)[-2i]\we\bigoplus_{i\geq 0}\bF_p(i)(\bZ_p)$$ is an $\bE_\infty$-algebra in graded spectra and
thus one can invert $v_1$ on $\bigoplus_{i\geq 0}\bF_p(i)(\bZ_p)$ to obtain $\bigoplus_{i\in\bZ}\bF_p(i)(\bZ_p)[v_1^{-1}]$.
Similarly, if $R$ is any quasisyntomic $\bZ_p$-algebra, the element $v_1$ acts on
$\bigoplus_{i\in\bZ}\bF_p(i)(R)$ and we can form the localization
$\bigoplus_{i\in\bZ}\bF_p(i)(R)[v_1^{-1}]$.

\begin{remark}\label{rem:v1}
    We will use that $v_1\in\K_{2p-2}(\Oscr_K;\bF_p)$ is non-zero
    for every $p$-adic field $K$. This follows from the fact that $v_1$ maps to (a unit
    multiple of)
    $\beta^{p-1}$ in $\K(\overline{K};\bF_p)\we\ku/p$, where $\beta$ denotes the Bott
    element.
\end{remark}

\begin{proposition}\label{prop:v1_nilpotence}
  Let $R=\cO_K/\varpi^n$ where $n\geq 2$. Let $v\in \N^{\geq
    p-1}\Prismhat^{(1)}_{R/W(k)\llbracket z\rrbracket}/p$ be the element
    $d^{p}$. For $j = \lceil \frac{n}{e}\rceil$, we have $v^{[j]_p}=0$ in
    $\N^{\geq [j]_p(p-1)}\Prismhat^{(1)}_{R/W(k)\llbracket z\rrbracket}/p$. If $e$ divides $n$, this is sharp, i.e. $v^{[j]_p-1}\neq 0$.
\end{proposition}

\begin{proof}
  After trivializing the Breuil--Kisin orientation and using $d\sim z^e$, we have
  \[
    v^{[j]_p} \sim z^{[j]_p e} d^{(p-1)[j]_p} = z^{[j]_p e} d^{p^j-1}.
  \]
  This is zero mod $p$ by Lemma~\ref{lem:znilpotence}.
  If $e$ divides $n$, we have $n-ej = 0$, and Lemma~\ref{lem:nygaardpowers} shows that
  \[
    z^{[j]_p e - 1} d^{i} \sim z^{n-1}f_0^{p-1}\cdots f_{j-2}^{p-1} d^{i-(p^{j-1}-1)}
  \]
  is nonzero mod $p$ for any large enough $i$, and so in particular
  \[
    v^{[j]_p-1} \sim z^{([j]_p-1) e} d^{(p-1)([j]_p-1)}
  \]
  is nonzero mod $p$.
\end{proof}

\begin{theorem}\label{thm:nilpotent}
    The class $v_1\in\H^0(\bF_p(p-1)(\bZ_p))$ acts on
    $\bigoplus_{i\in\bZ}\bF_p(i)(\Oscr_K/\varpi^n)$ with nilpotency
    degree at most $[j]_p$, where $j=\lceil\tfrac{n}{e}\rceil$. If $e$
    divides $n$, this is the exact nilpotence degree.
\end{theorem}

\begin{proof}
    We first locate $v_1$ in the relative-to-absolute descent spectral sequence for the syntomic
    cohomology of $\Oscr_K$, finding that there is an exact sequence
    $$0\rightarrow\H^0(\bF_p(p-1)(\Oscr_K))\rightarrow\H^0(\bF_p(p-1))(\Oscr_K/A^0)\xrightarrow{d^0-d^1}\H^0(\bF_p(p-1))(\Oscr_K/A^1).$$
    In particular, by the coconnectivity of the syntomic complexes involved, $v_1$ is
    non-zero in $\H^0(\bF_p(p-1)(\Oscr_K/A^0))$.
    By Proposition~\ref{prop:oriented_syntomic}, the relative syntomic complex
    $\bZ_p(p-1)(\Oscr_K/A^0)$ is equivalent to
    $$\fib\left(W(k)\llbracket z_0\rrbracket\cdot d^{p-1}\xrightarrow{f(z_0)\cdot d^{p-1}\mapsto
    (f(z_0)d^{p-1}-\varphi(f(z_0)))}W(k)\llbracket
    z_0\rrbracket\right),$$ where $d=E(z_0)$ is the distinguished element of the Breuil--Kisin prism
    $A^0$. Working modulo $p$ we find that the complex $\bF_p(p-1)(\Oscr_K/A^0)$ is given by
    $$\fib\left(k\llbracket z_0\rrbracket\cdot z_0^{p-1}\xrightarrow{f(z_0)\cdot d^{p-1}\mapsto
    (f(z_0)(uz_0^e)^{(p-1)}-\varphi(f(z_0)))}k\llbracket
    z_0\rrbracket\right),$$ where $u$ is the coefficient of $z^e_0$ in $E(z_0)$. As
    $u^{p-1}= 1\pmod p$, we can rewrite the map as $f(z_0)\mapsto
    f(z_0)z_0^{e(p-1)}-\varphi(f(z_0))$ modulo $p$. The kernel of this map is generated by
    $z_0^e\cdot d^{p-1}$, so that
    $\H^0(\bF_p(p-1)(\Oscr_K/A^0))$ is a $1$-dimensional $\bF_p$-vector space. Thus, by counting
    dimensions and Remark~\ref{rem:v1}, it follows that the class $z_0^e\cdot
    d^{p-1}\sim d^p=v$ survives to give a unit multiple of
    $v_1$ in the absolute $p$-adic syntomic complex $\bF_p(p-1)(\Oscr_K)$.

    To prove that $v_1^{[j]_p}=0$ in $\bF_p([j]_p(p-1))(\Oscr_K/\varpi^n)$ it suffices to show that
    $d^{p[j]_p}$ maps to zero in $\bF_p([j]_p(p-1))((\Oscr_K/\varpi^n)/A^0)$.
    This is the content of Proposition~\ref{prop:v1_nilpotence}. The claim about the
    exact nilpotence degree follows from the same proposition.
\end{proof}

We recover the result of~\cite{bhatt-clausen-mathew, land-mathew-meier-tamme}.

\begin{corollary}
    For any $n\geq 1$ and any prime number $p$, $$L_{T(1)}\K(\bZ/p^n)\we 0,$$
    where $T(1)$ denotes the height $1$ telescopic localization at the prime $p$.
\end{corollary}

\begin{proof}
    We have $L_{T(1)}\K(\bZ/p^n)\we L_{T(1)}\K(\bZ/p^n;\bZ_p)\we L_{T(1)}\TC(\bZ/p^n;\bZ_p)$.
    It is enough to show that $\TC(\bZ/p^n;\bF_p)[v_1^{-1}]$ vanishes. This follows from
    Theorem~\ref{thm:nilpotent} and the fact that the motivic filtration induces a finite
    filtration on each $\pi_r\TC(\bZ/p^n;\bF_p)$ by Corollary~\ref{cor:k_groups}.
\end{proof}

\begin{corollary}\label{cor:nil_on_k}
    Suppose that $p\geq 5$ is a prime and that $j=\lceil\frac{n}{e}\rceil$. Then, for any $n\geq 1$, the action of $v_1^{2[j]_p}$ on
    $\K(\Oscr_K/\varpi^n;\bF_p)$ is homotopic to $0$.
\end{corollary}

\begin{proof}
    We use that for $p\geq 5$, $\K(\Oscr_K/\varpi^n;\bF_p)$ is a homotopy associative ring
    spectrum. This reduces the corollary to checking that $v_1^{2[j]_p}=0$ in
    $\K_*(\Oscr_K/\varpi^n;\bF_p)$.
    However, $v_1^{[j]_p}=0$ in $$\H^0(\bF_p([j]_p(p-1))(\Oscr_K/\varpi^n)),$$ the motivic associated graded, so it lifts to some element, say
    $x\in\H^2(\bF_p([j]_p(p-1)+1))(\Oscr_K/\varpi^n))$. But, now, $v_1^{[j]_p}x=0$ in the motivic associated
    graded, and thus $v_1^{2[j]_p}=0$ in $\K_*(\Oscr_K/\varpi^n;\bF_p)$, as claimed.
\end{proof}

\section{An algorithm for syntomic cohomology}\label{sec:algorithm}

Section~\ref{sec:implemented} describes our algorithm as we implemented it.
Section~\ref{sec:precision} contains an analysis of the $p$-adic precision required in order to
guarantee correct results.

\subsection{The algorithm as implemented}\label{sec:implemented}

In this subsection, we give the algorithm as implemented in~\cite{akn_syn_coh} for the computation of the $\K$-groups of
$R=\Oscr_K/\varpi^n$.

\paragraph{{\ttfamily Input}:} integers $q=p^f,i,n$ for $f,i,n\geq 1$ and an Eisenstein polynomial $E(z)\in
W(\bF_q)[z]$, assumed to be normalized so that $E(0)=p$ and specified up to the working $p$-adic
precision specified in Step 1 below.

\paragraph{{\ttfamily Output}:} $p$-adic matrices $\syn^0$ and $\syn^1$ defining a complex
\begin{equation*}
    \cdots \to 0\to \bZ_p^{f(in-1)} \xto{\syn^0} \bZ_p^{2f(in-1)}
    \xto{\syn^1} \bZ_p^{f(in-1)} \to 0 \to  \cdots
\end{equation*}
quasi-isomorphic
to $\bZ_p(i)(\Oscr_K/\varpi^n)$. The cohomology of this matrix can be computed using Smith normal form
and gives the cohomology of $\bZ_p(i)(\Oscr_K/\varpi^n)$ and hence the structure of
$\K_{2i-2}(\Oscr_K/\varpi^n;\bZ_p)$ and $\K_{2i-1}(\Oscr_K/\varpi^n;\bZ_p)$.

\paragraph{{\ttfamily Strategy}:} we will compute the commutative square
\begin{equation}\label{eq:algorithm}
    \begin{gathered}
        \xymatrix{
            \F^{[1,in-1]}\N^{\geq i}\Prismhat^{(1)}_{R/W(\bF_q)\llbracket
            z_0\rrbracket}\{i\}\ar[r]^{\N^{\geq i}\nabla}\ar[d]^{\can-\varphi}&
            \F^{[1,in-1]}\N^{\geq i}\Prismhat^{(1),\nabla}_{R/W(\bF_q)\llbracket
            z_0\rrbracket}\{i\}\ar[d]^{\can-\varphi}\\
            \F^{[1,in-1]}\Prismhat^{(1)}_{R/W(\bF_q)\llbracket z_0\rrbracket}\{i\}\ar[r]^{\nabla}& \F^{[1,in-1]}\Prismhat^{(1),\nabla}_{R/W(\bF_q)\llbracket z_0\rrbracket}\{i\}\\
        }
    \end{gathered}
\end{equation}
Each of the vertices in this square is a free $W(\bF_q)$-module of rank $(in-1)$. The maps in the
diagram are not $W(\bF_q)$-linear thanks to Frobenius, but they are of course $\bZ_p$-linear, so we use
$W(\bF_q)\iso\bZ_p^f$, to compute matrix representatives for this square in $\bZ_p$-modules.
For the computation of $\nabla$, we will also work in
$\Prismhat^{(1)}_{\cO_K/W(\bF_q)\llbracket z_0\rrbracket}\{i\}$ and
$\Prismhat^{(1)}_{\cO_K/W(\bF_q)\llbracket z_0,z_1\rrbracket}\{i\}$.

\paragraph{{\ttfamily Step 1: find the working $p$-adic precision}.}
Lemma~\ref{lem:targetprecision} shows that $\syn^0$ and $\syn^1$ should be computed up to $p$-adic
precision $$\sum_{j=1}^{in-1}\epsilon(i,j)+v_p\{j,n\},$$
where $\epsilon(i,j)=1$ if $n\mid j$ and $0$ otherwise in the given range and $\{j,n\}=n$ if $n\mid j$ and
$j$ otherwise (compare Proposition \ref{prop:angeltveit}). This is the {\em target $p$-adic precision}. (It is also possible to work with a smaller target precision, at the cost of potentially having to restart the computation with a larger target precision, see Remark \ref{rem:movingtarget}.) Several steps in the algorithm below
result in a cumulative precision loss of at most
$$\lfloor\log_p(in-1)\rfloor+\sum_{s=p}^{i-1}nv_p(s!)+n\binom{i}{2}+\sum_{j=1}^{in-1}\epsilon(i,j)+v_p\{j,n\}.$$
Thus, all computations in the algorithm below should be made starting at a $p$-adic precision of
$$2\left(\sum_{j=1}^{in-1}\epsilon(i,j)+v_p\{j,n\}\right)+\lfloor\log_p(in-1)\rfloor+\sum_{s=p}^{i-1}nv_p(s!)+n\binom{i}{2},$$
the {\em working $p$-adic precision}, although we warn the reader that the effective
precision drops in several of the steps below.

\paragraph{{\ttfamily Step 2: initialize linear bases}.}
Choose a basis $1,x,x^2,\cdots,x^{f-1}$ for $\bF_q$ over $\bF_p$ and lift this to a basis for
$W(\bF_q)$ over $\bZ_p$. Assume that $x$ satisfies a fixed degree $f$ monic polynomial $h(x)=0$.

A $W(\bF_q)$-basis for $\F^{[1,in-1]}\Prismhat^{(1)}_{R/W(\bF_q)\llbracket z_0\rrbracket}$ is given by
monomials $z_0^k\prod f_u^{e_u}$ for $0\leq k<n$ and $1\leq k+\sum np^ue_u\leq in-1$
with $e_u<p$. A $W(\bF_q)$-basis for
$\F^{[1,in-1]}\N^{\geq i}\Prismhat^{(1)}_{R/W(\bF_q)\llbracket z_0\rrbracket}$ is given by elements of the
form $\widetilde{d}^{i-\sum p^ue_u} z_0^k\prod f_u^{e_u}$ with $0\leq k<n$, $1\leq
k+\sum np^ue_u\leq in-1$, and $e_u<p$. Bases of $\F^{[1,in-1]}\Prismhat^{(1)}_{R/W(\bF_q)\llbracket z_0\rrbracket }\{i\}$ and
$\F^{[1,in-1]}\N^{\geq i}\Prismhat^{(1)}_{R/W(\bF_q)\llbracket z_0\rrbracket }\{i\}$ are obtained symbolically by
multiplying with $w(s^i)$, where $s$ is our preferred filtered crystalline Breuil--Kisin orientation. As $\F^{[1,in-1]}\N^{\geq\star}\Prismhat^{(1),\nabla}_{R/W(\bF_q)\llbracket z_0\rrbracket}\{i\} = \F^{[0,in-2]} \N^{\geq \star-1}\Prismhat^{(1)}_{R/W(\bF_q)\llbracket z_0\rrbracket} \{i-1\}$, bases for the $\nabla$-terms are described analogously, with certain indexing shifts.
This determines $W(\bF_q)$-bases on all the terms in the diagram \eqref{eq:algorithm}.
Finally, $\bZ_p$-bases are given by $x^dz_0^k\prod f_u^{e_u}w(s^i)$ for
$0\leq d<f$, $0\leq k<n$, $1\leq k+\sum np^ue_u\leq in-1$, and $e_u<p$ on the lower left term, and analogously for the others, where we present $W(\bF_q) = \bZ_p[x]/\widetilde{h}(x)$ where $h$ is an arbitrary $\bZ_p$-lift of the minimal polynomial $h(x)$ used to present $\bF_q$.
Let $M_f$ be the largest $u$ such that $f_u$ appears ($M_f=\lfloor\log_p(i-1)\rfloor$ if
$i>1$ and $0$ otherwise).

For $\F^{[1,in-1]}\Prismhat_{\Oscr_K/W(\bF_q)\llbracket z_0\rrbracket}^{(1)}$ we have a
$W(\bF_q)$-linear basis $z_0,\ldots,z_0^{in-1}$. For $\F^{[1,in-1]}\N^{\geq
i}\Prismhat^{(1)}_{\Oscr_K/W(\bF_q)\llbracket z_0\rrbracket}$ a $W(\bF_q)$-linear basis is
given by $\widetilde{d}^i z_0^k$ for $1\leq k\leq in-1$. For,
$\F^{[1,in-1]}\Prismhat_{\Oscr_K/W(\bF_q)\llbracket z_0,z_1\rrbracket}^{(1)}$ there is a
$W(\bF_q)$-linear basis $z_0^k\prod g_u^{e_u}$ for $1\leq k+p^ue_u\leq in-1$ and
$e_u<p$. For $\F^{[1,in-1]}\N^{\geq i}\Prismhat^{(1)}_{\Oscr_K/W(\bF_q)\llbracket z_0,z_1\rrbracket}$,
a $W(\bF_q)$-linear basis is given by monomials $\widetilde{d}^{i-\sum p^ue_u}z_0^k\prod
g_u^{e_u}$ with $1\leq k+\sum p^ue_u\leq in-1$, and $e_u<p$. The Breuil--Kisin twists
obtain analogous bases by multiplying with $w(s^i)$ or shifting, respectively.
Let $M_g$ be the largest $u$ such that $g_u$ appears ($M_f=\lfloor\log_p(in-1)\rfloor$).

\paragraph{{\ttfamily Step 3: compute relations and structure of the prismatic envelopes}.} 

For $\Prismhat^{(1)}_{R/W(\bF_q)\llbracket z_0\rrbracket}$, find
elements $\lambda_u\in W(\bF_q)\llbracket z_0\rrbracket$ as in Lemma
\ref{lem:explicitrelations} and create a lookup table of relations for $f_u^p$ in the
prismatic envelope in terms of those, noting that here the terms $R_u'$ are all zero
(Remark \ref{rem:vanishingofRu_prime}). We compute the $\delta$ operation on
$W(\bF_q)\llbracket z_0\rrbracket$ by $\delta(x) = \frac{\varphi(x)-x^p}{p}$, so each
step of the recursion for $\lambda_u$ causes a $p$-adic precision loss by $1$ digit.
Also create a lookup table for the values of $\varphi_{p^u}(f_u)$ following Lemma
\ref{lem:f_uv_structure}.

For the structure of $\Prismhat^{(1)}_{\cO_K/W(\bF_q)\llbracket z_0,z_1\rrbracket}$, use
Construction \ref{cons:Ru_recursion} to compute elements $R_u'$ and a $\delta$-ring
structure on $W(\bF_q)\llbracket z_0\rrbracket [g_u\mid 0\leq u\leq M_g]$, using the additional trick that
we may rewrite
\[
  R_0' = \frac{\delta(z_1-z_0)}{\delta(d)} = \frac{1}{\delta(d)} \sum_{1\leq j\leq p-1} \frac{1}{p}\binom{p}{j} g_0^j,
\]
and so may work in $W(\bF_q)\llbracket z_0\rrbracket[g_u\mid 0\leq u\leq M_g]$ instead
of the bigger $W(\bF_q)\llbracket z_0,z_1\rrbracket[g_u\mid 0\leq u\leq M_g]$.
Following Lemma~\ref{lem:explicitrelations}, compute a lookup table of relations for $g_u^p$ in terms of those. The $\delta$-ring structure from Construction \ref{cons:Ru_recursion} is stored in terms of the Frobenius lift
\[
  \varphi(g_u) = \varphi(d)^{p^u}(\lambda_u g_{u+1,v} + R_u')
\]
instead of directly storing the values of $\delta(g_u)$, so each of the applications of $\delta$ in the recursion for $R_u'$ causes a $p$-adic precision loss by $1$ digit, which is the same as for $\lambda_u$. Also create a lookup table for the values of $\varphi_{p^u}(g_u)$ in following Lemma \ref{lem:f_uv_structure}.

\paragraph{{\ttfamily Subalgorithm: reduce}.} At various points below, we have to rewrite a general element
of one of $\Prismhat^{(1)}_{R/W(\bF_q)\llbracket z_0\rrbracket}$ or $\Prismhat^{(1)}_{\cO_K/W(\bF_q)\llbracket z_0\rrbracket}$ in terms of the chosen linear bases above; we do so using the relations introduced in {\ttfamily Step 3}. The rewriting process is
$W(\bF_q)$-linear and we use the relations $z_0^n=f_0$ and the lookup table for $f_u^p$ from {\ttfamily Step 3} in the first case, and $z_1 = z_0+g_0$ and the lookup table for $g_u^p$ from {\ttfamily Step 3} in the second case.
An application of the {\ttfamily reduce subalgorithm} means in this context to replace a given polynomial in $z_0$ and the $f_u$ (or $z_0, z_1$ and $g_u$) by an equal term in prismatic cohomology which is expressed as a sum of terms in the $W(\bF_q)$-linear basis introduced in {\ttfamily Step 2}. In practice, we perform this reduction
by simultaneously imposing all relations above as well as all of the relations in the lookup table
for the $p$th powers and repeating this process until no further changes occur in $\F^{[1,in-1]}$.
This process converges in $\F$ and hence terminates in finite time in $\F^{[1,in-1]}$.
(It is possible to recursively reduce the expressions in the lookup table for the $f_u^p$ and $g_u^p$ relations, which seems to improve runtimes in the case of the $f_u^p$, and worsen runtimes in the case of the $g_u^p$.)

\paragraph{{\ttfamily Step 4: compute the change-of-orientation unit $w(v)$}.} The map
$\nabla$ in Breuil--Kisin weight $i$ involves the unit $w(v)$ from Lemma \ref{lem:bk-factor-general}. In $\Prism_{\cO_K/W(\bF_q)\llbracket z_0\rrbracket}$, we have
\[
  u = \frac{E(z_1)}{E(z_0)} = 1 + \frac{E(z_1)-E(z_0)}{z_1-z_0} b,
\]
and thus in $\Prismhat^{(1)}_{\cO_K/W(\bF_q)\llbracket z_0,z_1\rrbracket}$,
\[
  w(u) = 1 + \varphi\left(\frac{E(z_1)-E(z_0)}{z_1-z_0}\right) \varphi_1(g_0).
\]
Using the value of $\varphi_1(g_0)$ described in Step 3, we may compute the value
\[
  w(v) = \prod_{r\geq 0} \varphi^r(w(u))
\]
as the fixed point of the recursion $x = w(u)\cdot \varphi(x)$ in
$\F^{[1,in-1]}\Prismhat^{(1)}_{\cO_K/W(\bF_q)\llbracket z_0,z_1\rrbracket}$, by
iterating (and reducing) until the value remains unchanged, starting with $x=w(u)$. This
converges rapidly in the $\F$-filtration and takes around $\lfloor\log_p{(in-1)}\rfloor$
number of steps.

\paragraph{{\ttfamily Step 5: compute $\nabla: \F^{[1,in-1]}\Prismhat^{(1)}_{\cO_K/W(\bF_q)\llbracket
z_0\rrbracket}\{i\}\rightarrow\F^{[1,in-1]}\Prismhat^{(1),\nabla}_{\cO_K/W(\bF_q)\llbracket z_0\rrbracket}\{i\}$}.}

A basis for $\Prismhat^{(1)}_{\cO_K/W(\bF_q)\llbracket z_0\rrbracket}\{i\}$ is given by
$z_0^k w(s^i)$ for $1\leq k\leq in-1$ and $\nabla(z_0^k w(s^i))$ can be computed as
($w(s^{i-1})$ times) the coefficient of $g_0$ with respect to the $W(\bF_q)\llbracket z_0\rrbracket$-basis $\prod g_u^{e_u}$ in 
\[
  w(v^i)(z_0+g_0)^k - z_0^k, 
\]
after reducing. Compute this using the value of $w(v)$ from Step 4.

\paragraph{{\ttfamily Step 6: compute the reduction maps $\Prismhat^{(1)}_{\cO_K/W(\bF_q)\llbracket z_0\rrbracket}\{i\} \to \Prismhat^{(1)}_{R/W(\bF_q)\llbracket z_0\rrbracket}\{i\}$}.}
 Reduce $z_0^k$ for $1\leq k\leq in-1$ with respect to the relations in $\Prismhat^{(1)}_{R/W(\bF_q)\llbracket z_0\rrbracket}$ and express as a linear combination of the basis elements $z_0^k\prod f_u^{e_u}$ to obtain a matrix $\red$ for the map
 \[
\Prismhat^{(1)}_{\cO_K/W(\bF_q)\llbracket z_0\rrbracket}\{i\} \to \Prismhat^{(1)}_{R/W(\bF_q)\llbracket z_0\rrbracket}\{i\}
\]
induced by $\cO_K\to \cO_K/\varpi^n=R$. Analogously (taking the indexing shift into account), compute a matrix $\red^\nabla$ for
 \[
\Prismhat^{(1),\nabla}_{\cO_K/W(\bF_q)\llbracket z_0\rrbracket}\{i\} \to
\Prismhat^{(1),\nabla}_{R/W(\bF_q)\llbracket z_0\rrbracket}\{i\}.
\]

\paragraph{{\ttfamily Step 7: compute $\nabla:\F^{[1,in-1]}\Prismhat^{(1)}_{R/W(\bF_q)\llbracket z_0\rrbracket}\{i\}\rightarrow\F^{[1,in-1]}\Prismhat^{(1),\nabla}_{R/W(\bF_q)\llbracket z_0\rrbracket}\{i\}$}.}

The map $\nabla$ is natural in $R$, so we may reduce to the $\Oscr_K$ case via the commutative diagram
\[
\begin{tikzcd}
  \F^{[1,in-1]}\Prismhat^{(1)}_{\Oscr_K/W(\bF_q)\llbracket z_0\rrbracket}\{i\} \dar{\red}\rar{\nabla}&
    \F^{[1,in-1]}\Prismhat^{(1),\nabla}_{\Oscr_K/W(\bF_q)\llbracket z_0\rrbracket}\{i\}\dar{\red^\nabla}\\
    \F^{[1,in-1]}\Prismhat^{(1)}_{R/W(\bF_q)\llbracket z_0\rrbracket}\{i\} \rar{\nabla} &
    \F^{[1,in-1]}\Prismhat^{(1),\nabla}_{R/W(\bF_q)\llbracket z_0\rrbracket}\{i\}.
\end{tikzcd}
\]
The vertical maps were computed in Step 6 and the upper horizontal map in Step 5. Using that the vertical maps are rational isomorphisms (Lemma \ref{lem:reduction_isogeny}) and that all terms are free $\bZ_p$ modules, compute the top map as $(\red^\nabla)^{-1}\circ \nabla \circ \red$.

The precision loss is bounded by the $p$-valuation of the determinant of $\red^\nabla$.

\paragraph{{\ttfamily Step 8: compute $\nabla$ on the Nygaard filtration}.}
Use $\nabla$ and the $\can$ maps to compute the unique map $\N^{\geq i}\nabla$ making the diagram below commute:
\[
\begin{tikzcd}
    \F^{[1,in-1]}\N^{\geq i}\Prism^{(1)}_{R/W(\bF_q)\llbracket z_0\rrbracket} \dar{\can}\rar{\N^{\geq
    i} \nabla} & \F^{[0,in-2]}\N^{\geq i-1}\Prism^{(1)}_{R/W(\bF_q)\llbracket z_0\rrbracket}\dar{\can}\\
    \F^{[1,in-1]}\Prism^{(1)}_{R/W(\bF_q)\llbracket z_0\rrbracket} \rar{\nabla} &
    \F^{[0,in-2]}\Prism^{(1)}_{R/W(\bF_q)\llbracket z_0\rrbracket}.
\end{tikzcd}
\]
Specifically, the $\can$ maps are rational isomorphisms by Lemma \ref{lem:can_isogeny}, so the top map
can be computed as $\can^{-1}\circ\nabla\circ\can$. The
resulting matrix has integral entries, but some $p$-adic precision has been lost, bounded by the valuation of the determinant of $\can$.

\paragraph{{\ttfamily Step 9: compute $\can-\varphi$}.}
Compute a matrix for $\can$ by reducing each of the basis elements $\widetilde{d}^{i-\sum p^ue_u} z_0^k \prod f_u^{e_u}$ in $\Prismhat^{(1)}_{R/W(k)\llbracket z_0\rrbracket}$ (replacing $\widetilde{d}$ to $E(z_0)$). Compute a matrix for $\varphi$ by reducing
\[
  \varphi_i\left(\widetilde{d}^{i-\sum p^ue_u} z_0^k \prod f_u^{e_u}\right) = z_0^{pk} \prod \varphi_{p^u}(f_u)^{e_u},
\]
using the values for $\varphi_{p^u}(f_u)$ computed in Step 3. Obtain a matrix for
\[
  \can-\varphi\colon \F^{[1,in-1]}\N^{\geq i}\Prismhat^{(1)}_{R/W(\bF_q)\llbracket z_0\rrbracket} \to 
  \F^{[1,in-1]}\Prismhat^{(1)}_{R/W(\bF_q)\llbracket z_0\rrbracket}
\]
as the difference.

\paragraph{{\ttfamily Step 10: compute $\can-\varphi^\nabla$}.}
Using that $\N^{\geq i}\nabla$ is a rational isomorphism, compute the unique map $\can-\varphi^\nabla$ such that
\[
    \begin{tikzcd}
      \F^{[1,in-1]}\N^{\geq i}\Prism^{(1)}_{R/W(\bF_q)\llbracket z_0\rrbracket}\{i\} \dar{\can-\varphi}\rar{\N^{\geq
        i} \nabla} & \F^{[1,in-1]}\N^{\geq i}\Prism^{(1),\nabla}_{R/W(\bF_q)\llbracket
        z\rrbracket}\{i\}\dar{\can-\varphi^\nabla}\\
        \F^{[1,in-1]}\Prism^{(1)}_{R/W(\bF_q)\llbracket z_0\rrbracket}\{i\} \rar{\nabla} &
        \F^{[1,in-1]}\Prism^{(1),\nabla}_{R/W(\bF_q)\llbracket z_0\rrbracket}\{i\}
    \end{tikzcd}
\]
commutes, as $\nabla\circ (\can-\varphi)\circ(\N^{\geq i}\nabla)^{-1}$.
There is a precision loss bounded by the valuation of the determinant of $\N^{\geq i}\nabla$.

\paragraph{{\ttfamily Return}:} $\syn^0=\begin{pmatrix}\can-\varphi\\\N^{\geq i}\nabla\end{pmatrix}$ and
    $\syn^1=\begin{pmatrix}\nabla,&-\can+\varphi^\nabla\end{pmatrix}$ at the target $p$-adic precision.

\subsection{The $p$-adic precision}\label{sec:precision}

In order to produce correct results, we give an {\em a priori} way of
specifying the $p$-adic precision for the algorithm by tracking bounds on the precision
loss of the reductions explained in Remark~\ref{rem:complete_the_squares}.

\subsubsection{The target $p$-adic precision}\label{sec:targetprecision}

Let $R=\Oscr_K/\varpi^n$.
We have found that $\bZ_p(i)(R)\we\F^{[1,in-1]}\bZ_p(i)(R)$, which is the total complex of
\[
    \begin{tikzcd}
      \F^{[1,in-1]} \N^{\geq i} \Prismhat^{(1)}_{R/W(k)\llbracket z_0\rrbracket}\{i\}\rar{\N^{\geq i}\nabla}\dar{\can-\varphi} & \F^{[1,in-1]}\N^{\geq i}\Prismhat^{(1),\nabla}_{R/W(k)\llbracket z_0\rrbracket}\{i\} \dar{\can - \varphi^\nabla}\\
      \F^{[1,in-1]}\Prismhat^{(1)}_{R/W(k)\llbracket z_0\rrbracket}\{i\}\rar{\nabla} & \F^{[1,in-1]}\Prismhat^{(1),\nabla}_{R/W(k)\llbracket z_0\rrbracket}\{i\}.
    \end{tikzcd}
\]
By  Lemma \ref{lem:isogeny} (or by crystalline degeneration and \cite{sulyma}), the horizontal maps are injective and their cokernels have orders
\[
  \prod_{j=1}^{in-1} q^{\epsilon(i,j)}q^{v_p\{j,n\}} \text{ and }
  \prod_{j=1}^{in-1} q^{v_p\{j,n\}}.
\]
The exponents of these groups are bounded above by the same expressions with $p$ replacing $q$.

We will use the following lemma throughout our analysis of $p$-adic precision.

\begin{lemma}\label{lem:elementarydivisorprecision}
    Let $C^\bullet$ be a cochain complex of finitely generated free $\bZ_p$-modules with
    differentials $d^s\colon C^s\rightarrow C^{s+1}$. If the cohomology of $C^\bullet$ is
    annihilated by $p^N$, then the cohomology groups can be computed from knowledge of the differentials 
    up to precision $p^{N+1}$ (i.e., their reductions modulo $p^{N+1}$). If additionally the ranks
    of each differential is known (for example if $C^\bullet$ is bounded), then the cohomology
    groups can be computed from knowledge of the differentials up to precision $p^N$.
\end{lemma}

\begin{proof}
    The cohomology of $C^\bullet$ is in particular torsion. This means that each cohomology group
    $\H^s(C^\bullet)$ can be computed by finding the elementary divisors of $d^{s-1}\colon
    C^{s-1}\rightarrow C^s$. Indeed, the Smith normal form of $d^{s-1}$ is a diagonal matrix with
    diagonal entries $[1,\ldots,1,p^{a_1},\ldots,p^{a_M},0,\ldots,0]$ (up to units). The cohomology in this case
    is $\bigoplus_{j=1}^M\bZ/p^{a_j}$. By hypothesis, each $a_j\leq N$. Thus, the diagonal matrix can
    be computed by computing the Smith normal form in $\bZ/p^{N+1}$.\footnote{For the Smith normal
    form in principal ideal rings (as opposed to domains), see for
    example~\cite[Chap.~15]{brown-matrices}.}

    For the second statement, note that if we compute the elementary divisors in $\bZ/p^N$, then
    some elementary divisors in $[1,\ldots,1,p^{a_1},\ldots,p^{a_M},0\ldots,0]$ might be
    indistinguishable from zero, namely those where $a_j=N$. However, if we already know the correct
    number of zeros in the list of elementary divisors, we can still find the number of elementary
    divisors with $a_j=N$ by taking the difference from the number produced by running the Smith
    normal form algorithm modulo $p^N$ and the given number.
\end{proof}

\begin{lemma}[Target precision]\label{lem:targetprecision}
    If the matrices $\syn^0,\syn^1$ in the cochain complex
    $\bZ_p(i)(\Oscr_K/\varpi^n)^\bullet$ are computed
    to $p$-adic precision 
    \[
    \sum_{j=1}^{in-1}\epsilon(i,j)+v_p\{j,n\},
    \]
    then the cohomology can be computed by the procedure described above.
\end{lemma}

\begin{proof}
    As discussed above, the $p$-exponents of the $\F^{[1,in-1]} \N^{\geq i} \Prismhat_R\{i\}$ and
    $\F^{[1,in-1]}\Prismhat_R\{i\}$ terms are bounded by $\sum \epsilon(i,j) + v_p\{j,n\}$. Moreover,
    the ranks of $d^0$ and $d^1$ are both $f(in-1)$. So, both parts of
    Lemma~\ref{lem:elementarydivisorprecision} apply.
\end{proof}

\begin{remark}\label{rem:movingtarget}
    Lemma~\ref{lem:elementarydivisorprecision} guarantees that if a $p$-adic cochain complex is known
    up to precision $p^N$, then as long as the cohomology computed using any given representatives of the maps
    happens be torsion with $p$-exponent strictly smaller than $N$, then the cohomology computation
    is correct. This means that instead
    of running the algorithm with the target precision derived from an {\em a priori} upper bound for the
    $p$-exponent as in Lemma \ref{lem:targetprecision}, we may instead run it with an arbitrary
    target precision. As long as the computed cohomology groups have strictly smaller exponent than
    this target precision, they are correct, otherwise, we have to restart at a higher precision.
    In practice, certain specific steps in the algorithm in this approach might result in
    division-by-zero errors, which would also have to be caught.
\end{remark}

It follows syntomic square should be computed to $p$-adic precision
$$\sum_{j=1}^{in-1}\epsilon(i,j)+v_p\{j,n\},$$ which is bounded above for
example by $i(n-1)+v_p((in-1)!)$, but is typically smaller, especially as $i$
grows with respect to $n$.

\begin{lemma}[Target precision in the even range]\label{lem:targetprecisioneven}
    If the matrix $\syn^0$ in the cochain complex $\bZ_p(i)(\Oscr_K/\varpi^n)^\bullet$ is computed to
    $p$-adic precision $\lfloor\tfrac{fi(n-1)}{2}\rfloor$ and
    $\H^2(\bZ_p(i)(\Oscr_K/\varpi^n)^\bullet)=0$, then $\H^1(\bZ_p(i)(\Oscr_K/\varpi^n)^\bullet)$ is
    computable.
\end{lemma}

\begin{proof}
    Under the vanishing hypothesis, the order of $\H^1(\bZ_p(i)(\Oscr_K/\varpi^n)^\bullet)$ is
    $q^{i(n-1)}$ by Proposition~\ref{prop:angeltveit}. The elementary divisors of $\syn^0$ are
    $\alpha_1,\ldots,\alpha_{f(in-1)},0,\ldots,0$, for some $p$-powers $\alpha_j$, where there are $f(in-1)$ copies of zero.
    Since $\prod_{j=1}^{f(in-1)}\alpha_j=q^{i(n-1)}=p^{fi(n-1)}$, it is enough to compute all but the last non-zero
    elementary divisor. Let $\nu_j=v_p(\alpha_j)$. Thus, $0\leq\nu_1\leq\cdots\leq\nu_{f(in-1)-1}\leq\nu_{f(in-1)}$ is
    an increasing sequence of integers whose sum is $fi(n-1)$. The worst case is when
    $\nu_{f(in-1)-1}$ is as large as possible, which is when
    $\nu_{f(in-1)-1}=\lfloor\tfrac{fi(n-1)}{2}\rfloor$ (and
    $\nu_{f(in-1)}=\lceil\tfrac{fi(n-1)}{2}\rceil$). These can be computed to the necessary
    precision by using the Smith normal form as in the proof of
    Lemma~\ref{lem:elementarydivisorprecision}.
\end{proof}

\subsubsection{Precision loss from $\delta$}

Applying $\delta$ results in a loss of a single digit of $p$-adic precision, as we compute the $\delta$-structure from $\delta(x)=\tfrac{\varphi(x)-x^p}{p}$. Thus, each time $\delta$ is
applied in the construction of the prismatic envelopes, we lose a digit of
precision. This causes $\lambda_u$ in Step 3 above to be known with a precision $u+1$ less than the initial precision. Similarly, it causes the $R'_u$ appearing in the relations for $g_u^p$ to be known to precision $u+2$ less than the initial precision.

\begin{lemma}[Precision loss from $\delta$]\label{lem:precisionlossfromdelta}
    \begin{enumerate}
        \item[{\em (a)}] Working in $p$-adic power series with $p$-adic precision $N$,
            the units $\lambda_u$ are computed with $p$-adic precision $N-u-1$, and the $R_u'$ are computable with $p$-adic precision $N-u-2$.
        \item[{\em (b)}]
          The largest index $u$ appearing in the required relations and values of $\varphi_{p^u}(f_u)$ and $\varphi(f_u)$ is $u=M_f-1$. The largest index $u$ appearing in the required relations and values of $\varphi_{p^u}(g_u)$ and $\varphi(g_u)$ is $u=M_g-1$.
        \item[{\em (c)}] All necessary data in Step 3 is computed to $p$-adic precision
            $N-\lfloor\log_p{(in-1)}\rfloor-1=N-M_g-1$.
    \end{enumerate}
\end{lemma}

\begin{proof}
    Part (a) follows from the inductive definition of the $R_u'$ from
    Construction~\ref{cons:Ru_recursion}. The others are clear from the fact that in $\F^{[1,in-1]}$ only $f_0,\ldots, f_{M_f}$ and $g_0,\ldots,g_{M_g}$ are involved, and $f_{M_f}^p$ and $\varphi_{p^{M_f}}(f_{M_f})$ etc. vanish for weight reasons already.
\end{proof}

\subsubsection{Precision loss from passage from $\Oscr_K$ to
$\Oscr_K/\varpi^n$}\label{sec:precisionlossfromquotient}

The map
$\F^{[1,in-1]}\Prismhat^{(1)}_{\Oscr_K/W(\bF_q)\llbracket z_0\rrbracket}\{i\}\xrightarrow{\nabla}\F^{[1,in-1]}\Prismhat^{(1),\nabla}_{\Oscr_K/W(\bF_q)\llbracket z_0\rrbracket}\{i\}$
can be computed exactly (i.e., without further loss of precision). It is also straightforward to compute the reduction maps
$\F^{[1,in-1]}\Prismhat^{(1)}_{\Oscr_K/W(\bF_q)\llbracket z_0\rrbracket}\{i\}\rightarrow\F^{[1,in-1]}\Prismhat^{(1)}_{(\Oscr_K/\varpi^n)/W(\bF_q)\llbracket z_0\rrbracket}\{i\}$
and
$\F^{[1,in-1]}\Prismhat^{(1),\nabla}_{\Oscr_K/W(\bF_q)\llbracket z_0\rrbracket}\{i\}\rightarrow\F^{[1,in-1]}\Prismhat^{(1),\nabla}_{(\Oscr_K/\varpi^n)/W(\bF_q)\llbracket z_0\rrbracket}\{i\}$
without further loss of precision.

\begin{lemma}[Precision loss from $\nabla$ on $\cO_K/\varpi^n$]\label{lem:precisionlossfromquotient}
    Filling in the square
    $$\xymatrix{
      \F^{[1,in-1]}\Prismhat^{(1)}_{\Oscr_K/W(\bF_q)\llbracket z_0\rrbracket}\ar[r]^\nabla\{i\}\ar[d]^\red&\F^{[1,in-1]}\Prismhat^{(1),\nabla}_{\Oscr_K/W(\bF_q)\llbracket z_0\rrbracket}\{i\}\ar[d]^{\red^\nabla}\\
        \F^{[1,in-1]}\Prismhat^{(1)}_{(\Oscr_K/\varpi^n)/W(\bF_q)\llbracket z_0\rrbracket}\{i\}\ar@{.>}[r]^\nabla&\F^{[1,in-1]}\Prismhat^{(1),\nabla}_{(\Oscr_K/\varpi^n)/W(\bF_q)\llbracket z_0\rrbracket}\{i\}
    }$$ involves a precision loss of at most $\sum_{s=p}^{i-1}nv_p(s!)$.
\end{lemma}

\begin{proof}
By Lemma \ref{lem:reduction_isogeny}, the left vertical map $\red$ is rationally invertible. Inverting it results in a matrix with
    entries in $\bQ_p$. By considering the Smith normal form of $\red$, it is clear that its inverse has $p$-adic valuation bounded below by $-\sigma$, where $\sigma$ is the valuation of the largest elementary divisor of $\red$. This agrees with the exponent of its cokernel. Since the cokernel has order
    \[
      \prod_{j=1}^{in-1} q^{v_p(\lfloor \frac{j}{n}\rfloor!)} 
    \]
    and is a finite $W(\bF_q)$-module, its exponent is bounded by 
    \[
      \sum_{j=1}^{in-1} v_p(\lfloor \tfrac{j}{n}\rfloor!) =
      \sum_{j=pn}^{in-1} v_p(\lfloor \tfrac{j}{n}\rfloor!) = \sum_{s=p}^{i-1} nv_p(s!).
    \]
\end{proof}

\begin{remark}
    In particular, there is no precision loss in this step up to weight $i=p$.
\end{remark}

\subsubsection{Precision loss from computing the Nygaard lift of
$\nabla$}\label{sec:precisionlossfromnygaardlift}

\begin{lemma}[Precision loss from $\N^{\geq i}\nabla$]\label{lem:precisionlossfromnygaardlift}
    Filling in the square
    $$\xymatrix{
        \F^{[1,in-1]}\N^{\geq
        i}\Prismhat_{R/W(\bF_q)\llbracket z_0\rrbracket}^{(1)}\{i\}\ar@{.>}[r]^{\N^{\geq i}\nabla}\ar[d]^\can&\F^{[1,in-1]}\N^{\geq
        i}\Prismhat_{R/W(\bF_q)\llbracket z_0\rrbracket}^{(1),\nabla}\{i\}\ar[d]^\can\\
        \F^{[1,in-1]}\Prismhat^{(1)}_{R/W(\bF_q)\llbracket z_0\rrbracket}\{i\}\ar[r]^\nabla&\F^{[1,in-1]}\Prismhat^{(1),\nabla}_{R/W(\bF_q)\llbracket z_0\rrbracket}\{i\}
    }$$
    involves a precision loss of at most $n\binom{i}{2}$.
\end{lemma}

\begin{proof}
  By Lemma \ref{lem:can_isogeny}, the right-hand vertical map is rationally invertible and has cokernel of order
  \[
    \prod_{j=1}^{in-1}q^{i-1 - \lfloor \tfrac{j-1}{n}\rfloor} = \prod_{s=0}^{i-2} q^{n(i-1-s)} = q^{n \binom{i}{2}}.
  \]
  Since it is a $W(\bF_q)$-module, its exponent is bounded by $n\binom{i}{2}$, so we conclude as in the previous argument that inverting the right-hand map causes a precision loss of at most $n\binom{i}{2}$.
\end{proof}

\subsubsection{Precision loss from computing $\can-\varphi$ on the primitives}\label{sec:precisionlossfromprimitives}

\begin{lemma}[Precision loss from $\can-\varphi^\nabla$]\label{lem:precisionlossfromprimitives}
    Filling in the square
    $$\xymatrix{
        \F^{[1,in-1]}\N^{\geq
        i}\Prismhat_{R/W(\bF_q)\llbracket z_0\rrbracket}^{(1)}\{i\}\ar[r]^{\N^{\geq i}\nabla}\ar[d]^{\can-\varphi}&\F^{[1,in-1]}\N^{\geq
        i}\Prismhat_{R/W(\bF_q)\llbracket z_0\rrbracket}^{(1),\nabla}\{i\}\ar@{.>}[d]^{\can-\varphi^\nabla}\\
        \F^{[1,in-1]}\Prismhat^{(1)}_{R/W(\bF_q)\llbracket z_0\rrbracket}\{i\}\ar[r]^\nabla&\F^{[1,in-1]}\Prismhat^{(1),\nabla}_{R/W(\bF_q)\llbracket z_0\rrbracket}\{i\}
    }$$
    involves a precision loss of at most $\sum_{j=1}^{in-1}\epsilon(i,j)+v_p\{j,n\}$.
\end{lemma}

\begin{proof}
    As in the proof of Lemma~\ref{lem:precisionlossfromquotient}, the $p$-adic valuation of the
    exponent of the cokernel of the top horizontal map is bounded above by $\sum_{j=1}^{in-1}\epsilon(i,j)+v_p\{j,n\}$
    by the analysis in Section~\ref{sec:targetprecision}.
\end{proof}

\newpage
\appendix

\section{Tables}\label{appendix:tables}

In this section we include calculations for the $p$-adic $\K$-groups of $\bZ/p^n$ for some small $p$
and $n$. More extensive calculations for other rings of the form $\Oscr_K/\varpi^n$ have been made
and will be released elsewhere. The labels refer to local fields in the {\ttfamily
lmfdb}~\cite{lmfdb}.
All calculations were performed on the high-performance cluster {\ttfamily QUEST} at Northwestern
University.

Each entry in the table consists of a list of positive integers or the mark `{\ttfamily x}'. In
the former case, it is a list of the {\em exponents} in
the decomposition of a given $\K$-group into cyclic groups. For example, $\K_{15}(\bZ/3^2;\bZ_3)$
is given as {\ttfamily 1,1,6}, which means the group is isomorphic to
$\bZ/3^1\oplus\bZ/3^1\oplus\bZ/3^6$. An empty list implies the group is zero. A mark
`{\ttfamily x}' indicates that the calculation did not succeed in the allotted time (48 hours per computation) on {\ttfamily QUEST}.

\input{2.1.0.1-data}
\input{3.1.0.1-data}
\input{5.1.0.1-data}
\input{7.1.0.1-data}

\bibliographystyle{amsplain}
\addcontentsline{toc}{section}{References}
\bibliography{kzpn}

\medskip
\noindent
\textsc{Department of Mathematics, Northwestern University}\\
{\ttfamily antieau@northwestern.edu}

\medskip
\noindent
\textsc{FB Mathematik und Informatik, Universit\"at M\"unster}\\
{\ttfamily krauseac@uni-muenster.de}

\medskip
\noindent
\textsc{FB Mathematik und Informatik, Universit\"at M\"unster}\\
{\ttfamily nikolaus@uni-muenster.de}

\end{document}

%% file: 2.1.0.1-data.tex
\subsection{2.1.0.1: $\bZ/2^n$}
\scalebox{.6}{
\begin{tabular}{|r||r|r|r|r|} \hline
$\K_{{r}}\backslash n$ & $\bZ/2^2$ & $\bZ/2^3$ & $\bZ/2^4$ & $\bZ/2^5$ \\ \hline \hline
$\K_{1}$ & {{\ttfamily1}} & {{\ttfamily1,1}} & {{\ttfamily1,2}} & {{\ttfamily1,3}} \\ \hline
$\K_{2}$ & {{\ttfamily1}} & {{\ttfamily1}} & {{\ttfamily1}} & {{\ttfamily1}} \\ \hline
$\K_{3}$ & {{\ttfamily3}} & {{\ttfamily2,3}} & {{\ttfamily3,4}} & {{\ttfamily3,6}} \\ \hline
$\K_{4}$ & {{\ttfamily}} & {{\ttfamily1}} & {{\ttfamily2}} & {{\ttfamily3}} \\ \hline
$\K_{5}$ & {{\ttfamily3}} & {{\ttfamily1,6}} & {{\ttfamily1,1,9}} & {{\ttfamily1,2,12}} \\ \hline
$\K_{6}$ & {{\ttfamily}} & {{\ttfamily}} & {{\ttfamily1}} & {{\ttfamily1}} \\ \hline
$\K_{7}$ & {{\ttfamily1,3}} & {{\ttfamily4,4}} & {{\ttfamily1,4,8}} & {{\ttfamily1,1,4,11}} \\ \hline
$\K_{8}$ & {{\ttfamily}} & {{\ttfamily}} & {{\ttfamily1}} & {{\ttfamily2}} \\ \hline
$\K_{9}$ & {{\ttfamily1,1,3}} & {{\ttfamily1,2,7}} & {{\ttfamily1,1,2,12}} & {{\ttfamily1,1,1,2,17}} \\ \hline
$\K_{10}$ & {{\ttfamily}} & {{\ttfamily}} & {{\ttfamily}} & {{\ttfamily1}} \\ \hline
$\K_{11}$ & {{\ttfamily1,5}} & {{\ttfamily3,9}} & {{\ttfamily3,3,12}} & {{\ttfamily1,3,5,16}} \\ \hline
$\K_{12}$ & {{\ttfamily}} & {{\ttfamily}} & {{\ttfamily}} & {{\ttfamily1}} \\ \hline
$\K_{13}$ & {{\ttfamily1,2,4}} & {{\ttfamily1,1,3,9}} & {{\ttfamily1,1,1,3,15}} & {{\ttfamily1,1,1,1,3,22}} \\ \hline
$\K_{14}$ & {{\ttfamily}} & {{\ttfamily}} & {{\ttfamily}} & {{\ttfamily1}} \\ \hline
$\K_{15}$ & {{\ttfamily1,1,1,5}} & {{\ttfamily1,1,6,8}} & {{\ttfamily1,1,2,5,15}} & {{\ttfamily1,1,2,3,5,21}} \\ \hline
$\K_{16}$ & {{\ttfamily}} & {{\ttfamily}} & {{\ttfamily}} & {{\ttfamily1}} \\ \hline
$\K_{17}$ & {{\ttfamily1,1,1,3,3}} & {{\ttfamily1,1,2,2,3,9}} & {{\ttfamily1,2,2,2,3,17}} & {{\ttfamily1,1,2,2,2,3,26}} \\ \hline
$\K_{18}$ & {{\ttfamily}} & {{\ttfamily}} & {{\ttfamily}} & {{\ttfamily}} \\ \hline
$\K_{19}$ & {{\ttfamily2,3,5}} & {{\ttfamily1,3,4,12}} & {{\ttfamily3,3,4,20}} & {{\ttfamily3,3,3,4,27}} \\ \hline
$\K_{20}$ & {{\ttfamily}} & {{\ttfamily}} & {{\ttfamily}} & {{\ttfamily}} \\ \hline
$\K_{21}$ & {{\ttfamily1,1,2,2,5}} & {{\ttfamily1,1,1,2,2,3,12}} & {{\ttfamily1,2,2,2,2,4,20}} & {{\ttfamily1,1,2,2,2,2,4,30}} \\ \hline
$\K_{22}$ & {{\ttfamily}} & {{\ttfamily}} & {{\ttfamily}} & {{\ttfamily}} \\ \hline
$\K_{23}$ & {{\ttfamily1,1,1,1,2,6}} & {{\ttfamily1,1,1,2,5,14}} & {{\ttfamily1,1,1,1,4,6,22}} & {{\ttfamily1,1,1,1,3,4,6,31}} \\ \hline
$\K_{24}$ & {{\ttfamily}} & {{\ttfamily}} & {{\ttfamily}} & {{\ttfamily x}} \\ \hline
$\K_{25}$ & {{\ttfamily1,1,1,1,2,3,4}} & {{\ttfamily1,1,1,2,2,3,4,12}} & {{\ttfamily1,1,1,1,2,3,3,4,23}} & {{\ttfamily x}} \\ \hline
$\K_{26}$ & {{\ttfamily}} & {{\ttfamily}} & {{\ttfamily}} & {{\ttfamily x}} \\ \hline
$\K_{27}$ & {{\ttfamily1,3,4,6}} & {{\ttfamily1,3,3,6,15}} & {{\ttfamily1,3,3,3,6,26}} & {{\ttfamily x}} \\ \hline
$\K_{28}$ & {{\ttfamily}} & {{\ttfamily}} & {{\ttfamily}} & {{\ttfamily x}} \\ \hline
$\K_{29}$ & {{\ttfamily1,1,1,2,2,3,5}} & {{\ttfamily1,1,1,1,2,2,3,4,15}} & {{\ttfamily1,1,1,1,2,2,3,4,4,26}} & {{\ttfamily x}} \\ \hline
$\K_{30}$ & {{\ttfamily}} & {{\ttfamily}} & {{\ttfamily}} & {{\ttfamily x}} \\ \hline
$\K_{31}$ & {{\ttfamily1,1,1,1,1,1,3,7}} & {{\ttfamily1,1,1,1,3,3,8,14}} & {{\ttfamily1,1,1,1,1,3,3,3,7,27}} & {{\ttfamily x}} \\ \hline
$\K_{32}$ & {{\ttfamily}} & {{\ttfamily}} & {{\ttfamily x}} & {{\ttfamily x}} \\ \hline
$\K_{33}$ & {{\ttfamily1,1,1,1,1,2,2,3,5}} & {{\ttfamily1,1,1,1,2,2,2,2,4,4,14}} & {{\ttfamily x}} & {{\ttfamily x}} \\ \hline
$\K_{34}$ & {{\ttfamily}} & {{\ttfamily}} & {{\ttfamily x}} & {{\ttfamily x}} \\ \hline
$\K_{35}$ & {{\ttfamily2,3,3,5,5}} & {{\ttfamily1,2,3,4,5,5,16}} & {{\ttfamily x}} & {{\ttfamily x}} \\ \hline
$\K_{36}$ & {{\ttfamily}} & {{\ttfamily}} & {{\ttfamily x}} & {{\ttfamily x}} \\ \hline
$\K_{37}$ & {{\ttfamily1,1,1,1,2,2,3,3,5}} & {{\ttfamily1,1,1,1,1,2,2,2,3,3,4,17}} & {{\ttfamily x}} & {{\ttfamily x}} \\ \hline
$\K_{38}$ & {{\ttfamily}} & {{\ttfamily}} & {{\ttfamily x}} & {{\ttfamily x}} \\ \hline
$\K_{39}$ & {{\ttfamily1,1,1,1,1,1,1,2,4,7}} & {{\ttfamily1,1,1,1,1,2,3,5,7,18}} & {{\ttfamily x}} & {{\ttfamily x}} \\ \hline
$\K_{40}$ & {{\ttfamily}} & {{\ttfamily}} & {{\ttfamily x}} & {{\ttfamily x}} \\ \hline
$\K_{41}$ & {{\ttfamily1,1,1,1,1,1,2,2,3,3,5}} & {{\ttfamily1,1,1,1,1,2,2,2,2,3,3,6,17}} & {{\ttfamily x}} & {{\ttfamily x}} \\ \hline
$\K_{42}$ & {{\ttfamily}} & {{\ttfamily x}} & {{\ttfamily x}} & {{\ttfamily x}} \\ \hline
$\K_{43}$ & {{\ttfamily1,3,3,4,4,7}} & {{\ttfamily x}} & {{\ttfamily x}} & {{\ttfamily x}} \\ \hline
$\K_{44}$ & {{\ttfamily}} & {{\ttfamily x}} & {{\ttfamily x}} & {{\ttfamily x}} \\ \hline
$\K_{45}$ & {{\ttfamily1,1,1,1,1,2,2,2,3,3,6}} & {{\ttfamily x}} & {{\ttfamily x}} & {{\ttfamily x}} \\ \hline
$\K_{46}$ & {{\ttfamily}} & {{\ttfamily x}} & {{\ttfamily x}} & {{\ttfamily x}} \\ \hline
$\K_{47}$ & {{\ttfamily1,1,1,1,1,1,1,1,1,3,5,7}} & {{\ttfamily x}} & {{\ttfamily x}} & {{\ttfamily x}} \\ \hline
$\K_{48}$ & {{\ttfamily}} & {{\ttfamily x}} & {{\ttfamily x}} & {{\ttfamily x}} \\ \hline
$\K_{49}$ & {{\ttfamily1,1,1,1,1,1,1,2,2,3,3,3,5}} & {{\ttfamily x}} & {{\ttfamily x}} & {{\ttfamily x}} \\ \hline
$\K_{50}$ & {{\ttfamily}} & {{\ttfamily x}} & {{\ttfamily x}} & {{\ttfamily x}} \\ \hline
$\K_{51}$ & {{\ttfamily2,3,3,3,4,5,6}} & {{\ttfamily x}} & {{\ttfamily x}} & {{\ttfamily x}} \\ \hline
$\K_{52}$ & {{\ttfamily}} & {{\ttfamily x}} & {{\ttfamily x}} & {{\ttfamily x}} \\ \hline
$\K_{53}$ & {{\ttfamily1,1,1,1,1,1,2,2,2,2,3,4,6}} & {{\ttfamily x}} & {{\ttfamily x}} & {{\ttfamily x}} \\ \hline
$\K_{54}$ & {{\ttfamily}} & {{\ttfamily x}} & {{\ttfamily x}} & {{\ttfamily x}} \\ \hline
$\K_{55}$ & {{\ttfamily1,1,1,1,1,1,1,1,1,1,2,3,6,7}} & {{\ttfamily x}} & {{\ttfamily x}} & {{\ttfamily x}} \\ \hline
$\K_{56}$ & {{\ttfamily}} & {{\ttfamily x}} & {{\ttfamily x}} & {{\ttfamily x}} \\ \hline
$\K_{57}$ & {{\ttfamily1,1,1,1,1,1,1,1,2,2,2,3,3,4,5}} & {{\ttfamily x}} & {{\ttfamily x}} & {{\ttfamily x}} \\ \hline
$\K_{58}$ & {{\ttfamily}} & {{\ttfamily x}} & {{\ttfamily x}} & {{\ttfamily x}} \\ \hline
$\K_{59}$ & {{\ttfamily1,3,3,3,4,4,5,7}} & {{\ttfamily x}} & {{\ttfamily x}} & {{\ttfamily x}} \\ \hline
\end{tabular}}

%% file: 3.1.0.1-data.tex
\subsection{3.1.0.1: $\bZ/3^n$}
\scalebox{.6}{
\begin{tabular}{|r||r|r|r|r|} \hline
$\K_{{r}}\backslash n$ & $\bZ/3^2$ & $\bZ/3^3$ & $\bZ/3^4$ & $\bZ/3^5$ \\ \hline \hline
$\K_{1}$ & {{\ttfamily1}} & {{\ttfamily2}} & {{\ttfamily3}} & {{\ttfamily4}} \\ \hline
$\K_{2}$ & {{\ttfamily}} & {{\ttfamily}} & {{\ttfamily}} & {{\ttfamily}} \\ \hline
$\K_{3}$ & {{\ttfamily1,1}} & {{\ttfamily1,3}} & {{\ttfamily1,5}} & {{\ttfamily1,7}} \\ \hline
$\K_{4}$ & {{\ttfamily1}} & {{\ttfamily1}} & {{\ttfamily1}} & {{\ttfamily1}} \\ \hline
$\K_{5}$ & {{\ttfamily4}} & {{\ttfamily7}} & {{\ttfamily10}} & {{\ttfamily13}} \\ \hline
$\K_{6}$ & {{\ttfamily}} & {{\ttfamily}} & {{\ttfamily}} & {{\ttfamily}} \\ \hline
$\K_{7}$ & {{\ttfamily1,3}} & {{\ttfamily1,1,6}} & {{\ttfamily1,2,9}} & {{\ttfamily1,2,13}} \\ \hline
$\K_{8}$ & {{\ttfamily}} & {{\ttfamily1}} & {{\ttfamily1}} & {{\ttfamily1}} \\ \hline
$\K_{9}$ & {{\ttfamily1,4}} & {{\ttfamily1,1,9}} & {{\ttfamily1,1,14}} & {{\ttfamily1,1,19}} \\ \hline
$\K_{10}$ & {{\ttfamily}} & {{\ttfamily}} & {{\ttfamily}} & {{\ttfamily}} \\ \hline
$\K_{11}$ & {{\ttfamily3,3}} & {{\ttfamily1,2,9}} & {{\ttfamily1,1,2,14}} & {{\ttfamily1,1,2,20}} \\ \hline
$\K_{12}$ & {{\ttfamily}} & {{\ttfamily1}} & {{\ttfamily2}} & {{\ttfamily2}} \\ \hline
$\K_{13}$ & {{\ttfamily1,1,5}} & {{\ttfamily1,1,1,12}} & {{\ttfamily1,1,2,19}} & {{\ttfamily1,1,2,26}} \\ \hline
$\K_{14}$ & {{\ttfamily}} & {{\ttfamily}} & {{\ttfamily}} & {{\ttfamily}} \\ \hline
$\K_{15}$ & {{\ttfamily1,1,6}} & {{\ttfamily1,1,2,12}} & {{\ttfamily1,1,1,2,19}} & {{\ttfamily1,1,1,2,27}} \\ \hline
$\K_{16}$ & {{\ttfamily}} & {{\ttfamily}} & {{\ttfamily1}} & {{\ttfamily1}} \\ \hline
$\K_{17}$ & {{\ttfamily2,7}} & {{\ttfamily1,2,15}} & {{\ttfamily2,2,24}} & {{\ttfamily2,2,33}} \\ \hline
$\K_{18}$ & {{\ttfamily}} & {{\ttfamily}} & {{\ttfamily}} & {{\ttfamily}} \\ \hline
$\K_{19}$ & {{\ttfamily1,1,1,2,5}} & {{\ttfamily1,1,1,1,2,14}} & {{\ttfamily1,1,1,1,1,2,23}} & {{\ttfamily1,1,1,1,2,2,32}} \\ \hline
$\K_{20}$ & {{\ttfamily}} & {{\ttfamily}} & {{\ttfamily1}} & {{\ttfamily1}} \\ \hline
$\K_{21}$ & {{\ttfamily1,1,1,2,6}} & {{\ttfamily1,1,2,3,15}} & {{\ttfamily1,1,1,2,3,26}} & {{\ttfamily1,1,1,2,3,37}} \\ \hline
$\K_{22}$ & {{\ttfamily}} & {{\ttfamily}} & {{\ttfamily}} & {{\ttfamily}} \\ \hline
$\K_{23}$ & {{\ttfamily1,3,8}} & {{\ttfamily1,1,2,3,17}} & {{\ttfamily1,1,1,2,4,27}} & {{\ttfamily1,1,1,2,5,38}} \\ \hline
$\K_{24}$ & {{\ttfamily}} & {{\ttfamily}} & {{\ttfamily1}} & {{\ttfamily2}} \\ \hline
$\K_{25}$ & {{\ttfamily1,1,1,1,2,7}} & {{\ttfamily1,1,1,2,3,18}} & {{\ttfamily1,1,1,1,2,3,31}} & {{\ttfamily1,1,1,1,3,3,44}} \\ \hline
$\K_{26}$ & {{\ttfamily}} & {{\ttfamily}} & {{\ttfamily}} & {{\ttfamily}} \\ \hline
$\K_{27}$ & {{\ttfamily1,1,1,1,2,8}} & {{\ttfamily1,1,1,1,1,2,2,19}} & {{\ttfamily1,1,1,1,1,1,2,2,32}} & {{\ttfamily1,1,1,1,1,2,2,2,45}} \\ \hline
$\K_{28}$ & {{\ttfamily}} & {{\ttfamily}} & {{\ttfamily1}} & {{\ttfamily1}} \\ \hline
$\K_{29}$ & {{\ttfamily1,2,3,9}} & {{\ttfamily1,1,3,4,21}} & {{\ttfamily1,1,1,3,4,36}} & {{\ttfamily1,1,1,3,4,51}} \\ \hline
$\K_{30}$ & {{\ttfamily}} & {{\ttfamily}} & {{\ttfamily}} & {{\ttfamily}} \\ \hline
$\K_{31}$ & {{\ttfamily1,1,1,1,2,2,8}} & {{\ttfamily1,1,1,1,1,1,2,2,22}} & {{\ttfamily1,1,1,1,1,1,1,2,2,37}} & {{\ttfamily1,1,1,1,1,1,2,2,2,52}} \\ \hline
$\K_{32}$ & {{\ttfamily}} & {{\ttfamily}} & {{\ttfamily1}} & {{\ttfamily x}} \\ \hline
$\K_{33}$ & {{\ttfamily1,1,1,1,2,2,9}} & {{\ttfamily1,1,1,2,2,2,2,23}} & {{\ttfamily1,1,1,1,2,2,2,2,40}} & {{\ttfamily x}} \\ \hline
$\K_{34}$ & {{\ttfamily}} & {{\ttfamily}} & {{\ttfamily}} & {{\ttfamily x}} \\ \hline
$\K_{35}$ & {{\ttfamily2,2,5,9}} & {{\ttfamily2,2,2,2,4,24}} & {{\ttfamily2,2,2,2,3,3,40}} & {{\ttfamily x}} \\ \hline
$\K_{36}$ & {{\ttfamily}} & {{\ttfamily}} & {{\ttfamily1}} & {{\ttfamily x}} \\ \hline
$\K_{37}$ & {{\ttfamily1,1,1,1,1,2,2,2,8}} & {{\ttfamily1,1,1,1,1,2,2,2,2,25}} & {{\ttfamily1,1,1,1,1,1,2,2,2,2,44}} & {{\ttfamily x}} \\ \hline
$\K_{38}$ & {{\ttfamily}} & {{\ttfamily}} & {{\ttfamily}} & {{\ttfamily x}} \\ \hline
$\K_{39}$ & {{\ttfamily1,1,1,1,1,1,2,2,10}} & {{\ttfamily1,1,1,1,1,2,2,2,2,27}} & {{\ttfamily1,1,1,1,1,1,2,2,2,3,45}} & {{\ttfamily x}} \\ \hline
$\K_{40}$ & {{\ttfamily}} & {{\ttfamily}} & {{\ttfamily x}} & {{\ttfamily x}} \\ \hline
$\K_{41}$ & {{\ttfamily1,2,3,3,12}} & {{\ttfamily1,1,1,3,3,4,29}} & {{\ttfamily x}} & {{\ttfamily x}} \\ \hline
$\K_{42}$ & {{\ttfamily}} & {{\ttfamily}} & {{\ttfamily x}} & {{\ttfamily x}} \\ \hline
$\K_{43}$ & {{\ttfamily1,1,1,1,1,1,1,2,2,11}} & {{\ttfamily1,1,1,1,1,2,2,2,2,3,28}} & {{\ttfamily x}} & {{\ttfamily x}} \\ \hline
$\K_{44}$ & {{\ttfamily}} & {{\ttfamily}} & {{\ttfamily x}} & {{\ttfamily x}} \\ \hline
$\K_{45}$ & {{\ttfamily1,1,1,1,1,1,1,2,2,12}} & {{\ttfamily1,1,1,1,1,1,1,1,2,3,3,30}} & {{\ttfamily x}} & {{\ttfamily x}} \\ \hline
$\K_{46}$ & {{\ttfamily}} & {{\ttfamily}} & {{\ttfamily x}} & {{\ttfamily x}} \\ \hline
$\K_{47}$ & {{\ttfamily1,2,2,3,3,13}} & {{\ttfamily1,1,1,1,1,3,3,5,32}} & {{\ttfamily x}} & {{\ttfamily x}} \\ \hline
$\K_{48}$ & {{\ttfamily}} & {{\ttfamily}} & {{\ttfamily x}} & {{\ttfamily x}} \\ \hline
$\K_{49}$ & {{\ttfamily1,1,1,1,1,1,1,2,2,2,12}} & {{\ttfamily1,1,1,1,1,1,1,1,2,2,2,3,33}} & {{\ttfamily x}} & {{\ttfamily x}} \\ \hline
$\K_{50}$ & {{\ttfamily}} & {{\ttfamily}} & {{\ttfamily x}} & {{\ttfamily x}} \\ \hline
$\K_{51}$ & {{\ttfamily1,1,1,1,1,1,1,2,2,2,13}} & {{\ttfamily1,1,1,1,1,1,1,1,2,2,2,3,35}} & {{\ttfamily x}} & {{\ttfamily x}} \\ \hline
$\K_{52}$ & {{\ttfamily}} & {{\ttfamily}} & {{\ttfamily x}} & {{\ttfamily x}} \\ \hline
$\K_{53}$ & {{\ttfamily2,2,2,2,4,15}} & {{\ttfamily1,1,1,2,2,2,2,2,4,37}} & {{\ttfamily x}} & {{\ttfamily x}} \\ \hline
$\K_{54}$ & {{\ttfamily}} & {{\ttfamily x}} & {{\ttfamily x}} & {{\ttfamily x}} \\ \hline
$\K_{55}$ & {{\ttfamily1,1,1,1,1,1,1,1,1,2,2,2,13}} & {{\ttfamily x}} & {{\ttfamily x}} & {{\ttfamily x}} \\ \hline
$\K_{56}$ & {{\ttfamily}} & {{\ttfamily x}} & {{\ttfamily x}} & {{\ttfamily x}} \\ \hline
$\K_{57}$ & {{\ttfamily1,1,1,1,1,1,1,1,2,2,2,3,12}} & {{\ttfamily x}} & {{\ttfamily x}} & {{\ttfamily x}} \\ \hline
$\K_{58}$ & {{\ttfamily}} & {{\ttfamily x}} & {{\ttfamily x}} & {{\ttfamily x}} \\ \hline
$\K_{59}$ & {{\ttfamily1,2,3,3,3,4,14}} & {{\ttfamily x}} & {{\ttfamily x}} & {{\ttfamily x}} \\ \hline
$\K_{60}$ & {{\ttfamily}} & {{\ttfamily x}} & {{\ttfamily x}} & {{\ttfamily x}} \\ \hline
$\K_{61}$ & {{\ttfamily1,1,1,1,1,1,1,1,1,2,2,2,3,13}} & {{\ttfamily x}} & {{\ttfamily x}} & {{\ttfamily x}} \\ \hline
$\K_{62}$ & {{\ttfamily}} & {{\ttfamily x}} & {{\ttfamily x}} & {{\ttfamily x}} \\ \hline
$\K_{63}$ & {{\ttfamily1,1,1,1,1,1,1,1,1,2,2,2,3,14}} & {{\ttfamily x}} & {{\ttfamily x}} & {{\ttfamily x}} \\ \hline
$\K_{64}$ & {{\ttfamily}} & {{\ttfamily x}} & {{\ttfamily x}} & {{\ttfamily x}} \\ \hline
$\K_{65}$ & {{\ttfamily1,2,2,3,3,3,4,15}} & {{\ttfamily x}} & {{\ttfamily x}} & {{\ttfamily x}} \\ \hline
$\K_{66}$ & {{\ttfamily}} & {{\ttfamily x}} & {{\ttfamily x}} & {{\ttfamily x}} \\ \hline
$\K_{67}$ & {{\ttfamily1,1,1,1,1,1,1,1,1,1,2,2,2,3,15}} & {{\ttfamily x}} & {{\ttfamily x}} & {{\ttfamily x}} \\ \hline
$\K_{68}$ & {{\ttfamily}} & {{\ttfamily x}} & {{\ttfamily x}} & {{\ttfamily x}} \\ \hline
$\K_{69}$ & {{\ttfamily1,1,1,1,1,1,1,1,1,1,2,2,2,3,16}} & {{\ttfamily x}} & {{\ttfamily x}} & {{\ttfamily x}} \\ \hline
$\K_{70}$ & {{\ttfamily}} & {{\ttfamily x}} & {{\ttfamily x}} & {{\ttfamily x}} \\ \hline
$\K_{71}$ & {{\ttfamily2,2,2,2,2,4,5,17}} & {{\ttfamily x}} & {{\ttfamily x}} & {{\ttfamily x}} \\ \hline
$\K_{72}$ & {{\ttfamily}} & {{\ttfamily x}} & {{\ttfamily x}} & {{\ttfamily x}} \\ \hline
$\K_{73}$ & {{\ttfamily1,1,1,1,1,1,1,1,1,1,1,2,2,2,2,3,15}} & {{\ttfamily x}} & {{\ttfamily x}} & {{\ttfamily x}} \\ \hline
$\K_{74}$ & {{\ttfamily}} & {{\ttfamily x}} & {{\ttfamily x}} & {{\ttfamily x}} \\ \hline
$\K_{75}$ & {{\ttfamily1,1,1,1,1,1,1,1,1,1,1,2,2,2,2,3,16}} & {{\ttfamily x}} & {{\ttfamily x}} & {{\ttfamily x}} \\ \hline
$\K_{76}$ & {{\ttfamily}} & {{\ttfamily x}} & {{\ttfamily x}} & {{\ttfamily x}} \\ \hline
$\K_{77}$ & {{\ttfamily1,2,2,3,3,3,3,4,18}} & {{\ttfamily x}} & {{\ttfamily x}} & {{\ttfamily x}} \\ \hline
$\K_{78}$ & {{\ttfamily}} & {{\ttfamily x}} & {{\ttfamily x}} & {{\ttfamily x}} \\ \hline
$\K_{79}$ & {{\ttfamily1,1,1,1,1,1,1,1,1,1,1,1,2,2,2,2,3,17}} & {{\ttfamily x}} & {{\ttfamily x}} & {{\ttfamily x}} \\ \hline
\end{tabular}}

%% file: 5.1.0.1-data.tex
\subsection{5.1.0.1: $\bZ/5^n$}
\scalebox{.55}{
\begin{tabular}{|r||r|r|r|r|r|} \hline
$\K_{{r}}\backslash n$ & $\bZ/5^2$ & $\bZ/5^3$ & $\bZ/5^4$ & $\bZ/5^5$ & $\bZ/5^6$ \\ \hline \hline
$\K_{1}$ & {{\ttfamily1}} & {{\ttfamily2}} & {{\ttfamily3}} & {{\ttfamily4}} & {{\ttfamily5}} \\ \hline
$\K_{2}$ & {{\ttfamily}} & {{\ttfamily}} & {{\ttfamily}} & {{\ttfamily}} & {{\ttfamily}} \\ \hline
$\K_{3}$ & {{\ttfamily2}} & {{\ttfamily4}} & {{\ttfamily6}} & {{\ttfamily8}} & {{\ttfamily10}} \\ \hline
$\K_{4}$ & {{\ttfamily}} & {{\ttfamily}} & {{\ttfamily}} & {{\ttfamily}} & {{\ttfamily}} \\ \hline
$\K_{5}$ & {{\ttfamily3}} & {{\ttfamily6}} & {{\ttfamily9}} & {{\ttfamily12}} & {{\ttfamily15}} \\ \hline
$\K_{6}$ & {{\ttfamily}} & {{\ttfamily}} & {{\ttfamily}} & {{\ttfamily}} & {{\ttfamily}} \\ \hline
$\K_{7}$ & {{\ttfamily1,3}} & {{\ttfamily1,7}} & {{\ttfamily1,11}} & {{\ttfamily1,15}} & {{\ttfamily1,19}} \\ \hline
$\K_{8}$ & {{\ttfamily1}} & {{\ttfamily1}} & {{\ttfamily1}} & {{\ttfamily1}} & {{\ttfamily1}} \\ \hline
$\K_{9}$ & {{\ttfamily6}} & {{\ttfamily11}} & {{\ttfamily16}} & {{\ttfamily21}} & {{\ttfamily26}} \\ \hline
$\K_{10}$ & {{\ttfamily}} & {{\ttfamily}} & {{\ttfamily}} & {{\ttfamily}} & {{\ttfamily}} \\ \hline
$\K_{11}$ & {{\ttfamily1,5}} & {{\ttfamily2,10}} & {{\ttfamily2,16}} & {{\ttfamily2,22}} & {{\ttfamily2,28}} \\ \hline
$\K_{12}$ & {{\ttfamily}} & {{\ttfamily}} & {{\ttfamily}} & {{\ttfamily}} & {{\ttfamily}} \\ \hline
$\K_{13}$ & {{\ttfamily1,1,5}} & {{\ttfamily1,1,12}} & {{\ttfamily1,1,19}} & {{\ttfamily1,1,26}} & {{\ttfamily1,1,33}} \\ \hline
$\K_{14}$ & {{\ttfamily}} & {{\ttfamily}} & {{\ttfamily}} & {{\ttfamily}} & {{\ttfamily}} \\ \hline
$\K_{15}$ & {{\ttfamily1,1,6}} & {{\ttfamily1,1,1,13}} & {{\ttfamily1,1,1,21}} & {{\ttfamily1,1,1,29}} & {{\ttfamily1,1,1,37}} \\ \hline
$\K_{16}$ & {{\ttfamily}} & {{\ttfamily1}} & {{\ttfamily1}} & {{\ttfamily1}} & {{\ttfamily1}} \\ \hline
$\K_{17}$ & {{\ttfamily1,8}} & {{\ttfamily1,1,17}} & {{\ttfamily1,1,26}} & {{\ttfamily1,1,35}} & {{\ttfamily1,1,44}} \\ \hline
$\K_{18}$ & {{\ttfamily}} & {{\ttfamily}} & {{\ttfamily}} & {{\ttfamily}} & {{\ttfamily}} \\ \hline
$\K_{19}$ & {{\ttfamily1,9}} & {{\ttfamily1,1,18}} & {{\ttfamily1,1,28}} & {{\ttfamily1,1,38}} & {{\ttfamily1,1,48}} \\ \hline
$\K_{20}$ & {{\ttfamily}} & {{\ttfamily}} & {{\ttfamily}} & {{\ttfamily}} & {{\ttfamily}} \\ \hline
$\K_{21}$ & {{\ttfamily1,1,1,8}} & {{\ttfamily1,1,2,18}} & {{\ttfamily1,1,2,29}} & {{\ttfamily1,1,2,40}} & {{\ttfamily1,1,2,51}} \\ \hline
$\K_{22}$ & {{\ttfamily}} & {{\ttfamily}} & {{\ttfamily}} & {{\ttfamily}} & {{\ttfamily}} \\ \hline
$\K_{23}$ & {{\ttfamily1,1,1,9}} & {{\ttfamily1,1,1,2,19}} & {{\ttfamily1,1,1,2,31}} & {{\ttfamily1,1,1,2,43}} & {{\ttfamily1,1,1,2,55}} \\ \hline
$\K_{24}$ & {{\ttfamily}} & {{\ttfamily1}} & {{\ttfamily1}} & {{\ttfamily1}} & {{\ttfamily1}} \\ \hline
$\K_{25}$ & {{\ttfamily1,1,1,10}} & {{\ttfamily1,1,1,1,23}} & {{\ttfamily1,1,1,1,36}} & {{\ttfamily1,1,1,1,49}} & {{\ttfamily1,1,1,1,62}} \\ \hline
$\K_{26}$ & {{\ttfamily}} & {{\ttfamily}} & {{\ttfamily}} & {{\ttfamily}} & {{\ttfamily}} \\ \hline
$\K_{27}$ & {{\ttfamily1,1,1,11}} & {{\ttfamily1,1,1,1,24}} & {{\ttfamily1,1,1,1,38}} & {{\ttfamily1,1,1,1,52}} & {{\ttfamily1,1,1,1,66}} \\ \hline
$\K_{28}$ & {{\ttfamily}} & {{\ttfamily}} & {{\ttfamily}} & {{\ttfamily}} & {{\ttfamily}} \\ \hline
$\K_{29}$ & {{\ttfamily1,2,12}} & {{\ttfamily2,2,26}} & {{\ttfamily2,2,41}} & {{\ttfamily2,2,56}} & {{\ttfamily2,2,71}} \\ \hline
$\K_{30}$ & {{\ttfamily}} & {{\ttfamily}} & {{\ttfamily}} & {{\ttfamily}} & {{\ttfamily}} \\ \hline
$\K_{31}$ & {{\ttfamily1,1,1,1,1,11}} & {{\ttfamily1,1,1,1,1,1,26}} & {{\ttfamily1,1,1,1,1,2,41}} & {{\ttfamily1,1,1,1,1,2,57}} & {{\ttfamily1,1,1,1,1,2,73}} \\ \hline
$\K_{32}$ & {{\ttfamily}} & {{\ttfamily1}} & {{\ttfamily1}} & {{\ttfamily1}} & {{\ttfamily1}} \\ \hline
$\K_{33}$ & {{\ttfamily1,1,1,2,12}} & {{\ttfamily1,1,1,1,2,29}} & {{\ttfamily1,1,1,1,2,46}} & {{\ttfamily1,1,1,1,2,63}} & {{\ttfamily1,1,1,1,2,80}} \\ \hline
$\K_{34}$ & {{\ttfamily}} & {{\ttfamily}} & {{\ttfamily}} & {{\ttfamily}} & {{\ttfamily}} \\ \hline
$\K_{35}$ & {{\ttfamily1,1,1,2,13}} & {{\ttfamily1,1,1,1,2,30}} & {{\ttfamily1,1,1,1,2,48}} & {{\ttfamily1,1,1,1,2,66}} & {{\ttfamily1,1,1,1,2,84}} \\ \hline
$\K_{36}$ & {{\ttfamily}} & {{\ttfamily}} & {{\ttfamily}} & {{\ttfamily}} & {{\ttfamily}} \\ \hline
$\K_{37}$ & {{\ttfamily1,1,1,1,1,14}} & {{\ttfamily1,1,1,1,1,1,32}} & {{\ttfamily1,1,1,1,1,1,51}} & {{\ttfamily1,1,1,1,1,1,70}} & {{\ttfamily1,1,1,1,1,1,89}} \\ \hline
$\K_{38}$ & {{\ttfamily}} & {{\ttfamily}} & {{\ttfamily}} & {{\ttfamily}} & {{\ttfamily x}} \\ \hline
$\K_{39}$ & {{\ttfamily1,2,3,14}} & {{\ttfamily1,1,1,2,2,33}} & {{\ttfamily1,1,2,2,2,52}} & {{\ttfamily1,1,2,2,2,72}} & {{\ttfamily x}} \\ \hline
$\K_{40}$ & {{\ttfamily}} & {{\ttfamily1}} & {{\ttfamily2}} & {{\ttfamily2}} & {{\ttfamily x}} \\ \hline
$\K_{41}$ & {{\ttfamily1,1,1,1,1,1,15}} & {{\ttfamily1,1,1,1,1,1,1,36}} & {{\ttfamily1,1,1,1,1,1,2,57}} & {{\ttfamily1,1,1,1,1,1,2,78}} & {{\ttfamily x}} \\ \hline
$\K_{42}$ & {{\ttfamily}} & {{\ttfamily}} & {{\ttfamily}} & {{\ttfamily x}} & {{\ttfamily x}} \\ \hline
$\K_{43}$ & {{\ttfamily1,1,1,1,1,1,16}} & {{\ttfamily1,1,1,1,1,1,2,36}} & {{\ttfamily1,1,1,1,1,1,2,58}} & {{\ttfamily x}} & {{\ttfamily x}} \\ \hline
$\K_{44}$ & {{\ttfamily}} & {{\ttfamily}} & {{\ttfamily}} & {{\ttfamily x}} & {{\ttfamily x}} \\ \hline
$\K_{45}$ & {{\ttfamily1,1,1,1,1,1,17}} & {{\ttfamily1,1,1,1,1,1,2,38}} & {{\ttfamily1,1,1,1,1,1,2,61}} & {{\ttfamily x}} & {{\ttfamily x}} \\ \hline
$\K_{46}$ & {{\ttfamily}} & {{\ttfamily}} & {{\ttfamily}} & {{\ttfamily x}} & {{\ttfamily x}} \\ \hline
$\K_{47}$ & {{\ttfamily1,1,1,1,1,1,18}} & {{\ttfamily1,1,1,1,1,1,2,40}} & {{\ttfamily1,1,1,1,1,1,1,2,63}} & {{\ttfamily x}} & {{\ttfamily x}} \\ \hline
$\K_{48}$ & {{\ttfamily}} & {{\ttfamily}} & {{\ttfamily1}} & {{\ttfamily x}} & {{\ttfamily x}} \\ \hline
$\K_{49}$ & {{\ttfamily2,2,2,19}} & {{\ttfamily1,2,2,2,43}} & {{\ttfamily2,2,2,2,68}} & {{\ttfamily x}} & {{\ttfamily x}} \\ \hline
$\K_{50}$ & {{\ttfamily}} & {{\ttfamily}} & {{\ttfamily x}} & {{\ttfamily x}} & {{\ttfamily x}} \\ \hline
$\K_{51}$ & {{\ttfamily1,1,1,1,1,1,1,2,17}} & {{\ttfamily1,1,1,1,1,1,1,1,2,42}} & {{\ttfamily x}} & {{\ttfamily x}} & {{\ttfamily x}} \\ \hline
$\K_{52}$ & {{\ttfamily}} & {{\ttfamily}} & {{\ttfamily x}} & {{\ttfamily x}} & {{\ttfamily x}} \\ \hline
$\K_{53}$ & {{\ttfamily1,1,1,1,1,1,1,2,18}} & {{\ttfamily1,1,1,1,1,1,1,2,2,43}} & {{\ttfamily x}} & {{\ttfamily x}} & {{\ttfamily x}} \\ \hline
$\K_{54}$ & {{\ttfamily}} & {{\ttfamily}} & {{\ttfamily x}} & {{\ttfamily x}} & {{\ttfamily x}} \\ \hline
$\K_{55}$ & {{\ttfamily1,1,1,1,1,1,1,2,19}} & {{\ttfamily1,1,1,1,1,1,1,2,3,44}} & {{\ttfamily x}} & {{\ttfamily x}} & {{\ttfamily x}} \\ \hline
$\K_{56}$ & {{\ttfamily x}} & {{\ttfamily}} & {{\ttfamily x}} & {{\ttfamily x}} & {{\ttfamily x}} \\ \hline
$\K_{57}$ & {{\ttfamily x}} & {{\ttfamily1,1,1,1,1,1,2,3,47}} & {{\ttfamily x}} & {{\ttfamily x}} & {{\ttfamily x}} \\ \hline
$\K_{58}$ & {{\ttfamily}} & {{\ttfamily}} & {{\ttfamily x}} & {{\ttfamily x}} & {{\ttfamily x}} \\ \hline
$\K_{59}$ & {{\ttfamily1,2,2,3,22}} & {{\ttfamily1,1,1,2,2,4,49}} & {{\ttfamily x}} & {{\ttfamily x}} & {{\ttfamily x}} \\ \hline
$\K_{60}$ & {{\ttfamily}} & {{\ttfamily}} & {{\ttfamily x}} & {{\ttfamily x}} & {{\ttfamily x}} \\ \hline
$\K_{61}$ & {{\ttfamily1,1,1,1,1,1,1,1,2,21}} & {{\ttfamily1,1,1,1,1,1,1,1,2,3,49}} & {{\ttfamily x}} & {{\ttfamily x}} & {{\ttfamily x}} \\ \hline
$\K_{62}$ & {{\ttfamily}} & {{\ttfamily}} & {{\ttfamily x}} & {{\ttfamily x}} & {{\ttfamily x}} \\ \hline
$\K_{63}$ & {{\ttfamily1,1,1,1,1,1,1,1,1,2,21}} & {{\ttfamily1,1,1,1,1,1,1,1,2,2,2,50}} & {{\ttfamily x}} & {{\ttfamily x}} & {{\ttfamily x}} \\ \hline
$\K_{64}$ & {{\ttfamily}} & {{\ttfamily}} & {{\ttfamily x}} & {{\ttfamily x}} & {{\ttfamily x}} \\ \hline
$\K_{65}$ & {{\ttfamily1,1,1,1,1,1,1,1,2,2,21}} & {{\ttfamily1,1,1,1,1,1,1,2,2,2,53}} & {{\ttfamily x}} & {{\ttfamily x}} & {{\ttfamily x}} \\ \hline
$\K_{66}$ & {{\ttfamily}} & {{\ttfamily x}} & {{\ttfamily x}} & {{\ttfamily x}} & {{\ttfamily x}} \\ \hline
$\K_{67}$ & {{\ttfamily1,1,1,1,1,1,1,2,2,23}} & {{\ttfamily x}} & {{\ttfamily x}} & {{\ttfamily x}} & {{\ttfamily x}} \\ \hline
$\K_{68}$ & {{\ttfamily}} & {{\ttfamily x}} & {{\ttfamily x}} & {{\ttfamily x}} & {{\ttfamily x}} \\ \hline
$\K_{69}$ & {{\ttfamily1,2,2,3,3,24}} & {{\ttfamily x}} & {{\ttfamily x}} & {{\ttfamily x}} & {{\ttfamily x}} \\ \hline
$\K_{70}$ & {{\ttfamily}} & {{\ttfamily x}} & {{\ttfamily x}} & {{\ttfamily x}} & {{\ttfamily x}} \\ \hline
$\K_{71}$ & {{\ttfamily1,1,1,1,1,1,1,1,1,2,2,23}} & {{\ttfamily x}} & {{\ttfamily x}} & {{\ttfamily x}} & {{\ttfamily x}} \\ \hline
$\K_{72}$ & {{\ttfamily}} & {{\ttfamily x}} & {{\ttfamily x}} & {{\ttfamily x}} & {{\ttfamily x}} \\ \hline
$\K_{73}$ & {{\ttfamily1,1,1,1,1,1,1,1,1,2,2,24}} & {{\ttfamily x}} & {{\ttfamily x}} & {{\ttfamily x}} & {{\ttfamily x}} \\ \hline
$\K_{74}$ & {{\ttfamily}} & {{\ttfamily x}} & {{\ttfamily x}} & {{\ttfamily x}} & {{\ttfamily x}} \\ \hline
$\K_{75}$ & {{\ttfamily1,1,1,1,1,1,1,1,1,2,2,25}} & {{\ttfamily x}} & {{\ttfamily x}} & {{\ttfamily x}} & {{\ttfamily x}} \\ \hline
$\K_{76}$ & {{\ttfamily}} & {{\ttfamily x}} & {{\ttfamily x}} & {{\ttfamily x}} & {{\ttfamily x}} \\ \hline
$\K_{77}$ & {{\ttfamily1,1,1,1,1,1,1,1,1,2,2,26}} & {{\ttfamily x}} & {{\ttfamily x}} & {{\ttfamily x}} & {{\ttfamily x}} \\ \hline
$\K_{78}$ & {{\ttfamily}} & {{\ttfamily x}} & {{\ttfamily x}} & {{\ttfamily x}} & {{\ttfamily x}} \\ \hline
$\K_{79}$ & {{\ttfamily1,2,2,2,3,3,27}} & {{\ttfamily x}} & {{\ttfamily x}} & {{\ttfamily x}} & {{\ttfamily x}} \\ \hline
$\K_{80}$ & {{\ttfamily}} & {{\ttfamily x}} & {{\ttfamily x}} & {{\ttfamily x}} & {{\ttfamily x}} \\ \hline
$\K_{81}$ & {{\ttfamily1,1,1,1,1,1,1,1,1,1,1,2,2,26}} & {{\ttfamily x}} & {{\ttfamily x}} & {{\ttfamily x}} & {{\ttfamily x}} \\ \hline
$\K_{82}$ & {{\ttfamily}} & {{\ttfamily x}} & {{\ttfamily x}} & {{\ttfamily x}} & {{\ttfamily x}} \\ \hline
$\K_{83}$ & {{\ttfamily1,1,1,1,1,1,1,1,1,1,2,2,28}} & {{\ttfamily x}} & {{\ttfamily x}} & {{\ttfamily x}} & {{\ttfamily x}} \\ \hline
$\K_{84}$ & {{\ttfamily}} & {{\ttfamily x}} & {{\ttfamily x}} & {{\ttfamily x}} & {{\ttfamily x}} \\ \hline
$\K_{85}$ & {{\ttfamily1,1,1,1,1,1,1,1,1,1,2,2,29}} & {{\ttfamily x}} & {{\ttfamily x}} & {{\ttfamily x}} & {{\ttfamily x}} \\ \hline
$\K_{86}$ & {{\ttfamily}} & {{\ttfamily x}} & {{\ttfamily x}} & {{\ttfamily x}} & {{\ttfamily x}} \\ \hline
$\K_{87}$ & {{\ttfamily1,1,1,1,1,1,1,1,1,1,1,1,2,30}} & {{\ttfamily x}} & {{\ttfamily x}} & {{\ttfamily x}} & {{\ttfamily x}} \\ \hline
\end{tabular}}

%% file: 7.1.0.1-data.tex
\subsection{7.1.0.1: $\bZ/7^n$}
\scalebox{.6}{
\begin{tabular}{|r||r|r|r|r|r|r|} \hline
$\K_{{r}}\backslash n$ & $\bZ/7^2$ & $\bZ/7^3$ & $\bZ/7^4$ & $\bZ/7^5$ & $\bZ/7^6$ & $\bZ/7^7$ \\ \hline \hline
$\K_{1}$ & {{\ttfamily1}} & {{\ttfamily2}} & {{\ttfamily3}} & {{\ttfamily4}} & {{\ttfamily5}} & {{\ttfamily6}} \\ \hline
$\K_{2}$ & {{\ttfamily}} & {{\ttfamily}} & {{\ttfamily}} & {{\ttfamily}} & {{\ttfamily}} & {{\ttfamily}} \\ \hline
$\K_{3}$ & {{\ttfamily2}} & {{\ttfamily4}} & {{\ttfamily6}} & {{\ttfamily8}} & {{\ttfamily10}} & {{\ttfamily12}} \\ \hline
$\K_{4}$ & {{\ttfamily}} & {{\ttfamily}} & {{\ttfamily}} & {{\ttfamily}} & {{\ttfamily}} & {{\ttfamily}} \\ \hline
$\K_{5}$ & {{\ttfamily3}} & {{\ttfamily6}} & {{\ttfamily9}} & {{\ttfamily12}} & {{\ttfamily15}} & {{\ttfamily18}} \\ \hline
$\K_{6}$ & {{\ttfamily}} & {{\ttfamily}} & {{\ttfamily}} & {{\ttfamily}} & {{\ttfamily}} & {{\ttfamily}} \\ \hline
$\K_{7}$ & {{\ttfamily4}} & {{\ttfamily8}} & {{\ttfamily12}} & {{\ttfamily16}} & {{\ttfamily20}} & {{\ttfamily24}} \\ \hline
$\K_{8}$ & {{\ttfamily}} & {{\ttfamily}} & {{\ttfamily}} & {{\ttfamily}} & {{\ttfamily}} & {{\ttfamily}} \\ \hline
$\K_{9}$ & {{\ttfamily5}} & {{\ttfamily10}} & {{\ttfamily15}} & {{\ttfamily20}} & {{\ttfamily25}} & {{\ttfamily30}} \\ \hline
$\K_{10}$ & {{\ttfamily}} & {{\ttfamily}} & {{\ttfamily}} & {{\ttfamily}} & {{\ttfamily}} & {{\ttfamily}} \\ \hline
$\K_{11}$ & {{\ttfamily1,5}} & {{\ttfamily1,11}} & {{\ttfamily1,17}} & {{\ttfamily1,23}} & {{\ttfamily1,29}} & {{\ttfamily1,35}} \\ \hline
$\K_{12}$ & {{\ttfamily1}} & {{\ttfamily1}} & {{\ttfamily1}} & {{\ttfamily1}} & {{\ttfamily1}} & {{\ttfamily1}} \\ \hline
$\K_{13}$ & {{\ttfamily8}} & {{\ttfamily15}} & {{\ttfamily22}} & {{\ttfamily29}} & {{\ttfamily36}} & {{\ttfamily43}} \\ \hline
$\K_{14}$ & {{\ttfamily}} & {{\ttfamily}} & {{\ttfamily}} & {{\ttfamily}} & {{\ttfamily}} & {{\ttfamily}} \\ \hline
$\K_{15}$ & {{\ttfamily1,7}} & {{\ttfamily2,14}} & {{\ttfamily2,22}} & {{\ttfamily2,30}} & {{\ttfamily2,38}} & {{\ttfamily2,46}} \\ \hline
$\K_{16}$ & {{\ttfamily}} & {{\ttfamily}} & {{\ttfamily}} & {{\ttfamily}} & {{\ttfamily}} & {{\ttfamily x}} \\ \hline
$\K_{17}$ & {{\ttfamily1,1,7}} & {{\ttfamily1,1,16}} & {{\ttfamily1,1,25}} & {{\ttfamily1,1,34}} & {{\ttfamily1,1,43}} & {{\ttfamily x}} \\ \hline
$\K_{18}$ & {{\ttfamily}} & {{\ttfamily}} & {{\ttfamily}} & {{\ttfamily}} & {{\ttfamily x}} & {{\ttfamily x}} \\ \hline
$\K_{19}$ & {{\ttfamily1,1,8}} & {{\ttfamily1,1,18}} & {{\ttfamily1,1,28}} & {{\ttfamily1,1,38}} & {{\ttfamily x}} & {{\ttfamily x}} \\ \hline
$\K_{20}$ & {{\ttfamily}} & {{\ttfamily}} & {{\ttfamily}} & {{\ttfamily}} & {{\ttfamily x}} & {{\ttfamily x}} \\ \hline
$\K_{21}$ & {{\ttfamily1,1,9}} & {{\ttfamily1,1,20}} & {{\ttfamily1,1,31}} & {{\ttfamily1,1,42}} & {{\ttfamily x}} & {{\ttfamily x}} \\ \hline
$\K_{22}$ & {{\ttfamily}} & {{\ttfamily}} & {{\ttfamily}} & {{\ttfamily x}} & {{\ttfamily x}} & {{\ttfamily x}} \\ \hline
$\K_{23}$ & {{\ttfamily1,1,10}} & {{\ttfamily1,1,1,21}} & {{\ttfamily1,1,1,33}} & {{\ttfamily x}} & {{\ttfamily x}} & {{\ttfamily x}} \\ \hline
$\K_{24}$ & {{\ttfamily}} & {{\ttfamily1}} & {{\ttfamily1}} & {{\ttfamily x}} & {{\ttfamily x}} & {{\ttfamily x}} \\ \hline
$\K_{25}$ & {{\ttfamily1,12}} & {{\ttfamily1,1,25}} & {{\ttfamily1,1,38}} & {{\ttfamily x}} & {{\ttfamily x}} & {{\ttfamily x}} \\ \hline
$\K_{26}$ & {{\ttfamily}} & {{\ttfamily}} & {{\ttfamily}} & {{\ttfamily x}} & {{\ttfamily x}} & {{\ttfamily x}} \\ \hline
$\K_{27}$ & {{\ttfamily1,13}} & {{\ttfamily1,1,26}} & {{\ttfamily1,1,40}} & {{\ttfamily x}} & {{\ttfamily x}} & {{\ttfamily x}} \\ \hline
$\K_{28}$ & {{\ttfamily}} & {{\ttfamily}} & {{\ttfamily x}} & {{\ttfamily x}} & {{\ttfamily x}} & {{\ttfamily x}} \\ \hline
$\K_{29}$ & {{\ttfamily1,1,1,12}} & {{\ttfamily1,1,2,26}} & {{\ttfamily x}} & {{\ttfamily x}} & {{\ttfamily x}} & {{\ttfamily x}} \\ \hline
$\K_{30}$ & {{\ttfamily}} & {{\ttfamily}} & {{\ttfamily x}} & {{\ttfamily x}} & {{\ttfamily x}} & {{\ttfamily x}} \\ \hline
$\K_{31}$ & {{\ttfamily1,1,1,13}} & {{\ttfamily1,1,2,28}} & {{\ttfamily x}} & {{\ttfamily x}} & {{\ttfamily x}} & {{\ttfamily x}} \\ \hline
$\K_{32}$ & {{\ttfamily}} & {{\ttfamily}} & {{\ttfamily x}} & {{\ttfamily x}} & {{\ttfamily x}} & {{\ttfamily x}} \\ \hline
$\K_{33}$ & {{\ttfamily1,1,1,1,13}} & {{\ttfamily1,1,1,1,30}} & {{\ttfamily x}} & {{\ttfamily x}} & {{\ttfamily x}} & {{\ttfamily x}} \\ \hline
$\K_{34}$ & {{\ttfamily}} & {{\ttfamily}} & {{\ttfamily x}} & {{\ttfamily x}} & {{\ttfamily x}} & {{\ttfamily x}} \\ \hline
$\K_{35}$ & {{\ttfamily1,1,1,1,14}} & {{\ttfamily1,1,1,1,1,31}} & {{\ttfamily x}} & {{\ttfamily x}} & {{\ttfamily x}} & {{\ttfamily x}} \\ \hline
$\K_{36}$ & {{\ttfamily}} & {{\ttfamily x}} & {{\ttfamily x}} & {{\ttfamily x}} & {{\ttfamily x}} & {{\ttfamily x}} \\ \hline
$\K_{37}$ & {{\ttfamily1,1,1,16}} & {{\ttfamily x}} & {{\ttfamily x}} & {{\ttfamily x}} & {{\ttfamily x}} & {{\ttfamily x}} \\ \hline
$\K_{38}$ & {{\ttfamily}} & {{\ttfamily x}} & {{\ttfamily x}} & {{\ttfamily x}} & {{\ttfamily x}} & {{\ttfamily x}} \\ \hline
$\K_{39}$ & {{\ttfamily1,1,1,17}} & {{\ttfamily x}} & {{\ttfamily x}} & {{\ttfamily x}} & {{\ttfamily x}} & {{\ttfamily x}} \\ \hline
$\K_{40}$ & {{\ttfamily}} & {{\ttfamily x}} & {{\ttfamily x}} & {{\ttfamily x}} & {{\ttfamily x}} & {{\ttfamily x}} \\ \hline
$\K_{41}$ & {{\ttfamily1,2,18}} & {{\ttfamily x}} & {{\ttfamily x}} & {{\ttfamily x}} & {{\ttfamily x}} & {{\ttfamily x}} \\ \hline
$\K_{42}$ & {{\ttfamily}} & {{\ttfamily x}} & {{\ttfamily x}} & {{\ttfamily x}} & {{\ttfamily x}} & {{\ttfamily x}} \\ \hline
$\K_{43}$ & {{\ttfamily1,1,1,1,1,17}} & {{\ttfamily x}} & {{\ttfamily x}} & {{\ttfamily x}} & {{\ttfamily x}} & {{\ttfamily x}} \\ \hline
$\K_{44}$ & {{\ttfamily}} & {{\ttfamily x}} & {{\ttfamily x}} & {{\ttfamily x}} & {{\ttfamily x}} & {{\ttfamily x}} \\ \hline
$\K_{45}$ & {{\ttfamily1,1,1,1,1,18}} & {{\ttfamily x}} & {{\ttfamily x}} & {{\ttfamily x}} & {{\ttfamily x}} & {{\ttfamily x}} \\ \hline
$\K_{46}$ & {{\ttfamily}} & {{\ttfamily x}} & {{\ttfamily x}} & {{\ttfamily x}} & {{\ttfamily x}} & {{\ttfamily x}} \\ \hline
$\K_{47}$ & {{\ttfamily1,1,1,1,1,19}} & {{\ttfamily x}} & {{\ttfamily x}} & {{\ttfamily x}} & {{\ttfamily x}} & {{\ttfamily x}} \\ \hline
$\K_{48}$ & {{\ttfamily}} & {{\ttfamily x}} & {{\ttfamily x}} & {{\ttfamily x}} & {{\ttfamily x}} & {{\ttfamily x}} \\ \hline
$\K_{49}$ & {{\ttfamily1,1,1,1,1,20}} & {{\ttfamily x}} & {{\ttfamily x}} & {{\ttfamily x}} & {{\ttfamily x}} & {{\ttfamily x}} \\ \hline
$\K_{50}$ & {{\ttfamily}} & {{\ttfamily x}} & {{\ttfamily x}} & {{\ttfamily x}} & {{\ttfamily x}} & {{\ttfamily x}} \\ \hline
$\K_{51}$ & {{\ttfamily1,1,1,1,1,21}} & {{\ttfamily x}} & {{\ttfamily x}} & {{\ttfamily x}} & {{\ttfamily x}} & {{\ttfamily x}} \\ \hline
\end{tabular}}